\documentclass[11pt]{article}
\usepackage{latexsym,amsmath,amscd,amssymb, graphics}
\usepackage{amsmath,amssymb,mathrsfs,framed,esint,slashed,color}
\usepackage{amsthm}
\usepackage{enumerate}
\usepackage[colorlinks]{hyperref}
\usepackage{enumitem}
\usepackage[margin=1in]{geometry}
\usepackage{stmaryrd}
\usepackage{graphicx}
\usepackage[all]{xy}
\usepackage[makeroom]{cancel}

\usepackage{framed}

\usepackage{pgfplots, pgfplotstable, booktabs, colortbl, array,multirow}
\usetikzlibrary{calc}
\usepgfplotslibrary{groupplots}
\pgfplotsset{compat=newest}

\theoremstyle{plain}
\newtheorem{theorem}{Theorem}[section]

\newtheorem{proposition}[theorem]{Proposition}

\theoremstyle{definition}

\newtheorem{remark}{Remark}[section]

\newcommand{\textoverline}[1]{$\overline{\mbox{#1}}$}

\DeclareMathOperator{\dv}{div}
\DeclareMathOperator{\curl}{curl}

\newcommand{\vv}{v}
\usepackage{verbatim}

\begin{document}

\title{Variational and thermodynamically consistent finite element discretization for heat conducting viscous fluids}

\author{Evan S. Gawlik\thanks{\noindent Department of Mathematics,  University of Hawai`i at M\textoverline{a}noa, \href{egawlik@hawaii.edu}{egawlik@hawaii.edu}} \; and \; Fran\c{c}ois Gay-Balmaz\thanks{\noindent Division of Mathematical Sciences, Nanyang Technological University, Singapore 637371, \href{francois.gb@ntu.edu.sg}{francois.gb@ntu.edu.sg}}}

\date{}

\maketitle

\begin{abstract}
Respecting the laws of thermodynamics is crucial for ensuring that numerical simulations of dynamical systems deliver physically relevant results.
In this paper, we construct a structure-preserving and thermodynamically consistent finite element method and time-stepping scheme for heat conducting viscous fluids, with general state equations. The method is deduced by discretizing a variational formulation for nonequilibrium thermodynamics that extends Hamilton's principle for fluids to systems with irreversible processes. The resulting scheme preserves the balance of energy and mass to machine precision, as well as the second law of thermodynamics, both at the spatially and temporally discrete levels.
The method is shown to apply both with insulated and prescribed heat flux boundary conditions, as well as with prescribed temperature boundary conditions. We illustrate the properties of the scheme with the Rayleigh-B\'enard thermal convection.
While the focus is on heat conducting viscous fluids, the proposed discrete variational framework paves the way to a systematic construction of thermodynamically consistent discretizations of continuum systems.
\end{abstract}



\section{Introduction}
 
Structure preserving discretization of continuum systems, such as fluids and elastic bodies, is today widely recognized as an essential tool for the construction of numerical schemes when the long time accuracy and the respect of the balance and conservation laws of the simulated system are crucial. Such properties are especially relevant in the context of geophysical fluid dynamics for weather and climate prediction, or in the context of plasma physics.

A well-known constructive approach to deriving such structure preserving discretizations is to exploit the variational formulation underlying the equations of motion. This formulation is deduced from Hamilton's critical action principle and provides a useful setting for both the temporal and spatial discretization steps. While variational time integrators for finite dimensional systems are today well-established (see \cite{MaWe2001} and the large series of subsequent works), spatial and spacetime discrete variational approaches for continuum systems are still undergoing foundational developments (e.g. \cite{MaPaSh1998,LeMaOrWe2003,PaMuToKaMaDe2010,DeGBRa2016,DeGB2022,GaGB2020}).

Despite their wide range of applicability, a main limitation of variational formulations issued from Hamilton's principle is their inability to consistently include irreversible processes in the systems. In most of the real world applications of continuum mechanics, however, such processes do have a deep impact on the dynamics. For instance, in the case of fluid dynamics, the processes of heat conduction, diffusion, viscosity, or chemical reactions, play a major role in geophysical, astrophysical, engineering and technological applications. Thermal convection occurring in the planets' oceans, atmospheres and mantles, as well as in stars, is a typical phenomenon occurring in conjunction with the process of heat transfer among others. Importantly, due to their irreversible character, such phenomena fit into the realm of nonequilibrium thermodynamics, governed by the two laws imposing constraints on the energy and entropy behavior. In order to get reliable and physically meaningful numerical solutions for such systems, it is of paramount importance to preserve these laws at the discrete level, thereby highlighting the need to extend variational discretization from reversible continuum mechanics to nonequilibrium thermodynamics.

We propose in this paper a first step in this direction for the case of fluid dynamics with heat conduction and viscosity. The general form of the equations of evolution for such fluids on a bounded domain $ \Omega \subset \mathbb{R} ^d$, $d=2,3$, with smooth boundary is
\begin{equation}\label{Equations_Intro} 
\left\{
\begin{array}{l}
\vspace{0.2cm}\displaystyle \rho  ( \partial _t u+ u \cdot \nabla u)= - \rho  \nabla \phi - \nabla p + \operatorname{div} \sigma\\
 \vspace{0.2cm}\displaystyle T (\partial_t s + \dv(s u)  + \operatorname{div} j _s) =  \sigma \!: \!\nabla u - j _s\! \cdot  \!\nabla T\\
\displaystyle \partial_t \rho + \dv(\rho u)=0,
\end{array}
\right.
\end{equation}
with $u$ the fluid velocity, $ \rho  $ the mass density, $s$ the entropy density, $T$ the temperature, and $ p$ the pressure. The equations also depend on the gravitational potential $ \phi  $, while the irreversible processes are described by the viscous stress tensor $ \sigma $ and the entropy flux $j_s$. In this paper we shall take for them the usual Navier-Stokes and Fourier expressions. The equations are supplemented by the no-slip boundary condition for the velocity $u|_{ \partial \Omega } =0$ and by one of the following thermal boundary conditions:
\begin{equation}\label{BC_Intro} 
T j_s \cdot n|_{ \partial \Omega }=q_0 \qquad\text{or}\qquad T|_{ \partial \Omega }=T_0.
\end{equation} 
These correspond to either prescribed heat flux or prescribed temperature boundary conditions, the case $q_0=0$ being that of an insulated boundary. The system is closed by a given state equation, which we keep general in this paper, not necessarily given by the perfect gas.

Our approach is based on a variational formulation for nonequilibrium thermodynamics that extends Hamilton's principle to include irreversibility, developed in \cite{GBYo2017a,GBYo2017b,GBYo2018b}. Importantly for the present work, this approach extends to the irreversible setting the well-known variational and geometric formulation of hydrodynamics on diffeomorphism groups initiated in \cite{Ar1966}. This variational formulation gives \eqref{Equations_Intro} together with an appropriate boundary condition from \eqref{BC_Intro} as critical point conditions.  It should be noted that in this paper, we use the term ``variational formulation'' to refer to a critical action principle that characterizes solution curves, not a variational (or weak) formulation in the finite element sense.  One can use the former to construct the latter, but not vice versa in general.

The resulting class of schemes, which is written down in standard finite element notation in~\eqref{fully_discrete_NSF}, satisfies the two laws of thermodynamics at the fully discrete level. More precisely, the total energy is shown to be exactly preserved at the fully discrete level when the fluid is adiabatically closed, while a discrete energy balance holds in the presence of external heating. 
Regarding the second law, the entropy generated by the internal irreversible processes is shown to grow at each time step on each cell. 
 
It turns out that the variational formulation for nonequilibrium thermodynamics yields the equations for heat conducting viscous fluids in a weak form that is quite suitable to achieve thermodynamic consistency and, at the same time, can naturally accommodate both prescribed heat flux (Neumann) or prescribed temperature (Dirichlet) boundary conditions. The variational formulation also naturally involves an internal entropy variable, helping identify the internal entropy production at the fully discrete level, which is well-known to differ from the rate of entropy change in the presence of the entropy flux.

To our knowledge, our scheme is the first scheme to respect both the balance of total energy and the second law of thermodynamics locally (i.e. elementwise) among finite element/finite volume approaches for heat conducting viscous fluids.  One existing scheme that comes close to achieving these goals is the finite volume scheme studied in \cite{BaLuMiShYu2023} and the references therein.  It satisfies a global energy balance law and a global entropy production law, but the laws contain artificial sources of energy dissipation and entropy production.  It also does not appear to guarantee entropy production locally.  Another approach that achieves global, but not local, entropy production is described in \cite{HuFrMa1986,ShHuJo1991}.  It relies on a change of variables that couples the energy density with other variables, and for this reason it does not appear to respect the balance of total energy.  Related techniques are used in \cite{TaZh2006} to achieve global entropy production in the one-dimensional setting.  In the absence of irreversible processes, \cite{LeRiDu2023} has recently constructed a scheme that conserves both total entropy and total energy. This was also accomplished in Section 6 of \cite{GaGB2020}.

Our paper is structured as follows. In \S\ref{sec_2} we recall the variational formulation for heat conducting viscous fluids in the insulated case following \cite{GBYo2017b} and then present the modifications needed for the treatment of a prescribed temperature or prescribed heat flux on the boundary, consistently with the variational formulation of open systems given in \cite{GBYo2018b}. We also give the weak formulation of the equations resulting from the variational framework. In \S\ref{Sec_discrete}, following previous works on the reversible case, we carry out this variational formulation on a discrete version of the diffeomorphism group, based on a discontinuous Galerkin discretization of functions. A suitable spatially discrete version of heat conducting viscous flow is obtained for any given discrete fluid Lagrangian.
A main subsequent step is the appropriate discretization of the thermodynamic fluxes, which is chosen in accordance with the considered Neumann or Dirichlet boundary conditions, and realized for the Navier-Stokes-Fourier case. It is then shown that the resulting spatial discretization satisfies the two laws of thermodynamics exactly. Thanks to its structure preserving form, the resulting finite element scheme can be followed by an energy preserving time discretization which allows satisfaction of the second law at each step on each fluid cell.
Rayleigh-B\'enard convection tests are carried out in \S\ref{sec_4} for both prescribed temperature (Dirichlet) and prescribed heat flux (Neumann) boundary conditions, for several values of Rayleigh numbers, illustrating the predictive value and thermodynamic consistency of our scheme. We also test the numerical convergence of our scheme with respect to spatial and temporal refinements, and provide a comparison with a recently derived finite volume method for the Navier-Stokes-Fourier equations.

\section{Variational formulation for heat conducting viscous fluids}\label{sec_2}

We review here the variational formulation of nonequilibrium thermodynamics underlying the structure preserving discretization method that we present for heat conducting viscous fluids.
This variational formulation, developed in \cite{GBYo2017a,GBYo2017b,GBYo2018b}, is an extension of the Hamilton principle  which allows one to systematically include irreversible processes  in the dynamics, such as friction, viscosity, heat conduction, matter transfer, or chemical reactions. Importantly for the present work, the variational formulation applies to adiabatically closed systems as well as systems exchanging heat and matter with their surroundings, see \cite{GBYo2018b}.

We recall in \S\ref{subsec_insulated} this formulation for heat conducting viscous fluids with insulated boundaries (homogeneous Neumann boundary conditions), and then present the modifications needed for the treatment of a prescribed temperature on the boundary (Dirichlet boundary conditions) or prescribed heat flux at the boundary (nonhomogeneous Neumann boundary conditions), see \S\ref{Dirichlet_T}. A weak formulation of the equation is deduced from the variational formulation in \S\ref{weak_form_continuous} in a unified way for all boundary conditions. 

In order to help identify the role and meaning of each variable in a simpler context, both for the adiabatically closed case and for the open case, we present in \ref{A} an application of the variational formulation to an elementary finite dimensional example.

\subsection{Insulated boundaries}\label{subsec_insulated}

The variational formulation is best expressed in the material (or Lagrangian) description since it is in this description that it is an extension of the Hamilton principle of continuum mechanics and takes its simpler form. The variational formulation in the spatial (or Eulerian) description that underlies our approach is then deduced by using the fluid relabelling symmetry.
A finite element spatial discretization of this ``Lagrangian-to-Eulerian" variational approach will be developed in \S\ref{Sec_discrete}.

\paragraph{Lagrangian description.} Let $ \Omega \subset \mathbb{R} ^d$, $d=2,3$, be a bounded domain with smooth boundary, $ \operatorname{Diff}( \Omega )$ the group of diffeomorphisms of $ \Omega $, and $ \operatorname{Diff}_0( \Omega)$ the subgroup of diffeomorphisms keeping $\partial  \Omega $  pointwise fixed. We denote by $ \mathcal{F} ( \Omega )$ and $ \mathcal{F} ( \Omega ) ^* $ the spaces of functions and densities on $ \Omega $ with sufficient regularity. We identify densities with functions, bearing in mind that the action of $\operatorname{Diff}( \Omega )$ on $ \mathcal{F} ( \Omega ) ^* $ under this identification differs from the action of $\operatorname{Diff}( \Omega )$ on $ \mathcal{F} ( \Omega )$; see~(\ref{Lagr_to_Eul_1}).  
In the Lagrangian description, the motion of a compressible fluid in the domain $ \Omega$ is given by two time dependent maps, the fluid configuration map $\varphi :[t_0,t_1] \rightarrow \operatorname{Diff}_0( \Omega )$, giving the position $x= \varphi (t,X)$ at time $t$ of a particle located at $X$ at $t=t_0$, and the entropy density $ S:[t_0,t_1]  \rightarrow \mathcal{F} ( \Omega )^*$. From mass conservation, the mass density $ \varrho _0 \in \mathcal{F} ( \Omega )^*$ is constant in time in the Lagrangian description, $ \varrho (t,X)= \varrho _0(X)$.

\medskip

In the absence of irreversible processes, the entropy density is also constant in time $S(t,X)=S_0(X)$ and the equations of motion follow from the Hamilton principle
\begin{equation}\label{HP_fluid} 
\delta \int_{t_0}^{t_1}L( \varphi , \dot \varphi , S_0, \varrho _0) {\rm d}t=0
\end{equation}
for variations $ \delta \varphi $ with $ \delta \varphi |_{t=t_0,t_1}=0$. In \eqref{HP_fluid} the function $L:T \operatorname{Diff}_0( \Omega ) \times \mathcal{F} ( \Omega )^* \times \mathcal{F} ( \Omega )^*  \rightarrow \mathbb{R}$ is the Lagrangian of the compressible fluid model, the standard expression being
\begin{equation}\label{material_L} 
L( \varphi , \dot \varphi , S, \varrho ) =\int_ \Omega \Big[\frac{1}{2} \varrho  | \dot   \varphi | ^2 - \epsilon   \big( \varrho / J \varphi , S/ J \varphi \big)- \varrho \phi ( \varphi )\Big] {\rm d} X
\end{equation}
with $J \varphi $ the Jacobian of $\varphi $, $ \epsilon  $ the internal energy density, and $ \phi (x)$ the gravitational potential.

\medskip 

For the heat conducting viscous fluid one also needs to specify the phenomenological expressions of the viscous stress tensor and entropy flux, denoted $ P $ and $J_S$ in the Lagrangian description. The extension of Hamilton's principle \eqref{HP_fluid} to heat conducting viscous fluids given in \cite{GBYo2017b} involves two additional variables besides $ \varphi (t) \in \operatorname{Diff}_0( \Omega ) $ and $S(t) \in \mathcal{F} ( \Omega ) ^*$: the internal entropy density variable $ \Sigma (t) \in \mathcal{F} ( \Omega )^*$, whose time rate of change is the internal entropy production, and the thermal displacement $ \Gamma (t) \in \mathcal{F} ( \Omega )$, whose time rate of change is the temperature.
The variational principle reads as follows.

Find the curves $\varphi: [t_0,t_1] \rightarrow  \operatorname{Diff}_0( \Omega )$, $S, \Sigma :[t_0,t_1]\rightarrow  \mathcal{F} ( \Omega)^*$, and $\Gamma: [t_0,t_1]\rightarrow \mathcal{F} ( \Omega )$, which are critical for the \textit{variational condition}
\begin{equation}\label{VP_fluid} 
\delta \int_{t_0}^{t_1} \Big[L\big(\varphi, \dot \varphi,  S, \varrho _0\big) +  \int_ \Omega(S- \Sigma ) \dot \Gamma   \,{\rm d} X  \Big]{\rm d}t =0
\end{equation}
subject to the \textit{phenomenological constraint}
\begin{equation}\label{KC_fluid}
\frac{\delta  L }{\delta  S}\dot \Sigma= \underbrace{- P: \nabla\dot \varphi }_{\text{viscosity}} +  \underbrace{J_S \cdot \nabla \dot\Gamma  }_{\text{heat conduction}}
\end{equation}  
and for variations $ \delta \varphi $, $ \delta \Sigma $, $ \delta \Gamma $ subject to the \textit{variational constraint}
\begin{equation}\label{VC_fluid}
\frac{\delta   L}{\delta  S}\delta \Sigma = \underbrace{- P: \nabla \delta  \varphi }_{\text{viscosity}}  + \underbrace{J_S \cdot \nabla \delta \Gamma  }_{\text{heat conduction}}
\end{equation} 
with $\delta \varphi|_{t=t_0,t_1}= \delta\Gamma|_{t=t_0,t_1}=0$, and $\delta\varphi|_{\partial\Omega}=0$, while the variations $ \delta S$ are free. The critical condition of this variational formulation gives the heat conducting viscous fluid equation in the Lagrangian description, see \S\ref{B_0}.
In \eqref{KC_fluid} and \eqref{VC_fluid}  $\frac{\delta  L }{\delta  S}$ is the functional derivative of $L$ with respect to $S$, defined as $ \left. \frac{d}{d\varepsilon}\right|_{\varepsilon=0} L( \varphi , \dot \varphi , \varrho , S+ \varepsilon \delta S)= \int_ \Omega \frac{\delta L}{\delta S} \delta S \,{\rm d} X$, for all $ \delta S$. It is identified with minus the temperature of the fluid, denoted $ \mathfrak{T} = - \frac{\delta L}{\delta S}$ in the Lagrangian description.

\begin{remark}[Structure of the variational formulation] The variational formulation \eqref{VP_fluid}--\eqref{VC_fluid} is an extension of the Hamilton principle \eqref{HP_fluid} for fluids which includes two types of constraints: a kinematic (phenomenological) constraint \eqref{KC_fluid} on the critical curve and a variational constraint \eqref{VC_fluid} on the variations to be considered when computing this critical curve. The two constraints are related in a systematic way which formally involves replacing the time rate of changes (here $\dot \Sigma $, $\dot \varphi $, and $ \dot  \Gamma $) by the $ \delta $-variations (here $\delta  \Sigma $, $\delta  \varphi $, and $ \delta   \Gamma $). More precisely, on the right hand side of \eqref{KC_fluid} the two terms correspond to the dissipated power density associated to the processes of viscosity and heat conduction, with their virtual version appearing in \eqref{VC_fluid}. In coordinates $P: \nabla \dot \varphi = P^A_a \partial _A\dot \varphi ^a $ and $J_S \cdot \nabla \dot \Gamma = (J_S)^A \partial _A\dot \Gamma  $. This setting is common to the variational formulation of adiabatically closed thermodynamic systems, see \cite{GBYo2017a,GBYo2017b}, and is a nonlinear version of the \textit{Lagrange-d'Alembert principle} used in nonholonomic mechanics. In \ref{A} we recall an application of this type of variational formulation to an elementary finite dimensional thermodynamic system for which the computation of the critical condition is straightforward and which helps explain the role and the meaning of the variables $ \Sigma $ and $ \Gamma $.
\end{remark}

\paragraph{Eulerian description.} We recall here how  \eqref{VP_fluid}--\eqref{VC_fluid} can be converted to the Eulerian frame, thereby yielding the variational formulation underlying our structure preserving finite element discretization. The Eulerian versions of the variables  $\dot \varphi , \varrho _0, S, \Sigma , \Gamma $ are the Eulerian velocity $u$, mass density $ \rho  $, entropy density $s$, internal entropy density $ \varsigma $, and thermal displacement $ \gamma $ given as
\begin{equation}\label{Lagr_to_Eul_1} 
\begin{aligned} 
u&= \dot \varphi \circ \varphi ^{-1} \in \mathfrak{X} _0( \Omega ),\\
\rho & = (\varrho _0  \circ \varphi ^{-1} )J \varphi^{-1} \in \mathcal{F} ( \Omega )^* , \qquad s= (S \circ \varphi ^{-1} )J \varphi^{-1}  \in \mathcal{F} ( \Omega )^*,\\
\varsigma&=( \Sigma \circ \varphi ^{-1} )J \varphi ^{-1} \in \mathcal{F} ( \Omega ) ^*, \qquad \gamma = \Gamma  \circ \varphi ^{-1} \in \mathcal{F} ( \Omega ),
\end{aligned} 
\end{equation} 
where $\mathfrak{X} _0( \Omega )= \{ u \in \mathfrak{X} ( \Omega )\mid u|_ { \partial \Omega }=0\}$ is the space of vector fields on $ \Omega $ vanishing on the boundary.
The Eulerian viscous stress tensor $ \sigma $ and entropy flux $j_s$ are related to their Lagrangian counterparts $P$ and $J_S$ via the Piola transformations
\begin{equation}\label{Lagr_to_Eul_2} 
\sigma =  ( ( P \cdot \nabla \varphi ^\mathsf{T}  ) \circ \varphi ^{-1} )J \varphi ^{-1}  \qquad\text{and}\qquad j_s= ((\nabla \varphi \cdot J_S) \circ \varphi ^{-1} ) J \varphi ^{-1},
\end{equation}
see \cite{MaHu1994}.
From its relabelling symmetries, the Lagrangian $L$ can be expressed in terms of the Eulerian variables as $L(\varphi,\dot{\varphi},S,\varrho_0)=\ell(u,\rho,s)$. For the standard expression given in \eqref{material_L} one gets
\begin{equation}\label{standard_ell} 
\ell(u, \rho  , s) = \int_ \Omega \Big[\frac{1}{2} \rho  | u| ^2 - \epsilon ( \rho  , s) - \rho  \phi \Big]{\rm d} x.
\end{equation}

With relations \eqref{Lagr_to_Eul_1} and \eqref{Lagr_to_Eul_2}, the Eulerian version of the principle \eqref{VP_fluid}--\eqref{VC_fluid} reads as follows. Find the curves $u: [t_0,t_1] \rightarrow  \mathfrak{X} _0( \Omega )$, $s, \varsigma  :[t_0,t_1] \rightarrow  \mathcal{F} ( \Omega)^*$, and $ \gamma : [t_0,t_1] \rightarrow \mathcal{F} ( \Omega )$, which are critical for the \textit{variational condition}
\begin{equation}\label{VP_NSF_spatial}
\delta \int_{t_0}^{t_1}\Big[ \ell(u , \rho , s)+\int_ \Omega (s- \varsigma )D_t \gamma \, {\rm d} x \Big] {\rm d}t=0,
\end{equation}
with the \textit{phenomenological constraint} and \textit{variational constraint} given by
\begin{equation}\label{KC_NSF_spatial}
\frac{\delta   \ell}{\delta   s} \bar D_t \varsigma   = \underbrace{  -  \sigma : \nabla u  }_{\text{viscosity}}+ \underbrace{ j_s \cdot \nabla (D_t \gamma) }_{\text{heat conduction}},
\end{equation} 
\begin{equation}\label{VC_NSF_spatial}
\frac{\delta   \ell}{ \delta  s}  \bar D_\delta  \varsigma   =  \underbrace{ -\sigma : \nabla  v }_{\text{viscosity}}+ \underbrace{ j_s \cdot \nabla (D_ \delta \gamma)}_{\text{heat conduction}},
\end{equation}
and the Euler-Poincar\'e constraints
\begin{equation}\label{EP_constraints} 
\delta u = \partial _t v+[ u , v ], \qquad \delta \rho =- \operatorname{div}( \rho v),
\end{equation} 
where $[u,v]$ denotes the Lie bracket of the vector fields $u$ and $v$; see \cite{GBYo2017b} for details.
Here $ v= \delta \varphi \circ \varphi ^{-1} :[t_0,t_1] \rightarrow  \mathfrak{X} _0(\Omega )$ and $ \delta  \gamma:[t_0,t_1] \rightarrow  \mathcal{F} ( \Omega )$ are arbitrary curves with $ v |_{t=t_0,t_1}=0$ and $ \delta \gamma |_{t=t_0,t_1}=0$. Also, we have $ \delta s, \delta \varsigma:[t_0,t_1] \rightarrow  \mathcal{F} ( \Omega )^*$, $ \delta u:[t_0,t_1] \rightarrow \mathfrak{X} ( \Omega )$ and $ \delta \rho  :[t_0,t_1] \rightarrow \mathcal{F} ( \Omega ) ^* $.
In \eqref{VP_NSF_spatial}--\eqref{VC_NSF_spatial} we have used the following notations for the Lagrangian derivatives and variations of functions $f$ and densities $g$:
\begin{equation}\label{Lagr_derivatives} 
\begin{aligned} 
D_tf &= \partial _t f + u \cdot \nabla f &  & \qquad D_\delta  f = \delta f  + v \cdot \nabla f\\
\bar D_tg &= \partial _t g+ \operatorname{div}(gu) & &\qquad  \bar D_\delta  g= \delta g  +  \operatorname{div} ( g v ),
\end{aligned}
\end{equation} 
where $ v = \delta \varphi \circ \varphi ^{-1} $ and $u= \dot  \varphi \circ \varphi ^{-1} $.

By applying the variational principle \eqref{VP_NSF_spatial}--\eqref{EP_constraints}, we get the equations for a compressible heat conducting viscous fluid with Lagrangian $\ell(u, \rho  , s)$
\begin{equation}\label{Equations_Eulerian} 
\left\{
\begin{array}{l}
\vspace{0.2cm}\displaystyle
( \partial _t + \pounds _ u ) \frac{\delta  \ell}{\delta u }= \rho  \nabla \frac{\delta \ell}{\delta \rho }+ s \nabla \frac{\delta \ell }{\delta  s }+ \operatorname{div}  \sigma \\
\vspace{0.2cm}\displaystyle  - \frac{\delta \ell }{\delta s } (\bar D_ts  + \operatorname{div} j _s) =  \sigma \!: \!\nabla u + j _s\! \cdot  \!\nabla \frac{\delta \ell}{\delta s }\\
\displaystyle \bar D _t \rho =0,
\end{array}
\right.
\end{equation}
with $\pounds _u$ the Lie derivative of one-form densities, together with the conditions
\begin{equation}\label{additional_conditions} 
D_t \gamma = - \frac{\delta  \ell}{\delta  s}, \qquad \bar D_t \varsigma = \bar D_t s + \operatorname{div} j_s, \qquad\text{and}\qquad j_s \cdot n =0 \;\;\text{on}\;\; \partial \Omega,
\end{equation} 
see \S\ref{B_1} for details on the derivation. We have used the functional derivatives of $\ell$, defined as $ \left. \frac{d}{d\varepsilon}\right|_{\varepsilon=0} \ell(u+ \varepsilon \delta u, \rho  , s)= \int_ \Omega \frac{\delta \ell}{\delta u} \cdot \delta u  \,{\rm d} x$, etc.  
Since $- \frac{\delta \ell}{\delta s}$ is identified with the temperature $T$, the first condition in \eqref{additional_conditions} implies that the variable $ \gamma $ is the thermal displacement. From the second condition it follows that $\bar D_t \varsigma$ is the rate of internal entropy production, which must be positive by the second law of thermodynamics, namely
\begin{equation}\label{2nd_law_continuum}
\bar D_t \varsigma= \bar D_t s + \operatorname{div}j_s \geq 0. 
\end{equation}
The last condition in \eqref{additional_conditions} is the insulated boundary condition.

We refer to \cite{GB2019,ElGB2021,GBPu2022,GBPu2023} for the use of this type of variational formulation for modelling purposes in nonequilibrium thermodynamics.

\paragraph{Standard Lagrangian.} By using the Lagrangian for Euler fluids given in \eqref{standard_ell} 
one gets $\frac{\delta \ell}{\delta u}= \rho  u$, $\frac{\delta  \ell}{\delta  \rho  } = \frac{1}{2}| u | ^2 - \frac{\partial \epsilon }{\partial \rho  } - \phi $, and $-\frac{\delta \ell}{\delta s}= \frac{\partial \epsilon  }{\partial s}= T >0$ the temperature, so that the fluid momentum equation and the entropy equation in \eqref{Equations_Eulerian} take the usual form
\[
\rho  ( \partial _t u+ u \cdot \nabla u)= - \rho  \nabla \phi - \nabla p + \operatorname{div} \sigma, \qquad T (\bar D_ts  + \operatorname{div} j _s) =  \sigma \!: \!\nabla u - j _s\! \cdot  \!\nabla T
\]
with $p= \frac{\partial \epsilon }{\partial \rho  } \rho  + \frac{\partial \epsilon }{\partial s} s - \epsilon$ the pressure.

\paragraph{Energy and entropy balances.} Defining the total energy $ \mathcal{E} = \left\langle \frac{\delta \ell}{\delta u}, u \right\rangle - \ell(u, \rho  , s)$, from \eqref{Equations_Eulerian} and the boundary conditions $u|_{ \partial \Omega }=0$ and $j_s \cdot n=0$ we directly get the energy conservation
\begin{equation}\label{energy_cons} 
\frac{d}{dt} \mathcal{E} = \int_ \Omega \operatorname{div}\Big( \rho  \frac{\delta \ell}{\delta \rho  } u + s  \frac{\delta \ell}{\delta s  }u + \sigma \cdot u + j_s \frac{\delta \ell}{\delta s} \Big) {\rm d} x=0.
\end{equation}
From now on we will focus on the Navier-Stokes-Fourier case with the phenomenological relations
\begin{equation}\label{phenom_relations}
\begin{aligned} 
\sigma =\sigma (u)&= 2 \mu \operatorname{Def} u +  \lambda (\operatorname{div}u) \delta, & &\text{with $\mu \geq 0$ and $\zeta = \lambda + \frac{2}{d} \mu \geq 0$}\\
j_s=j_s (T)&=-  \frac{1}{T} \kappa \nabla T, & &\text{with $\kappa \geq 0$}.
\end{aligned}
\end{equation} 
Here $\operatorname{Def}u = \frac{1}{2}(\nabla u + \nabla u^\mathsf{T}) $ is the rate of deformation tensor, $\mu\geq 0$ and $ \zeta = \lambda + \frac{2}{d} \mu \geq 0$ are the shear and bulk viscosity coefficients, and $ \kappa\geq 0$ is the thermal conductivity coefficient. The signs of the coefficients are imposed by the second law of thermodynamics $ \bar D _t s+ \operatorname{div}j_s \geq 0$ . Indeed, with \eqref{phenom_relations}, the entropy equation reads
\begin{align}
\bar D_t s+ \operatorname{div}j_s &=  \frac{1}{T} \sigma \!: \!\nabla u - \frac{1}{T} j _s\! \cdot  \!\nabla T \nonumber \\
&= \frac{2 \mu }{T}\left( \operatorname{Def} u \right) ^{(0)} \!: \! \left( \operatorname{Def} u \right) ^{(0)} + \frac{\zeta}{T} (\operatorname{div}u) ^2  +  \frac{\kappa}{T ^2 } | \nabla T| ^2\geq 0, \label{entrop_prod} 
\end{align} 
where $\left( \operatorname{Def} u \right) ^{(0)}$ denotes the trace-free part of $\operatorname{Def} u$.
For simplicity, in this paper we will assume that these coefficients are constant; however, see \S\ref{enhancements}.

The discretization that we will develop preserves both the energy conservation \eqref{energy_cons} and the positivity of the rate of internal entropy production \eqref{entrop_prod}.

\subsection{Dirichlet boundary conditions}\label{Dirichlet_T}

We present here a modification of the variational formulation which allows the treatment of a fluid with prescribed temperature on the boundary $ \partial \Omega $. 

\paragraph{Lagrangian description.} Since in this case the fluid system is no longer adiabatically closed, 
the appropriate formulation is found by considering the continuum version of the variational formulation for finite dimensional open thermodynamic systems developed in \cite{GBYo2018b}. This amounts to replacing the constraints \eqref{KC_fluid} and \eqref{VC_fluid} by the kinematic and variational constraints
\begin{equation}
\begin{aligned}\label{KC_fluid_W}
\int_ \Omega W \frac{\delta  L }{\delta  S}\dot \Sigma \,{\rm d} X = - \int_ \Omega W \underbrace{(P: \nabla\dot \varphi)}_{\text{viscosity}}\, {\rm d} X +  \int_ \Omega W\!\!\!\!\underbrace{(J_S \cdot \nabla \dot\Gamma)}_{\text{heat conduction}}\!\!\!\! \, {\rm d} X \\ - \int_ { \partial \Omega } W \underbrace{(J_S \cdot n) (\dot \Gamma - \mathfrak{T} _0)}_{\text{heat transfer}}\!\, {\rm d} A,\;\forall\;W
\end{aligned}
\end{equation}
\begin{equation}
\begin{aligned}\label{VC_fluid_W}
\int_ \Omega W\frac{\delta   L}{\delta  S}\delta \Sigma \, {\rm d} X=  - \int_ \Omega W\underbrace{(P: \nabla \delta  \varphi)}_{\text{viscosity}} \,{\rm d} X +  \int_ \Omega W\!\!\!\!\underbrace{(J_S \cdot \nabla \delta \Gamma)}_{\text{heat conduction}}\!\!\!\!\, {\rm d} X \\ - \int_ { \partial \Omega } W \underbrace{(J_S \cdot n)  \delta  \Gamma}_{\text{heat transfer}} \!\,  {\rm d} A ,\;\forall\;W,
\end{aligned}
\end{equation} 
with $ \mathfrak{T}_0(X)$ the prescribed temperature on the boundary in the Lagrangian description.
When $J_S \cdot n=0$ they consistently recover \eqref{KC_fluid} and \eqref{VC_fluid}.

We refer to the \ref{A} for an application of this type of variational principle for an elementary finite dimensional thermodynamic system exchanging heat with the exterior, which helps the understanding of the new boundary term. In particular, \eqref{KC_fluid_W} and \eqref{VC_fluid_W} are continuum versions of the constraints \eqref{CK_simple_open} and \eqref{CV_simple_open} in \ref{A}.

\paragraph{Eulerian description.} Exactly as in \S\ref{subsec_insulated}, the variational principle can be converted to the Eulerian description, thereby yielding \eqref{VP_NSF_spatial}--\eqref{EP_constraints} with the constraints \eqref{KC_NSF_spatial} and \eqref{VC_NSF_spatial} replaced by
\begin{equation}
\begin{aligned}\label{KC_NSF_spatial_w_Dirichlet}
\int_ \Omega w\frac{\delta   \ell}{\delta   s} \bar D_t \varsigma  \,{\rm d} x =-  \int_ \Omega w(\sigma : \nabla u)\, {\rm d} x +   \int_ \Omega w (j_s \cdot \nabla D_t \gamma)\, {\rm d} x  \\ - \int_ { \partial \Omega } w(j_s \cdot n) \left( D_ t \gamma  -T_0 \right)\,  {\rm d} a,\;\forall\; w
\end{aligned}
\end{equation} 
\begin{equation}
\begin{aligned}\label{VC_NSF_spatial_w_Dirichlet}
\int_ \Omega w\frac{\delta   \ell}{ \delta  s}  \bar D_\delta  \varsigma\, {\rm d} x  =-\int_ \Omega w (\sigma : \nabla  v)\, {\rm d} x+  \int_ \Omega w (j_s \cdot \nabla D_ \delta \gamma)\, {\rm d} x  \\ - \int_ { \partial \Omega } w(j_s \cdot n) D_ \delta \gamma \,  {\rm d} a,\;\forall\;w.
\end{aligned}
\end{equation}
An application of the variational formulation \eqref{VP_NSF_spatial}-\eqref{KC_NSF_spatial_w_Dirichlet}-\eqref{VC_NSF_spatial_w_Dirichlet}-\eqref{EP_constraints} yields the equations \eqref{Equations_Eulerian} and \eqref{additional_conditions}, with the last equation of \eqref{additional_conditions} replaced by
\begin{equation}\label{Dirichlet_BC} 
(j_s \cdot n)( T-T_0)=0 \quad\text{on}\quad  \partial \Omega,
\end{equation} 
see \S\ref{B_2} for details.
With the boundary condition \eqref{Dirichlet_BC}, the energy balance \eqref{energy_cons} is modified as
\begin{equation}\label{energy_cons_modified} 
\frac{d}{dt} \mathcal{E} = \int_ \Omega \operatorname{div}\Big( \rho  \frac{\delta \ell}{\delta \rho  } u + s  \frac{\delta \ell}{\delta s  }u + \sigma \cdot u + j_s \frac{\delta \ell}{\delta s} \Big) {\rm d} x=- \int_{ \partial \Omega } (j_s \cdot n)T_0 \, {\rm d} a.
\end{equation}

\begin{remark}[Prescribed heat flux]\label{PHF}\rm It is also possible to treat the boundary condition
\[
j_q \cdot n = q_0
\]
where $j_q=Tj_s$ is the heat flux and for some given function $q_0: \partial \Omega \rightarrow \mathbb{R} $, which may itself depend on the boundary temperature as $q_0(T)$. This is achieved by replacing the last term in \eqref{KC_NSF_spatial_w_Dirichlet} with
\begin{equation}\label{modif_presc_flux} 
- \int_ { \partial \Omega } w\big( (j_s \cdot n) D_ t \gamma  - q_0 \big)\, {\rm d} a.
\end{equation} 
The variational constraint \eqref{VC_NSF_spatial_w_Dirichlet} is kept unchanged, as it follows from the general variational framework for open systems, see \ref{A}. The energy balance now reads
\[
\frac{d}{dt} \mathcal{E}= - \int_{ \partial \Omega } q_0 \, {\rm d} a.
\]
\end{remark}

\subsection{Associated weak formulation}\label{weak_form_continuous}

The variational derivation presented above yields a weak formulation of the equations and boundary conditions, that will be shown to have a discrete version.

In the absence of irreversible processes, this weak formulation has been derived in \cite{GaGB2020} and is based on the trilinear forms $a : L^\infty(\Omega)^d \times H^1(\Omega)^d \times H^1(\Omega)^d \to \mathbb{R}$ and $b : L^2(\Omega) \times H^1(\Omega) \times L^\infty(\Omega)^d \to \mathbb{R}$ defined by
\begin{align*} 
a(w,u,v)&= -\int_ \Omega w \cdot [u,v] \,{\rm d} x\\
b( f, \rho  , v)&= - \int_ \Omega \rho  \nabla f \cdot v \,{\rm d} x.
\end{align*}

For the treatment of the irreversible part, we restrict to the expressions $ \sigma = \sigma (u)$ and $ j_s=j_s(T)$ given in \eqref{phenom_relations}. For viscosity, we define the trilinear form $c : L^\infty(\Omega) \times H^1(\Omega)^d \times H^1(\Omega)^d \rightarrow \mathbb{R}$ by
\begin{equation}\label{def_c} 
c(w,u,v) = \int_\Omega w \, \sigma (u) : \nabla v \, {\rm d}x.
\end{equation}
For heat conduction, we set $F = \{f \in H^1(\Omega) \mid 1/f \in L^\infty(\Omega), \, \dv(\nabla f/f) \in L^2(\Omega)\}$ and define $d : W^{1,\infty}(\Omega) \times F \times H^1(\Omega) \rightarrow \mathbb{R}$ by
\begin{equation}\label{d}
d(w,f,g)=\left\{
\begin{array}{ll}
\vspace{0.2cm}\displaystyle\int_\Omega w \, j_s(f) \cdot \nabla g \, {\rm d}x & \;\text{for homogeneous Neumann}\\
\displaystyle \int_\Omega w \, j_s(f) \cdot \nabla g \, {\rm d}x - \int_{\partial\Omega} w j_s(f) \cdot n g \, {\rm d}a & \;\parbox{15em}{ for Dirichlet and \\nonhomogeneous Neumann}
\end{array}
\right. 
\end{equation} 
and $e :  W^{1,\infty}(\Omega) \times F \rightarrow \mathbb{R}$ by
\begin{equation}\label{e}
e(w,f)=\left\{
\begin{array}{ll}
\vspace{0.2cm}\displaystyle 0 & \text{for homogeneous Neumann}\\
\vspace{-0.4cm}\displaystyle \int_{\partial\Omega} w j_s(f) \cdot n T_0 \, {\rm d}a & \text{for Dirichlet}\\
\\
\displaystyle  \int_{\partial\Omega} w q_0(f) \, {\rm d}a &  \text{for nonhomogeneous Neumann}
\end{array}
\right. 
\end{equation} 
so that the constraints \eqref{KC_NSF_spatial}--\eqref{VC_NSF_spatial} and \eqref{KC_NSF_spatial_w_Dirichlet}--\eqref{VC_NSF_spatial_w_Dirichlet} (including the modified version in \eqref{modif_presc_flux}) can be written in a unified way for all boundary conditions as
\begin{equation}\label{unified_writing_KC} 
\Big\langle  w,\frac{\delta   \ell}{\delta   s} \bar D_t \varsigma \Big\rangle =- c(w,u,u) + d\Big( w, - \frac{\delta \ell}{\delta s}, D_t \gamma \Big)+ e\Big(w, - \frac{\delta \ell}{\delta s} \Big), \quad \forall w
\end{equation}
and
\begin{equation}\label{unified_writing_VC}  
\Big\langle w, \frac{\delta   \ell}{ \delta  s}  \bar D_\delta  \varsigma \Big\rangle =- c(w,u, v ) + d\Big( w, - \frac{\delta \ell}{\delta s}, D_\delta  \gamma \Big), \quad \forall w,
\end{equation} 
with $ \left\langle \cdot , \cdot \right\rangle $ the $L^2$ inner product. Note that the condition $\dv(\nabla f / f) \in L^2(\Omega)$ was included in the definition of $F$ so that the term $-\int_{\partial\Omega} w j_s(f) \cdot n g \, {\rm d}a$ can be given meaning using the identity
\[
\int_{\partial\Omega} w \frac{1}{f} \nabla f \cdot n g \, {\rm d}a =  \int_\Omega \nabla w \cdot \frac{1}{f} (\nabla f) g \, {\rm d}x + \int_\Omega w \dv\left(\frac{1}{f} \nabla f\right) g \,{\rm d}x + \int_\Omega w \frac{1}{f} \nabla f \cdot \nabla g \, {\rm d}x.
\]
This condition can be omitted from the definition of $F$ in the homogeneous Neumann setting. Note also that we really have two ways to impose homogeneous Neumann boundary conditions: by choosing $d$ and $e$ as indicated above, or by using the ``nonhomogeneous Neumann'' $d$ and $e$ with $q_0=0$.  Both approaches lead to the same equations of motion, but the first approach is slightly simpler, and it simplifies the statement of the second law of thermodynamics below.  Therefore we prefer to treat it separately.

Importantly, by using these notations when carrying out the variational formulations \eqref{VP_NSF_spatial}--\eqref{EP_constraints} and \eqref{VP_NSF_spatial}-\eqref{KC_NSF_spatial_w_Dirichlet}-\eqref{VC_NSF_spatial_w_Dirichlet}-\eqref{EP_constraints}, we get the equations \eqref{Equations_Eulerian}, together with the boundary conditions $j_s \cdot n=0$ or $(j_s \cdot n)(T-T_0)=0$ in the weak form
\begin{equation}\label{good_weak_form}
\left\{ 
\begin{array}{l}
\displaystyle\vspace{0.2cm}\left\langle \partial _t \frac{\delta \ell}{\delta u}, v  \right\rangle + a \Big( \frac{\delta \ell}{\delta u} , u, v \Big) + b\Big( \frac{\delta \ell}{\delta \rho  } , \rho  , v \Big) +b\Big( \frac{\delta \ell}{\delta s  } , s  , v \Big) =- c(1,u,v),\;\forall\,v\\
\displaystyle\vspace{0.2cm} - \left\langle  \partial _t  s, \frac{\delta \ell}{\delta s}w \right\rangle - b \Big( \frac{\delta \ell}{\delta s}w, s, u\Big) + d \Big( 1, -\frac{\delta \ell}{\delta s}, \frac{\delta \ell}{\delta s} w\Big)\\
\displaystyle\vspace{0.2cm} \hspace{2cm}= c(w, u,u ) + d\Big( w, -\frac{\delta \ell}{\delta s}, \frac{\delta \ell}{\delta s}\Big)- e\Big(w, - \frac{\delta \ell}{\delta s} \Big),\;\forall\,w\\
\displaystyle\left\langle \partial _t \rho  , \theta  \right\rangle + b(\theta , \rho  , u)=0,\;\forall\,\theta,
\end{array}
\right.
\end{equation} 
see \S\ref{B_3} for details. Note in particular the very specific form of the two terms involving $d( \cdot , \cdot , \cdot )$ in the weak form of the entropy equation, which plays a crucial role in our discretization.

With these notations, the energy balance follows as
\begin{equation}\label{energy_bal}
\begin{aligned}
\frac{d}{dt} \mathcal{E} &= \left\langle \partial _t \frac{\delta \ell}{\delta u} , u \right\rangle - \left\langle \frac{\delta \ell}{\delta \rho  }, \partial _t \rho  \right\rangle - \left\langle \frac{\delta \ell}{\delta s  }, \partial _t s  \right\rangle \\
&= -a \Big( \frac{\delta \ell}{\delta u} , u, u \Big) - b\Big( \frac{\delta \ell}{\delta \rho  } , \rho  , u \Big) -b\Big( \frac{\delta \ell}{\delta s  } , s  , u \Big) - c(1,u,u)\\
& \qquad + b\Big( \frac{\delta \ell}{\delta \rho  } , \rho  , u \Big) + b \Big( \frac{\delta \ell}{\delta s}, s, u\Big) - d \Big( 1,- \frac{\delta \ell}{\delta s}, \frac{\delta \ell}{\delta s} \Big)\\
&\qquad+ c(1, u,u ) + d\Big( 1, -\frac{\delta \ell}{\delta s}, \frac{\delta \ell}{\delta s}\Big)- e\Big(1, - \frac{\delta \ell}{\delta s} \Big)=-e\Big(1, - \frac{\delta \ell}{\delta s} \Big)
\end{aligned}
\end{equation}  
from the property
\[
a(w,u,v) =- a(w,v,u), \;\forall\,u,v,w
\]
of the trilinear form $a$.

Conservation of total mass $\int_ \Omega \rho \, {\rm d} x$ follows from the property
\[
b(1, \rho  , v)=0,\;\;\forall \rho  , v.
\]

Regarding entropy production, we have the inequalities
$\sigma(u) \!: \!\nabla u \geq 0$ and $- j _s(T)\! \cdot  \!\nabla T \geq 0$ for all $u$ and all $T>0$, consistently with the second law of thermodynamics written in \eqref{2nd_law_continuum}. In terms of $c$, the first condition is equivalently written as
\begin{equation}\label{positivity_c} 
c(w, u, u) \geq 0, \;\;  \forall\,u, \;\; \forall\, w\geq 0.
\end{equation} 
For homogeneous Neumann boundary conditions, the second condition is equivalently written as
\begin{equation}\label{positivity_d_N} 
d(w, f, f)\leq 0,\;\; \forall\,f>0,\;\; \forall\, w\geq 0,
\end{equation} 
while for Dirichlet and nonhomogeneous Neumann boundary conditions, it can be equivalently written as
\begin{equation}\label{positivity_d_D}
d(w, f, f)\leq 0,\;\; \forall\,f>0,\;\; \forall\, \text{$w\geq 0$, $w$ with compact support in the interior of $ \Omega $},
\end{equation}
where we recall that the expression for $d$ depends on the boundary condition used.
Discrete versions of \eqref{positivity_c}--\eqref{positivity_d_D} will be shown to hold in the discrete case.

Finally, the second law \eqref{2nd_law_continuum} can be equivalently written by using $b$ and $d$ as
\begin{equation}\label{2nd_law_continuum_b_d} 
\left\langle \partial _t s, Tw \right\rangle +b(Tw, s, u) - d(1, T, Tw)\geq 0
\end{equation}
for all $w\geq 0$ with compact support in $ \Omega $.

\section{Structure preserving variational discretization}\label{Sec_discrete}

The structure preserving finite element integrator is obtained by developing a discrete version of the variational formulation presented above. In particular, exactly as in the continuous case, the discrete variational formulation of the Eulerian form of the equation is inherited by a variational formulation extending Hamilton's principle in the Lagrangian description.

This is achieved thanks to the introduction of a discrete version $G_h$ of the diffeomorphism group $ \operatorname{Diff}( \Omega )$ of fluid motion, acting on discrete functions and densities. We shall follow the approach developed for compressible fluids in \cite{GaGB2020} based on the earlier works \cite{PaMuToKaMaDe2010,GaMuPaMaDe2011,DeGaGBZe2014,NaCo2018,BaGB2019}.
We refer to \cite{CoGB2022} for another approach to the variational discretization of \eqref{VP_NSF_spatial}--\eqref{EP_constraints} for heat conducting viscous fluids, also based on a discrete version of the diffeomorphism group.

\subsection{Discrete setting}

Let $ \mathcal{T} _h$ be a triangulation of $ \Omega $. We regard $ \mathcal{T} _h$ as a member of a family of triangulations parametrized by $h=\max_{K \in \mathcal{T} _h}h_K$, where $h_K= \operatorname{diam} K$ denotes the diameter of a simplex $K$. We assume that this family is shape-regular, meaning that the ratio $\max_{K \in \mathcal{T} _h}h_K/ \rho  _K$ is bounded above by a positive constant for all $h > 0$. Here, $\rho  _K$ denotes the inradius of $K$.

We shall discretize functions with the discontinuous Galerkin space
\[
V_h=DG_q( \mathcal{T} _h):=\{f \in L^2( \Omega ) \mid f|_K \in P_q(K), \;\forall\, K \in \mathcal{T} _h\}.
\]
Considering the discrete diffeomorphism group as a certain subgroup  $G_h \subset GL(V_h)$, the action $f \in V_h \mapsto  gf \in V_h$, $ g \in G_h$ is understood as a discrete analog of the action $f \in \mathcal{F} ( \Omega ) \mapsto f \circ \varphi ^{-1}  \in \mathcal{F} ( \Omega )$, $ \varphi \in \operatorname{Diff}( \Omega )$, of diffeomorphisms on functions. The discrete analog of the action $ \rho  \in \mathcal{F} ( \Omega ) ^*\mapsto (\rho  \circ \varphi )J \varphi \in \mathcal{F} ( \Omega ) ^* $ on densities, written as $ \rho  \in V_h \mapsto \rho  \cdot g \in V_h$, is defined by $L^2$ duality as in the continuous case
\begin{equation}\label{def_dual} 
\left\langle \rho \cdot g , f \right\rangle = \left\langle \rho , g f \right\rangle , \quad \forall f, \rho   \in V_h.
\end{equation}
In particular, the Lagrange-to-Euler relations \eqref{Lagr_to_Eul_1} have the discrete analog
\begin{equation}\label{Lagr_to_Eul_1_discrete} 
\begin{aligned} 
A&= \dot  g g ^{-1} , \qquad \rho  = \varrho _0 \cdot g ^{-1} \in V_h , \qquad s= S \cdot g ^{-1} \in V_h,\\
\varsigma&=\Sigma \cdot g^{-1} \in V_h, \qquad \gamma = g \Gamma  \in V_h.
\end{aligned} 
\end{equation}

The Lie algebra $ \mathfrak{g} _h \subset L(V_h,V_h)$ of $G_h$ acts on discrete functions and densities as $f \in V_h \mapsto fA \in V_h$ and $ \rho  \in V_h \mapsto \rho  \cdot A \in V_h$, which satisfy
\begin{equation}\label{discrete_action_rho} 
\left\langle \rho   \cdot A, f \right\rangle = \left\langle \rho  , A f \right\rangle , \quad \forall f, \rho   \in V_h.
\end{equation} 
As shown in \cite{GaGB2020}, the realization of elements of this Lie algebra as discrete vector fields is obtained by associating to each $u \in H_0(\operatorname{div}, \Omega ) \cap L^p( \Omega )^d$ (with $p>2$) the Lie algebra element $A_u \in \mathfrak{g} _h$ defined by
\[
\left\langle A_u f, g \right\rangle := -\sum_{K \in \mathcal{T} _h}\int_K( \nabla _uf)g\, {\rm d} x + \sum_{e \in \mathcal{E} ^0_h}\int_e u \cdot \llbracket f \rrbracket \{g\} \, {\rm d} a, \quad \forall f, g   \in V_h,
\]
which yields a consistent approximation of the distributional derivative in the direction $u$. 
Moreover, the linear map $u \in H_0(\operatorname{div}, \Omega ) \cap L^p( \Omega )^d\mapsto A_u \in \mathfrak{g} _h$ becomes injective on the Raviart-Thomas finite element space $RT_{2q}( \mathcal{T} _h)$, see Prop. 3.4 in \cite{GaGB2020}. Above, we used the notation $\llbracket f \rrbracket$ and $\{ f \}$ for the jump and average of a scalar function $f$ across $e = K_1 \cap K_2 \in \mathcal{E}_h^0$, which are defined by
\[
\llbracket f \rrbracket = f_1 n_1 + f_2 n_2, \quad \{ f \} = \frac{1}{2}(f_1 + f_2).
\]
Here, $f_i = \left.f\right|_{K_i}$ and $n_i$ is the unit normal vector to $e$ pointing outward from $K_i$.  Later we will also apply $\{\cdot\}$ to vector fields and interpret it componentwise.

This setting is used in \cite{GaGB2020} to develop a finite element variational discretization of compressible fluids, by writing the  analog to the Hamilton principle \eqref{HP_fluid} on the discrete diffeomorphism group $G_h$ and using the first three relations \eqref{Lagr_to_Eul_1_discrete} to deduce its Eulerian version. The discrete Lagrangian $\ell_d: \mathfrak{g} _h \times V_h \times V_h \rightarrow \mathbb{R} $ can be defined from a given continuous Lagrangian $\ell$ as
\[
\ell_d(A, \rho  , s):=\ell(\widehat{A}, \rho  , s)
\]
thanks to the Lie algebra-to-vector field map $ A \in \mathfrak{g}_h \mapsto \widehat{A} \in [V_h]^d$ defined by
\[
\widehat{A}:= -\sum_{k=1}^d A ( \pi _h(x^k)) e_k,
\]
where $x^k : \Omega \to \mathbb{R}$ is the $k$th coordinate function, $\{e_k\}_{k=1}^d$ is the standard basis for $\mathbb{R}^d$, and $ \pi _h:L^2( \Omega ) \rightarrow V_h$ the $L^2$-orthogonal projector, which we interpret componentwise when applied to a vector field.  It satisfies $\widehat{A_u}=  \pi _h(u)$. 

\subsection{Discrete variational formulation}

We discretize the velocity by using the continuous Galerkin space 
\[
U_h^{\rm grad}=CG_r( \mathcal{T} _h)^d:=\{u \in H^1_0( \Omega )^d\mid u|_K \in P_r(K)^d,\;\forall\, K \in \mathcal{T} _h\}.
\]
Assuming $r\leq q$ (see Remark \ref{remark_q_r} for $r>q$), we denote by $ \Delta _h \subset \mathfrak{g} _h$ the subspace corresponding to $U_h^{\rm grad}$ via the injective map $u \mapsto A_u$.  We denote by $\pi_h^{\rm grad} : L^2(\Omega)^d \to U_h^{\rm grad}$ the $L^2$-orthogonal projector onto $U_h^{\rm grad}$.

Consider discretizations of $d$ and $e$ given as $d_h:V_h \times V_h \times V_h \rightarrow \mathbb{R} $ and $e_h: V_h \times V_h \rightarrow \mathbb{R}$. For now, we only suppose that $d_h$ is linear in its first and third arguments, and $e_h$ is linear in its first argument. The explicit form and their properties are stated later.

With this setting, the discrete version of the variational formulation \eqref{VP_NSF_spatial}-\eqref{unified_writing_KC}-\eqref{unified_writing_VC}-\eqref{EP_constraints} reads as follows. Find $A :[t_0,t_1] \rightarrow  \Delta _h$ and $ \rho  , s  , \varsigma , \gamma :[t_0,t_1] \rightarrow  V_h$ which are critical for the \textit{variational condition}
\begin{equation} \label{VP_NSF_spatial_discrete}
\delta \int_{t_0}^{t_1} \Big[ \ell_d(A,\rho,s) + \langle s- \varsigma, D_t^h \gamma  \rangle \Big] \, {\rm d}t = 0,
\end{equation}
with the \textit{phenomenological constraint} and \textit{variational constraint} given by
\begin{equation} \label{KC_NSF_spatial_discrete}
\!\!\!\left\langle \frac{\delta  \ell_d}{\delta s} \bar D_t^h \varsigma, w \right\rangle = - c(w, \widehat{A},\widehat{A}) + d_h\left(w , -\frac{\delta \ell_d}{\delta s}, D_t^h \gamma  \right) + e_h \left( w, -\frac{\delta \ell_d}{\delta s}\right) ,  \forall w \in V_h,
\end{equation}
\begin{equation} \label{VC_NSF_spatial_discrete}
\left\langle \frac{\delta \ell_d}{\delta s}  \bar D_\delta ^h\varsigma, w \right\rangle = - c(w, \widehat{A},\widehat{B}) + d_h\left(w , -\frac{\delta \ell_d}{\delta s},D_ \delta ^h \gamma \right),  \forall w \in V_h,
\end{equation}
and with Euler-Poincar\'e variations
\begin{align}
\delta A &= \partial_t B + [B,A], \label{discrete_EP1} \\
\delta \rho &= -\rho \cdot B, \label{discrete_EP2}
\end{align}
for $B (t) $ an arbitrary curve in $\Delta_h$ with $B(0)=B(T)=0$, and with $\delta\gamma|_{t=0,T}=0$.

Above we have defined the discrete analogs to the Lagrangian time derivatives and variations considered in \eqref{Lagr_derivatives} as 
\begin{equation}\label{Lagr_derivatives_discrete} 
\begin{aligned} 
D_t^hf &= \partial _t f -  A f  &  & \quad D_\delta^h  f = \delta f  - B f\\
\bar D_t^h \rho  &= \partial _t \rho  + \rho   \cdot A& &\quad  \bar D_\delta ^h \rho  = \delta \rho    + \rho   \cdot B.
\end{aligned}
\end{equation} 

In \eqref{KC_NSF_spatial_discrete} and \eqref{VC_NSF_spatial_discrete}, the partial derivative $ \frac{\delta \ell_d}{\delta s} \in V_h$ is defined exactly as in the continuous case, with respect to the $L^2$ duality pairing on $V_h$.

\begin{proposition} The equations of motion that result from \eqref{VP_NSF_spatial_discrete}--\eqref{VC_NSF_spatial_discrete} and from the definition $ \rho  = \varrho _{0} \cdot g ^{-1} $ are
\begin{equation}\label{Eulerian_multi_discrete} 
\left\{
\begin{array}{l}
\vspace{0.2cm}\displaystyle
\left\langle  \partial_t \frac{\delta  \ell}{\delta u}, \vv  \right\rangle  + a_h\Big( \pi _h \frac{\delta \ell}{\delta u}, u , \vv \Big)  + b_h \Big( \pi _h \frac{\delta \ell}{\delta \rho   }, \rho , \vv \Big) \\
\displaystyle\vspace{0.2cm} \hspace{4cm}+ b_h \Big( \pi _h \frac{\delta \ell}{\delta s   }, s , \vv \Big) =- c(1, u, \vv ), \;\; \forall\, \vv \in U_h^{\rm grad},\\
\vspace{0.2cm}\displaystyle \left\langle \partial _t \rho  , \theta \right\rangle +b_h( \theta , \rho  , u)=0, \;\; \forall \,\theta \in V_h,\\
\vspace{0.2cm}\displaystyle - \left\langle  \partial _t  s, \pi_h \Big(\frac{\delta \ell}{\delta s}\Big)w  \right\rangle - b_h \Big( \pi _h \Big( \pi_h \Big(\frac{\delta \ell}{\delta s} \Big) w \Big), s, u\Big)+ d_h \Big( 1, -\pi_h \frac{\delta  \ell}{\delta  s}, \pi _h \Big( \pi_h \Big( \frac{\delta  \ell}{\delta  s} \Big) w \Big) \Big)\\
\displaystyle\vspace{0.2cm} \hspace{2cm}=  c(w,u,u) +d_h \Big( w,-\pi_h \frac{\delta \ell}{\delta s} , \pi_h \frac{\delta \ell}{\partial s}  \Big)- e_h\Big(w, - \pi_h  \frac{\delta \ell}{\delta s} \Big), \;\; \forall \,w \in V_h,
\end{array}
\right.
\end{equation}
where
\begin{align} 
a_h(w, u,v)&=- \int_ \Omega \widehat{A_w} \cdot \widehat{[A_u, A_v]} \,{\rm d} x=- \int_ \Omega w \cdot [u,v]\, {\rm d} x\label{def_ah}\\
b_h(f, \rho  , u)&= \langle \rho  , A_u f  \rangle = - \sum_{K \in \mathcal{T} _h}\int_K( \nabla _uf) \rho\, {\rm d} x + \sum_{e \in \mathcal{E} ^0_h}\int_e u \cdot \llbracket f \rrbracket \{\rho\} \, {\rm d} a\label{def_bh}.
\end{align}
The variational principle \eqref{VP_NSF_spatial_discrete}--\eqref{VC_NSF_spatial_discrete}  also yields the conditions
\begin{equation}\label{additional_conditions_discrete} 
D^h_t \gamma = - \pi _h \frac{\delta  \ell}{\delta  s}, \qquad \Big\langle \bar D_t^h (\varsigma  - s ), \delta \gamma  \Big\rangle =- d_h \Big( 1, -\pi _h\frac{\delta  \ell}{\delta  s}, \delta \gamma  \Big),\;\forall\, \delta \gamma  \in V_h.
\end{equation} 
\end{proposition}
\begin{remark}\label{remark_q_r}
The second equality in~(\ref{def_ah}) is valid if $r \le q$.  If $r>q$, then we can still arrive at the same scheme by considering a discrete diffeomorphism group $G_h \subset GL(DG_r(\mathcal{T}_h))$ and treating $\rho,s,\varsigma,\gamma$ as elements of $DG_q(\mathcal{T}_h) \subset DG_r(\mathcal{T}_h)$; see Remark 5.1 in \cite{GaGB2020}.  In both cases, the fact that our discrete velocity field is continuous is important.  In the absence of continuity, the formula~(\ref{def_ah}) contains additional terms involving jumps across codimension-1 faces; see Proposition 4.3 in \cite{GaGB2020}.  The absence of these jump terms renders $a_h$ independent of $h$, so we drop the subscript $h$ in what follows. 
\end{remark}

\begin{remark}
Note that \eqref{def_bh} is a standard discontinuous Galerkin discretization of the scalar advection operator, see \cite{BrMaSu2004}.
\end{remark}

\begin{proof} Taking the variations in \eqref{VP_NSF_spatial_discrete} yields
\begin{equation}
\begin{aligned} 
&\int_{t_0}^{t_1}  \Big[ \Big\langle\!\!\Big\langle \frac{\delta \ell_d}{\delta A} , \delta A \Big\rangle\!\!\Big\rangle + \Big\langle \frac{\delta \ell_d}{\delta \rho  } ,\delta   \rho \Big\rangle   +\Big\langle  \frac{\delta \ell_d}{\delta s} ,\delta s \Big\rangle  + \Big\langle \delta s,  D_t^h \gamma \Big\rangle  - \Big\langle \delta \varsigma,  D_t^h \gamma \Big\rangle  \\
& \qquad \qquad \qquad  -  \Big\langle \bar D_t^h(s- \varsigma ) ,\delta \gamma \Big\rangle  - \Big\langle s- \varsigma , \delta A  \gamma  \Big\rangle \Big] {\rm d}t=0,
\end{aligned} 
\end{equation}
where $\big\langle\!\big\langle \cdot, \cdot \big\rangle\!\big\rangle : \mathfrak{g}_h^* \times \mathfrak{g}_h \to \mathbb{R}$ denotes the duality pairing on $\mathfrak{g}_h^* \times \mathfrak{g}_h$.
We used $\int_{t_0}^{t_1}  \left\langle s - \varsigma , D_t^ h \delta \gamma  \right\rangle {\rm d} t= -\int_{t_0}^{t_1}  \left\langle \bar D_t^h (s - \varsigma ), \delta \gamma  \right\rangle {\rm d}t $ which follows from \eqref{discrete_action_rho}, \eqref{Lagr_derivatives_discrete}, and $ \delta {\gamma}| _{t=t_0,t_1}=0$.
Since $ \delta s $ is arbitrary in $V_h$, we get the condition
\begin{equation}\label{D_t_gamma_h} 
D_t^h \gamma = - \frac{\delta \ell_d}{\delta s}.
\end{equation} 
Making use of this condition and using \eqref{VC_NSF_spatial_discrete} yields 
\begin{align*}
&\int_{t_0}^{t_1}   \Big[\Big\langle\!\!\Big\langle \frac{\delta \ell_d}{\delta A} , \delta A \Big\rangle\!\!\Big\rangle + \Big\langle \frac{\delta \ell_d}{\delta \rho } , \delta \rho \Big\rangle  - c(1, \widehat{A}, \widehat{B}) + d_ h \Big( 1, -\frac{\delta \ell_d}{\delta s}, D_ \delta ^h \gamma  \Big) \\
&\qquad \qquad \qquad \qquad   - \Big\langle \frac{\delta \ell_d}{\delta s}, \varsigma \cdot B \Big\rangle - \Big\langle   \bar D_{ t}^{ h}(s- \varsigma ) , \delta \gamma \Big\rangle  - \Big\langle s- \varsigma , \delta A  \gamma  \Big\rangle \Big] {\rm d}t=0.
\end{align*}
Using that $ \delta \gamma $ is arbitrary and independent of the other variations, we get
\begin{equation}\label{D_t_s_h} 
\Big\langle \bar D_t^h ( \varsigma -s), \delta \gamma  \Big\rangle = - d_h \Big( 1, -\frac{\delta \ell_d}{\delta s}, \delta \gamma  \Big) ,\;\;\forall\, \delta \gamma  \in V_h.
\end{equation}
With this, the previous condition becomes
\[
\begin{aligned}
\int_{t_0}^{t_1}   \Big[\Big\langle\!\!\Big\langle \frac{\delta \ell_d}{\delta A}, \delta A  \Big\rangle\!\!\Big\rangle- \Big\langle s- \varsigma , \delta A  \gamma  \Big\rangle - \Big\langle \frac{\delta \ell_d}{\delta \rho  } ,  \rho  \cdot B \Big\rangle - \Big\langle \frac{\delta \ell_d}{\delta s  } ,  \varsigma   \cdot B \Big\rangle  \\- c(1, \widehat{A}, \widehat{B}) - d_h \Big( 1, -\frac{\delta \ell_d}{\delta s}, B \gamma  \Big)\Big] {\rm d} t=0.
\end{aligned}
\]
Using $ \delta A = \partial _t B + [B, A]$ and the computation
\begin{align*}
&\int_{t_0}^{t_1}  \left\langle s- \varsigma , \delta A  \gamma  \right\rangle {\rm d} t=\int_{t_0}^{t_1}  \left\langle s- \varsigma , ( \partial _t B + [B, A])\gamma  \right\rangle {\rm d} t\\
&=\int_{t_0}^{t_1}  \left\langle - \partial _t(s- \varsigma ), B\gamma  \right\rangle- \left\langle s- \varsigma , B\partial _t \gamma  \right\rangle + \left\langle (s- \varsigma ) \cdot B, A \gamma \right\rangle- \left\langle (s- \varsigma ) \cdot A, B \gamma \right\rangle {\rm d} t\\
&=\int_{t_0}^{t_1}  \left\langle - \bar D^h _t(s- \varsigma ), B\gamma  \right\rangle- \Big\langle (s- \varsigma) \cdot B , D^h _t \gamma  \Big\rangle  {\rm d} t\\
&=\int_{t_0}^{t_1}  - d_h \Big(1, -\frac{\delta \ell_d}{\delta s}, B \gamma \Big) + \Big\langle (s- \varsigma) \cdot B , \frac{\delta \ell_d}{\delta s } \Big\rangle  {\rm d} t,
\end{align*}
the previous condition becomes
\[
\int_{t_0}^{t_1} \Big[\Big\langle\!\!\Big\langle - \partial _t \frac{\delta \ell_d}{\delta A}, B\Big\rangle\!\!\Big\rangle  + \Big\langle\!\!\Big\langle \frac{\delta \ell_d}{\delta A}, [B,A]\Big\rangle\!\!\Big\rangle  - \Big\langle \frac{\delta \ell_d}{\delta \rho  } ,  \rho  \cdot B \Big\rangle - \Big\langle \frac{\delta \ell_d}{\delta s  } ,  s   \cdot B \Big\rangle  - c(1, \widehat{A}, \widehat{B}) \Big] {\rm d} t=0.
\]
Using that $A \in \Delta _h$ means $A= A_u$ for $u \in U^{\rm grad}_h$ and using the definitions of $a$ and $b_h$ in \eqref{def_ah} and \eqref{def_bh}, together with $ \frac{\delta \ell_d}{\delta \rho  }= \pi_h \frac{\delta \ell}{\delta \rho  } $, $ \frac{\delta \ell_d}{\delta s  }= \pi _h\frac{\delta \ell}{\delta s  } $,  and $ \Big\langle\!\!\Big\langle  \frac{\delta \ell_d}{\delta A_u}, C \Big\rangle\!\!\Big\rangle = \left\langle \pi _h \frac{\delta \ell}{\delta u}, \widehat{C} \right\rangle $, we get the first equations of \eqref{Eulerian_multi_discrete}. 

Combining the constraint \eqref{KC_NSF_spatial_discrete} and condition \eqref{D_t_gamma_h}, we get 
\[
\Big\langle \bar D_t^h \varsigma ,  \pi _h \Big( \frac{\delta \ell_d}{\delta s}  w \Big)  \Big\rangle = - c(w, \widehat{A},\widehat{A}) - d_h\Big(w , -\frac{\delta \ell_d}{\delta s}, \frac{\delta \ell_d}{\delta s} \Big)+ e_h \Big( w, -\frac{\delta \ell_d}{\delta s}\Big), \quad \forall w \in V_h.
\]
Using \eqref{D_t_s_h} and $A=A_u$ this becomes
\[
\begin{aligned}
\Big\langle \bar D_t^h s ,  \pi _h \Big( \frac{\delta \ell_d}{\delta s}  w \Big)  \Big\rangle - d_h \Big( 1, -\frac{\delta \ell_d}{\delta s}, \pi _h \Big( \frac{\delta \ell_d}{\delta s}  w \Big)  \Big) = - c(w, u,u) - d_h\Big(w , -\frac{\delta \ell_d}{\delta s}, \frac{\delta \ell_d}{\delta s} \Big) \\+ e_h \Big( w, -\frac{\delta \ell_d}{\delta s}\Big),
\end{aligned}
\]
$\forall w \in V_h$.
We have thus derived the entropy equation in \eqref{Eulerian_multi_discrete}. The mass density equation readily follows from $\rho  = \varrho _0 \cdot g ^{-1} $ while the conditions \eqref{additional_conditions_discrete} have been obtained in \eqref{D_t_gamma_h} and  \eqref{D_t_s_h}.
\end{proof}

\paragraph{Energy balance and mass conservation.} The exact same proof as in the continuous case works, see \eqref{energy_bal}, giving energy balance and mass conservation in the spatially discrete setting as
\begin{align}
\frac{d}{dt} \mathcal{E}& = - e_h\Big(1, - \pi _h\frac{\delta \ell}{\delta s} \Big)\label{semidiscrete_energy}\\
\frac{d}{dt} \int_ \Omega \rho \, {\rm d} x&=0.\label{semidiscrete_mass}
\end{align}

\begin{remark}
The conservation properties above continue to hold if we omit the outermost projection $\pi_h$ from the terms of the form $\pi _h \Big( \pi_h \Big(\frac{\delta \ell}{\delta s} \Big) w \Big)$ in~(\ref{Eulerian_multi_discrete}).  They also continue to hold if we omit the $\pi_h$ from $\pi_h \frac{\delta\ell}{\delta u}$ in~(\ref{Eulerian_multi_discrete}).  We find it advantageous to make these modifications, since they simplify the implementation of the scheme.  We do not omit the $\pi_h$ from $\pi_h \frac{\delta\ell}{\delta \rho}$, since this interferes with the balance of energy.  Likewise, we do not omit the $\pi_h$ from $\pi_h \frac{\delta\ell}{\delta s}$, since we have found that the time-discrete version of~(\ref{Eulerian_multi_discrete}) is difficult to solve numerically without it; Newton's method converges much more reliably when the projection is present.

In summary, we solve
\begin{equation}\label{Eulerian_multi_discrete3} 
\!\!\left\{
\begin{array}{l}
\vspace{0.2cm}\displaystyle
\left\langle  \partial_t \frac{\delta  \ell}{\delta u}, \vv  \right\rangle  + a\Big( \frac{\delta \ell}{\delta u}, u , \vv \Big)  + b_h \Big( \pi _h \frac{\delta \ell}{\delta \rho   }, \rho , \vv \Big) - b_h( T, s , \vv ) =- c(1, u, \vv ), \, \forall\, \vv \in U_h^{\rm grad},\\
\vspace{0.2cm}\displaystyle \left\langle \partial _t \rho  , \theta \right\rangle +b_h( \theta , \rho  , u)=0, \;\; \forall \,\theta \in V_h,\\
\vspace{0.2cm}\displaystyle \left\langle  \partial _t  s, T w  \right\rangle + b_h( T w, s, u) - d_h( 1, T, T w) \\ \quad =  c(w,u,u) - d_h(w, T , T)- e_h(w, T), \;\; \,w \in V_h,
\end{array}
\right.
\end{equation}
where $T = - \frac{\delta \ell_d}{\delta s}= -\pi_h \frac{\delta \ell}{\delta s}$.
\end{remark}

\subsection{Discretization of thermodynamic fluxes}

We now discuss our discretizations $d_h$ and $e_h$ of the maps $d$ and $e$ defined in~(\ref{d}) and~(\ref{e}).  We will design $d_h$ and $e_h$ so that the discrete entropy equation
\begin{equation} \label{discrete_entropy}
\begin{aligned}
&\langle  \partial_t  s, T w \rangle + b_h(Tw, s, u)  - d_h  ( 1, T, Tw )\\
&\qquad \qquad \qquad =  c(w,u,u) -d_h  ( w,T ,T   )- e_h (w, T ), \;\; \forall w \in V_h
\end{aligned}
\end{equation}
yields a consistent discretization of the continuous entropy equation
\[
T(\bar{D}_t s + \dv j_s) = \sigma : \nabla u - j_s \cdot \nabla T.
\]
We restrict our attention to the setting where $j_s=j_s(T) = -\frac{1}{T} \kappa \nabla T$, so that the above equation simplifies to
\[
T \bar{D}_t s - \sigma : \nabla u = \kappa \Delta T.
\] 

Note that an integration by parts shows that the exact solution $(u,\rho,s,T)$ of the continuous problem satisfies
\[
\langle  \partial_t  s, T w \rangle + b_h(Tw, s, u) - c(w,u,u) = \langle T \bar{D}_t s - \sigma : \nabla u, w \rangle, \quad \forall w \in V_h.
\]
Thus, to ensure consistency, we aim to design $d_h$ and $e_h$ so that the exact solution also satisfies
\begin{equation} \label{consistent}
d_h  ( 1, T, Tw ) -d_h  ( w,T ,T   )- e_h (w, T ) = \langle \kappa \Delta T, w \rangle, \quad \forall w \in V_h.
\end{equation}

Our discussion is split into three cases: homogeneous Neumann boundary conditions on $T$, Dirichlet boundary conditions on $T$, and nonhomogneous Neumann boundary conditions on $T$.  To simplify the discussion, we consider the setting where~(\ref{discrete_entropy}) is implemented with $T=-\frac{\delta \ell}{\delta s}$.  The case in which one uses $T = -\pi_h \frac{\delta \ell}{\delta s}$ is addressed in Remark~\ref{remark:variationalcrime}.

\paragraph{Homogeneous Neumann boundary conditions.}
In this case we set $e_h=0$ and discretize $d$ with
\begin{equation} \label{dh2}
\begin{split}
d_h(w,f,g) 
&= -\sum_{K \in \mathcal{T}_h} \int_K \frac{w}{f} \kappa \nabla f \cdot \nabla g \, {\rm d}x + \sum_{e \in \mathcal{E}_h^0} \int_e  \frac{1}{\{f\}} \{ w \kappa \nabla f \} \cdot \llbracket g \rrbracket \,  {\rm d}a \\
&\qquad - \sum_{e \in \mathcal{E}_h^0} \int_e \frac{1}{\{f\}} \{w \kappa \nabla g\} \cdot \llbracket f \rrbracket \, {\rm d}a  - \sum_{e \in \mathcal{E}_h^0} \frac{\eta}{h_e} \int_e  \frac{\{w\}}{\{f\}} \llbracket f \rrbracket \cdot \llbracket g \rrbracket \, {\rm d}a,
\end{split}
\end{equation}
where $\eta >0$ is a penalty parameter.  This is a standard non-symmetric interior penalty discretization of $-\int_\Omega \kappa \nabla f \cdot \nabla g \, {\rm d}x$ (see Section 10.5 in \cite{BrSc2008}), generalized to incorporate the weight $w/f$ appearing in $d(w,f,g) = -\int_\Omega \frac{w}{f} \kappa \nabla f \cdot \nabla g \, dx$.
Using the identity $\llbracket f g \rrbracket = \llbracket f \rrbracket \{g\} + \{ f \} \llbracket g \rrbracket$, a few calculations show that for every $w,T \in V_h$,
\begin{equation} \label{dsTT2}
\begin{aligned}
& d_h(1,T,Tw) - d_h(w,T,T) \\
&=-  \sum_{K \in \mathcal{T}_h} \int_K \kappa \nabla T \cdot \nabla w \, {\rm d}x + \sum_{e \in \mathcal{E}_h^0} \int_e \{ \kappa \nabla T\} \cdot \llbracket w \rrbracket \,  {\rm d}a \\
& \quad - \sum_{e \in \mathcal{E}_h^0} \int_e \kappa (\{\nabla w\}+\varepsilon) \cdot \llbracket T \rrbracket \, {\rm d}a   - \sum_{e \in \mathcal{E}_h^0} \frac{\eta}{h_e} \int_e \llbracket T \rrbracket \cdot \llbracket w \rrbracket \,  {\rm d}a,
\end{aligned}
\end{equation}
where
\begin{equation} \label{extra}
\varepsilon = \frac{1}{\{T\}} \left( \{\nabla(Tw)\} - \{\nabla T\} \{ w \} - \{T\} \{\nabla w\} \right).
\end{equation}
Hence, for smooth $T$, we get
\begin{equation} \label{consistent_Neu}
d_h(1,T,Tw) - d_h(w,T,T) - e_h(w,T) 
=
 \int_\Omega \kappa\Delta T w \, {\rm d}x -  \int_{\partial\Omega}\kappa \nabla T \cdot n w \,  {\rm d}a, \quad \forall w \in V_h.
\end{equation}
This shows that the method is consistent: The exact solution (with homogeneous Neumann boundary conditions on $T$) satisfies~(\ref{consistent}) (where here $e_h=0$).  Furthermore, the presence of $\int_{\partial\Omega} \nabla T \cdot n w \,  {\rm d}a$ in~(\ref{consistent_Neu}) shows that the scheme enforces homogeneous Neumann boundary conditions on $T$ in a natural way.

We also have for every $w,T \in V_h$,
\begin{equation}\label{dhwff} 
d_h(w,T,T) =-  \sum_{K \in \mathcal{T}_h} \int_K \frac{w}{T} \kappa |\nabla T|^2 \, {\rm d}x - \sum_{e \in \mathcal{E}_h^0} \frac{\eta}{h_e} \int_e \frac{\{w\}}{\{T\}} |\llbracket T \rrbracket|^2  \,  {\rm d}a,
\end{equation} 
so the following inequality holds:
\begin{equation} \label{positivity_d_N_discrete}
d_h(w,T,T) \le 0, \;\;\forall \;T > 0, \;\;\forall\; w \geq 0, \qquad T, w \in V_h.
\end{equation}
This is the discrete version of the thermodynamic consistency condition \eqref{positivity_d_N}. In particular, $d_h(1_K,T,T) \le 0$, for all $K$, where $1_K$ denotes the indicator function for $K$.

\paragraph{Dirichlet boundary conditions.}
Next we consider Dirichlet boundary conditions.  To distinguish the choice of $d_h$ above from the forthcoming choice, we denote the former by $d_h^N$ and the latter by $d_h^D$, and similarly for $e_h$.  We define
\begin{equation} \label{eh}
e_h^D(w,f) = - \int_{\partial \Omega} \frac{w}{f} \kappa  \nabla f \cdot n T_0 \, {\rm d}a + \sum_{e \in \mathcal{E}_h^{\partial}} \frac{\eta}{h_e} \int_e w (f-T_0) \, {\rm d}a
\end{equation}
and
\begin{equation} \label{dh}
\begin{split}
d_h^D(w,f,g) 
&=- \sum_{K \in \mathcal{T}_h} \int_K \frac{w}{f} \kappa \nabla f \cdot \nabla g \, {\rm d}x + \sum_{e \in \mathcal{E}_h^0} \int_e  \frac{1}{\{f\}} \{ w \kappa \nabla f \} \cdot \llbracket g \rrbracket \,  {\rm d}a \\&\quad - \sum_{e \in \mathcal{E}_h^0} \int_e \frac{1}{\{f\}} \{w  \kappa \nabla g\} \cdot \llbracket f \rrbracket \, {\rm d}a  - \sum_{e \in \mathcal{E}_h^0} \frac{\eta}{h_e} \int_e  \frac{\{w\}}{\{f\}} \llbracket f \rrbracket \cdot \llbracket g \rrbracket \,  {\rm d}a \\&\quad - \sum_{e \in \mathcal{E}_h^{\partial}} \int_e \frac{w}{f} \kappa \nabla g \cdot n (f-T_0) \,  {\rm d}a  +\sum_{e \in \mathcal{E}_h^{\partial}} \int_e \frac{w}{f} \kappa \nabla f \cdot n g \,  {\rm d}a,
\end{split}
\end{equation}
where $T_0 : \partial\Omega \rightarrow \mathbb{R}$ is the prescribed temperature on the boundary.  Like~(\ref{dh2}),~(\ref{dh}) is a standard non-symmetric interior penalty discretization of $-\int_\Omega \kappa \nabla f \cdot \nabla g \, {\rm d}x$ (this time for problems with Dirichlet boundary conditions), generalized to incorporate the weight $w/f$ appearing in $d(w,f,g) = -\int_\Omega \frac{w}{f} \kappa \nabla f \cdot \nabla g \, {\rm d} x +\int_{\partial\Omega} \frac{w}{f} \kappa \nabla f \cdot n g \,  {\rm d}a$, see \eqref{d}.  It is related to~(\ref{dh2}) via the addition of three boundary terms: 
\begin{enumerate}
\item A term $\sum_{e \in \mathcal{E}_h^{\partial}} \int_e \frac{w}{f} \kappa  \nabla f \cdot n g \,  {\rm d}a$ that corresponds to the term $-\int_{\partial\Omega} w j_s(f) \cdot n g \, {\rm d}a$ appearing in~(\ref{d}).
\item A term $-\sum_{e \in \mathcal{E}_h^{\partial}} \int_e \frac{w}{f} \kappa \nabla g \cdot n f \,  {\rm d}a$ whose role is to cancel with $\sum_{e \in \mathcal{E}_h^{\partial}} \int_e \frac{w}{f} \kappa \nabla f \cdot n g \,  {\rm d}a$ when $f=g$.
\item A term $\sum_{e \in \mathcal{E}_h^{\partial}} \int_e \frac{w}{f}\kappa  \nabla g \cdot n T_0 \,  {\rm d}a$ that restores the consistency of $d_h^D$.
\end{enumerate}

Similar calculations as earlier show that for every $w,T \in V_h$, 
\begin{equation} \label{dsTT}
\begin{split}
& d_h^D(1,T,Tw) - d_h^D(w,T,T) - e_h^D(w,T)\\
&= -\sum_{K \in \mathcal{T}_h} \int_K\kappa  \nabla T \cdot \nabla w \, {\rm d}x + \sum_{e \in \mathcal{E}_h^0} \int_e \{ \kappa \nabla T\} \cdot \llbracket w \rrbracket \,  {\rm d}a - \sum_{e \in \mathcal{E}_h^0} \int_e \kappa (\{\nabla w\}+\varepsilon) \cdot \llbracket T \rrbracket \,  {\rm d}a \\
&\quad - \sum_{e \in \mathcal{E}_h^0} \frac{\eta}{h_e} \int_e \llbracket T \rrbracket \cdot \llbracket w \rrbracket \,  {\rm d}a  - \sum_{e \in \mathcal{E}_h^{\partial}} \int_e \kappa \nabla (Tw) \cdot n \left(1-\frac{T_0}{T}\right) \, {\rm d}a \\
&\quad  + \sum_{e \in \mathcal{E}_h^{\partial}} \int_e \kappa \nabla T \cdot n w \, {\rm d}a - \sum_{e \in \mathcal{E}_h^{\partial}} \frac{\eta}{h_e} \int_e w (T-T_0) \, {\rm d}a,
\end{split}
\end{equation}
where $\varepsilon$ is given by~(\ref{extra}).  
Hence, for smooth $T$,
\begin{align*}
&d_h^D(1,T,Tw) - d_h^D(w,T,T) - e_h^D(w,T) \\
&=
 \int_\Omega\kappa \Delta T w \, {\rm d}x +  \int_{\partial\Omega}\kappa \nabla (Tw) \cdot n \left(\frac{T_0}{T}-1 \right) \, {\rm d}a  \\&\quad +  \sum_{e \in \mathcal{E}_h^{\partial}} \frac{\eta}{h_e} \int_e w (T_0-T) \, {\rm d}a, \quad \forall w \in V_h.
\end{align*}
This shows that the method is consistent, and it enforces Dirichlet boundary conditions on $T$ in a nearly standard way.

If $w$ has compact support in the interior of $ \Omega $ and $T>0$, then $d^D_h(w,T,T)$ takes exactly the same form as in the homogeneous Neumann case in \eqref{dhwff}. So the discrete version of the thermodynamic consistency condition \eqref{positivity_d_D} for Dirichlet boundary conditions holds. In particular, $d_h^D(1_K,T,T) \le  0$, for all $K$ away from $\partial\Omega$.

The terms associated with heat conduction processes on the right hand side of the discrete entropy equation \eqref{discrete_entropy} are given by
\begin{equation}\label{rhs_D}
\begin{aligned}
d_h^D(w,T,T) + e_h^D(w,T) =& -\sum_{K \in \mathcal{T}_h} \int_K \frac{w}{T} \kappa  |\nabla T|^2 \, {\rm d}x - \sum_{e \in \mathcal{E}_h^0} \frac{\eta}{h_e} \int_e \frac{\{w\}}{\{T\}} |\llbracket T \rrbracket|^2  \,  {\rm d}a \\
 & +\sum_{e \in \mathcal{E}^ \partial _h} \frac{ \eta }{h_e}\int_e w (T-T_0) {\rm d} a.
\end{aligned}
\end{equation}

\paragraph{Nonhomogeneous Neumann boundary conditions.} In view of the expressions \eqref{d} and \eqref{e} for $d$ and $e$ in the nonhomogeneous Neumann case, we discretize $d$ and $e$ with
\[
d^{NN}_h(w,f,g)= d_h^N(w,f,g) +  \sum_{e \in \mathcal{E}_h^{\partial}} \int_e \frac{w}{f} \kappa \nabla f \cdot n g \, {\rm d}a \quad\text{and}\quad e^{NN}_h(w,f) = \int_{\partial\Omega} w q_0 \, {\rm d}a,
\]
where $d_h^N$ is taken from the homogeneous Neumann case, see \eqref{dh2}.
The consistency of the method can be checked by following the same steps as before, and similarly for the discrete version of the thermodynamic consistency condition \eqref{positivity_d_D}.

Now the expression entering into the right hand side of the discrete entropy equation \eqref{discrete_entropy} reads
\begin{equation}\label{rhs_NN}
\begin{aligned}
d^{NN}_h(w,T,T) + e^{NN}_h(w,T)=& -\sum_{K \in \mathcal{T}_h} \int_K \frac{w}{T} \kappa  |\nabla T|^2 \, {\rm d}x - \sum_{e \in \mathcal{E}_h^0} \frac{\eta}{h_e} \int_e \frac{\{w\}}{\{T\}} |\llbracket T \rrbracket|^2  \,  {\rm d}a \\
 & -\int_{ \partial \Omega } w \big( T j_s(T) \cdot n - q_0\big) {\rm d} a;
\end{aligned}
\end{equation}
compare to \eqref{rhs_D}.

\begin{remark} \label{remark:variationalcrime}
In the discussion above, we established the consistency of~(\ref{discrete_entropy}) in the setting where $T=-\frac{\delta \ell}{\delta s}$. If instead $T=-\pi_h\frac{\delta \ell}{\delta s}$, then we have consistency in a weaker sense: The solution to the continuous problem satisfies the discrete equations approximately rather than exactly.  One can equivalently view this situation as a variational crime; see Chapter 10 in \cite{BrSc2008}.  Indeed, by using $-\pi_h\frac{\delta \ell}{\delta s}$ in place of $-\frac{\delta \ell}{\delta s}$, we are in essence modifying the maps $d$ and $e$.  This is analogous to the effect of quadrature on classical finite element methods for elliptic problems; see Section 8.1.3 in \cite{ErGu2004}.
\end{remark}

\paragraph{Discrete form of the second law.} 
The rate of entropy production in the discrete setting is $D^h _t \varsigma$.  According to~(\ref{D_t_s_h}), it satisfies
\[
\langle  \bar D^h _t \varsigma , T w \rangle = \langle  \bar D^h _t  s, T w \rangle  - d_h  ( 1, T, \pi_h(Tw) ), \quad \forall w \in V_h,
\]
where $T = -\pi_h \frac{\delta \ell}{\delta s}$.  For piecewise constant $w$, we have $\pi_h(Tw) = Tw$, so 
\begin{align*}
\langle  \bar D^h _t \varsigma , T w \rangle 
&= \langle  \bar D^h _t  s, T w \rangle  - d_h  ( 1, T, Tw ) \\
&= c_h(w,u,u) -d_h  ( w,T ,T   )- e_h (w, T ), \quad \forall w \in DG_0(\mathcal{T}_h).
\end{align*}
See \eqref{dhwff}, \eqref{rhs_D}, and \eqref{rhs_NN} for the concrete expressions for each boundary condition.

From the definition of $c$ and the thermodynamic consistency of $d_h^N$, $d_h^D$, $d_h^{NN}$, the numerical solution satisfies 
\begin{equation}\label{space_discrete_2nd_law}
\langle  \bar D^h _t \varsigma , T w \rangle=\langle  \bar D^h _t  s, T w \rangle  - d_h  ( 1, T, Tw )\geq 0,
\end{equation}
for all $w \in DG_0(\mathcal{T}_h)$ with $w\geq 0$ (provided that $w$ has compact support in $ \Omega $ for Dirichlet and nonhomogeneous Neumann boundary conditions), which is the discrete form of the second law \eqref{2nd_law_continuum} as reformulated in \eqref{2nd_law_continuum_b_d}. 
In particular $\int_K T \bar{D}_t^h \varsigma \, {\rm d}x \ge 0$ on each $K$ (provided that $K$ is away from $\partial\Omega$ in the Dirichlet and nonhomogeneous Neumann cases). If $T$ is piecewise constant, then this implies that $\int_K \bar{D}_t^h \varsigma \, {\rm d}x \ge 0$ on those $K$.

\subsection{Temporal discretization}

As shown in Proposition 5.1 of \cite{GaGB2020}, in the absence of irreversible processes $(c=d=0)$ an energy preserving time discretization can be developed for Lagrangians of the standard form \eqref{standard_ell}.

We show that a natural extension of this scheme to the discrete heat conducting viscous fluid equations \eqref{Eulerian_multi_discrete3} allows one to exactly preserve energy balance and mass conservation, while satisfying the second law of thermodynamics on each cell at the fully discrete level.
The scheme reads
\begin{equation}\label{fully_discrete_NSF}
\begin{aligned} 
&\left\langle D_{\Delta t}( \rho  u) , v \right\rangle + a \left( (\rho  u)_{k+ 1/2} , u_{k+1/2}, v\right) - b_h \left( D_2 \epsilon, s _{k+ 1/2}  , v\right) \\&\quad + b_h \left(  \frac{1}{2} \pi_h (u _k \cdot u_{k+1}) -   D _1 \epsilon  -\pi_h\phi , \rho_{k+ 1/2}   , v\right)  = - c(1,u_{k+ 1/2} ,v), \quad \forall \,v \in U_h^{\rm grad} \\[0.1cm]
&\left\langle  D_ { \Delta t}  \rho, \theta  \right\rangle  + b_h ( \theta   , \rho  _{k+ 1/2 }, u_{k+ 1/2}) = 0, \quad \forall \,\theta  \in V_h   \\[0.1cm]
&\left\langle  D_{ \Delta t} s, D_2 \epsilon  w  \right\rangle + b_h \left( D_2 \epsilon w , s_{k+1/2}, u_{k+1/2} \right)  - d_h \left( 1, D_2 \epsilon  , D_2 \epsilon  w \right) \\[0.1cm]
& \hspace{2cm} = c(w, u_{k+1/2}, u_{k+1/2}) - d_h\left( w  , D_2 \epsilon ,D_2 \epsilon   \right)-e_h(w, D_2 \epsilon), \quad \forall \, w \in V_h
\end{aligned}
\end{equation} 
with
\[
D_1 \epsilon = \pi_h \left( \frac{  \delta _1 ( \rho  _k, \rho  _{k+1}, s_k)+  \delta _1 ( \rho  _k, \rho  _{k+1}, s_{k+1})}{2} \right), \qquad \delta _1( \rho  , \rho  ', s)= \frac{ \epsilon ( \rho  ',s)- \epsilon ( \rho  , s)}{ \rho  '- \rho  } 
\]
\[
D_2 \epsilon = \pi_h \left( \frac{  \delta _2 ( s  _k, s  _{k+1},  \rho  _k)+  \delta _2 ( s  _k, s  _{k+1},  \rho  _{k+1})}{2} \right) , \qquad \delta _2( s  , s  ', \rho  )= \frac{ \epsilon ( \rho  ,s')- \epsilon ( \rho  , s)}{ s  '- s  } 
\]
\[
D_ { \Delta t} x= \frac{x_{k+1}- x_k}{ \Delta t} , \qquad 
D_{\Delta t}(\rho u) = \frac{\rho_{k+1} u_{k+1} - \rho_k u_k}{\Delta t}
\]
\[
x_{k+1/2} = \frac{x_k+x_{k+1}}{2}, \qquad (\rho u)_{k+1/2} = \frac{\rho_k u_k + \rho_{k+1} u_{k+1}}{2}.
\]

\begin{proposition} The solution of \eqref{fully_discrete_NSF} satisfies the energy balance and conservation of mass
\begin{align}
\frac{\mathcal{E} _{k+1}- \mathcal{E} _k}{\Delta t}& = - e_h(1, D_2 \epsilon )\label{energy_balance} \\
\int_ \Omega \rho  _{k+1} {\rm d} x&= \int_ \Omega \rho  _k {\rm d} x,\label{mass_conservation}
\end{align}
as well as the second law
\begin{equation}\label{fully_discrete_2nd_law}
\begin{aligned}
\left\langle  D_{ \Delta t} s, D_2 \epsilon  w  \right\rangle + b_h \left( D_2 \epsilon w , s_{k+1/2}, u_{k+1/2} \right) &  \\ - d_h \left( 1, D_2 \epsilon  , D_2 \epsilon  w \right) &\geq 0, \quad \forall w \in DG_0(\mathcal{T}_h), \, w\geq 0,
\end{aligned}
\end{equation}  
for all $k$ (provided that $w$ has compact support in $ \Omega $ for Dirichlet and nonhomogeneous Neumann boundary conditions).

Here $ \mathcal{E}  _k=\int_ \Omega \left[  \frac{1}{2} \rho  _k | u _k | ^2 + \epsilon  ( \rho  _k, s_k) +\rho  _k \phi  \right] {\rm d} x$ is the total energy of the system at time $k\Delta t$.
\end{proposition}
\begin{proof}
Taking $v=u_{k+1/2}$ in the fluid momentum equation, $\theta=-\left( \frac{1}{2}\pi_h(u_k \cdot u_{k+1}) - D_1 \epsilon - \pi_h\phi\right)$ in the mass equation, $w=1$ in the entropy equation, and adding the three equations together, we get
\begin{align*}
-e_h(1,D_2\epsilon)
&= \langle D_{\Delta t} (\rho u), u_{k+1/2} \rangle + \langle D_{\Delta t} \rho, \theta \rangle + \langle D_{\Delta t} s, D_2\epsilon \rangle \\
&= \langle D_{\Delta t} (\rho u), u_{k+1/2} \rangle - \frac{1}{2} \langle D_{\Delta t} \rho, u_k \cdot u_{k+1} \rangle + \langle D_{\Delta t} \rho, D_1\epsilon \rangle + \langle D_{\Delta t} \rho, \phi \rangle \\&\quad + \langle D_{\Delta t} s, D_2\epsilon \rangle.
\end{align*}
The energy balance relation~(\ref{energy_balance}) then follows from the identities
\begin{align*}
 \langle D_{\Delta t} (\rho u), u_{k+1/2} \rangle - \frac{1}{2} \langle D_{\Delta t} \rho, u_k \cdot u_{k+1} \rangle &= \frac{1}{\Delta t} \int_{\Omega} \frac{1}{2}\rho_{k+1} |u_{k+1}|^2 - \frac{1}{2} \rho_k |u_k|^2 \, {\rm d}x, \\
\langle D_{\Delta t} \rho, D_1\epsilon \rangle + \langle D_{\Delta t} s, D_2\epsilon \rangle &= \frac{1}{\Delta t} \int_{\Omega} \epsilon(\rho_{k+1},s_{k+1}) - \epsilon(\rho_k,s_k) \, {\rm d}x, \\
 \langle D_{\Delta t} \rho, \phi \rangle &= \frac{1}{\Delta t}\int_{\Omega} \rho_{k+1}\phi - \rho_k \phi \, {\rm d}x.
\end{align*}

Conservation of mass \eqref{mass_conservation} follows by taking $ \theta =1$ in the mass density equation.

Finally, the discrete second law is a direct consequence of the last equation together with the definition of $c$ and the thermodynamic consistency of $d_h^N$, $d_h^D$, and $d_h^{NN}$.
\end{proof}

\subsection{Enhancements}\label{enhancements}

The scheme~(\ref{fully_discrete_NSF}) admits several enhancements and generalizations.

\paragraph{Rotation.} Rotational effects, which are important in geophysical applications, can be easily incorporated by using a Lagrangian
\[
\ell(u, \rho  , s) =  \int_ \Omega \Big[\frac{1}{2} \rho  | u| ^2 + \rho  R \cdot u- \rho  \phi - \epsilon ( \rho  , s) \Big]{\rm d} x
\]
in place of~(\ref{standard_ell}).  Here, $R$ is half the vector potential of the angular velocity $\omega$ of the fluid domain; that is, $\curl R = 2\omega$.

\paragraph{Variable coefficients.} Temperature-dependent and density-dependent coefficients $\mu ( \rho  , T)$ and $ \zeta ( \rho  , T)$ can be treated without difficulty.  The map $c$ then depends parametrically on $\rho$ and $T$, and the balance of energy, conservation of mass, and second law of thermodynamics remain valid at the discrete level if we evaluate $c(1,u_{k+1/2},v)$ and $c(w,u_{k+1/2},u_{k+1/2})$ at, e.g., $(\rho,T)=(\rho_{k+1/2},D_2\epsilon)$ in~(\ref{fully_discrete_NSF}).  It is also possible to treat a temperature-dependent and density-dependent heat conduction coefficient $\kappa(\rho,T)$. In~(\ref{dh2}) and~(\ref{dh}), one must choose appropriate discretizations of the trace of $\kappa(\rho,f)$ (such as $\{\kappa(\rho,f)\}$ or $\kappa(\{\rho\},\{f\})$) when treating integrals over $e \in \mathcal{E}_h^0$.  The balance of energy and conservation of mass remain valid at the discrete level, and so does the second law of thermodynamics, assuming that $\kappa$ remains positive and its discrete trace along each $e \in \mathcal{E}_h^0$ remains positive.

\paragraph{Upwinding.}
Upwinding can be incorporated into the discrete advection equations for the density and entropy, while still preserving the balance of energy, conservation of mass, and second law of thermodynamics.  Following \cite{GaGB2020b,GaGB2021,GaGB2022}, one introduces a $u$-dependent trilinear form
\[
\widetilde{b}_h(u; f,g,v) = b_h(f,g,v) + \sum_{e \in \mathcal{E}_h^0} \int_e \beta_e(u) \left( \frac{v \cdot n}{u \cdot n} \right) \llbracket f \rrbracket \cdot \llbracket g \rrbracket \, \mathrm{d}a,
\]
where $\{\beta_e(u)\}_{e \in \mathcal{E}_h^0}$ are nonnegative scalars.  One then replaces every appearance of $b_h(\cdot,\cdot,\cdot)$ in~(\ref{fully_discrete_NSF}) by $\widetilde{b}_h(u_{k+1/2};\,\cdot,\cdot,\cdot)$.  We used this upwinding strategy with 
\begin{equation} \label{abs}
\beta_e(u) = \frac{1}{\pi} (u \cdot n) \arctan\left( 10u \cdot n \right) \approx \frac{1}{2}|u \cdot n|
\end{equation}
in the numerical experiments below.

\section{Numerical tests}\label{sec_4}

\subsection{Rayleigh-B\'enard convection} \label{sec:rayleigh}

We use our scheme to simulate Rayleigh-B\'enard convection driven by a temperature difference between the top and bottom boundaries of the fluid's container. 

Rayleigh-B\'enard convection has been mainly studied in the Boussinesq approximation which is valid only for thin layers of fluid and cannot take into account the full thermodynamic description. We consider here Rayleigh-B\'enard convection in the general compressible Navier-Stokes-Fourier equations, which are well adapted to treat deep convection.

Taking a perfect gas with adiabatic exponent $ \gamma >1$, in a gravitational potential $ \phi = g z$, and assuming $ \lambda = - \frac{2}{d} \mu $ (Stokes hypothesis), we consider the following nondimensional form of the equations:
\begin{align}
 \rho  ( \partial _t u+ u \cdot \nabla u) &= -  \nabla p -  \frac{1}{\rm Fr}\rho\hat z + \frac{1}{\rm Re}  \operatorname{div} \Sigma, \quad &\Sigma &=\operatorname{Def} u - \frac{1}{d} (\operatorname{div}u) \delta, \label{adim_u} \\
T ( \bar D_t s + \operatorname{div}j_s)  &= \frac{1}{\rm Re}  \Sigma   \!: \!\nabla u - j _s \cdot \nabla T, \quad&j_s &=-\frac{1}{\rm Re} \frac{1 }{\rm Pr }  \frac{\gamma }{ \gamma -1}  \frac{1}{T} \nabla   T, \label{adim_T} \\
\partial_t \rho + \dv(\rho u) &= 0, \label{adim_rho}
\end{align}
with Fr, Re, and Pr, the Froude, Reynolds, and Prandtl numbers, respectively.  The internal energy density is $\epsilon(\rho,s) = \rho^\gamma e^{(\gamma-1)s/\rho}$.

We assume that the fluid moves in a 2D domain $[0,2] \times [0,1] \ni( x,z)$ that is periodic in the $x$-direction, 
with no slip boundary condition $u=0$ for the velocity at the top and bottom.
We consider an initial temperature profile $T(z)= 1+ Z(1-z)$ and initial velocity
\[
u(x,z) = \left( 0, \, \psi(x,z) \right),
\]
where
\[
\psi(x,z) = \begin{cases}
\exp\left(\frac{1}{(x-1)^2+(z-0.5)^2-0.2}\right) &\mbox{ if } (x-1)^2+(z-0.5)^2 < 0.2, \\
0 &\mbox{ otherwise. }
\end{cases}
\]
Using the hydrostatic equilibrium $ \nabla p=- \frac{1}{\rm Fr} \rho  \hat z$, one gets the initial mass density and pressure profiles $ \rho  (z) = T(z)^m$ and $p(z)= T(z)^{m+1}$ with $m= \frac{1}{{\rm Fr} Z}-1$ the polytropic index.

In the Boussinesq case the critical conditions for the onset of convection are uniquely determined by a certain value of the Rayleigh number (e.g., \cite{Ch1961}). In the compressible case, the situation is more involved since the critical value depends on other parameters, such as the temperature difference $Z$ (\cite{Bo2001,AlRi2017} and references therein). Following \cite{RoCh2004}, we will use the following expression of the Rayleigh number in the compressible case:
\[
{\rm Ra} = {\rm Re} ^2 (m+1) Z ^2 {\rm Pr} ( 1- ( \gamma -1) m)/ \gamma .
\]

\begin{figure}[t]
\centering
\includegraphics[scale=0.245,trim=1.85in 0.7in 1.9in 0.6in,clip=true]{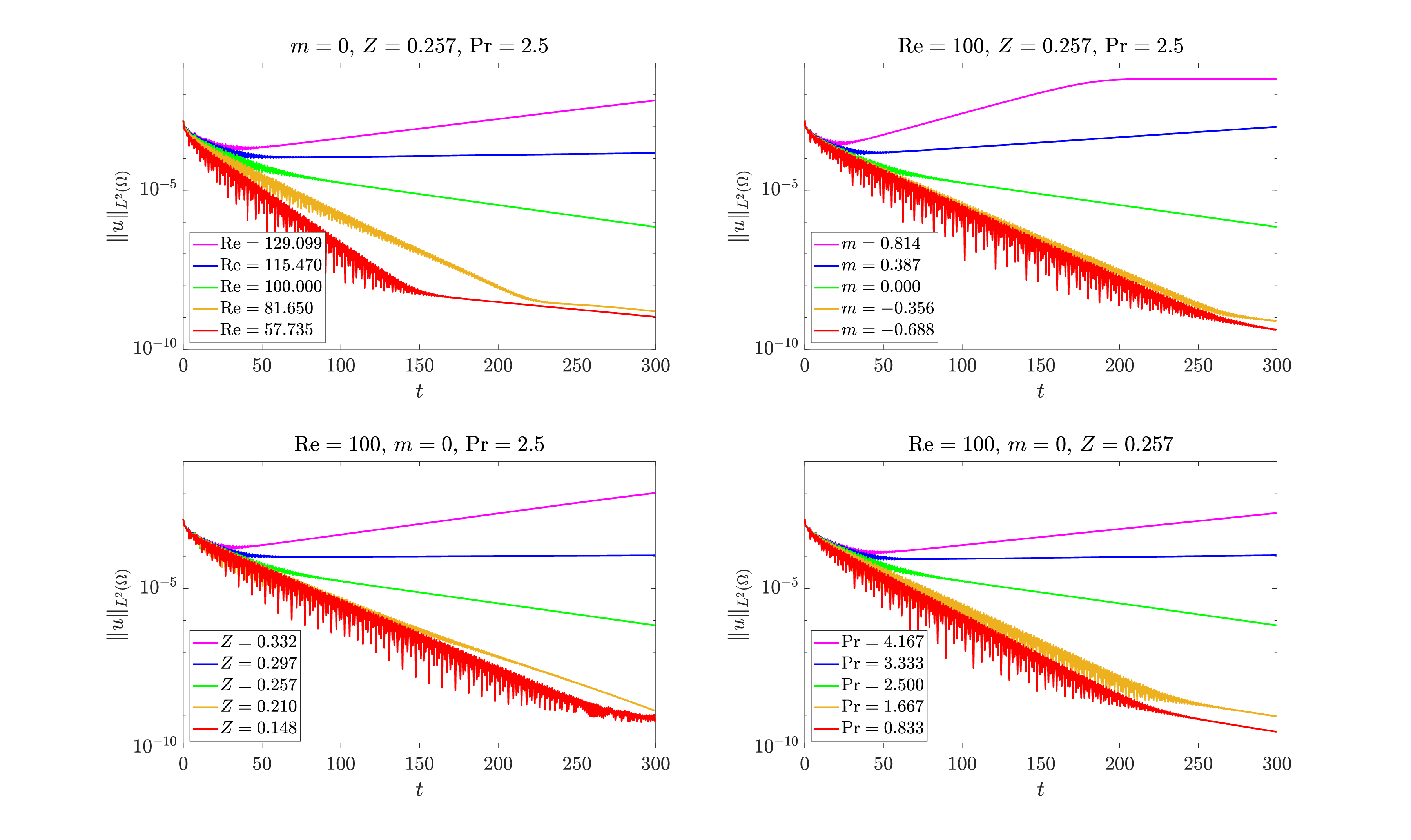}
\caption{The onset of Rayleigh-B\'enard convection with Dirichlet boundary conditions.}
\label{fig:Dirichlet}
\end{figure}

\begin{figure}[t]
\centering
\includegraphics[scale=0.245,trim=1.85in 0.7in 1.9in 0.6in,clip=true]{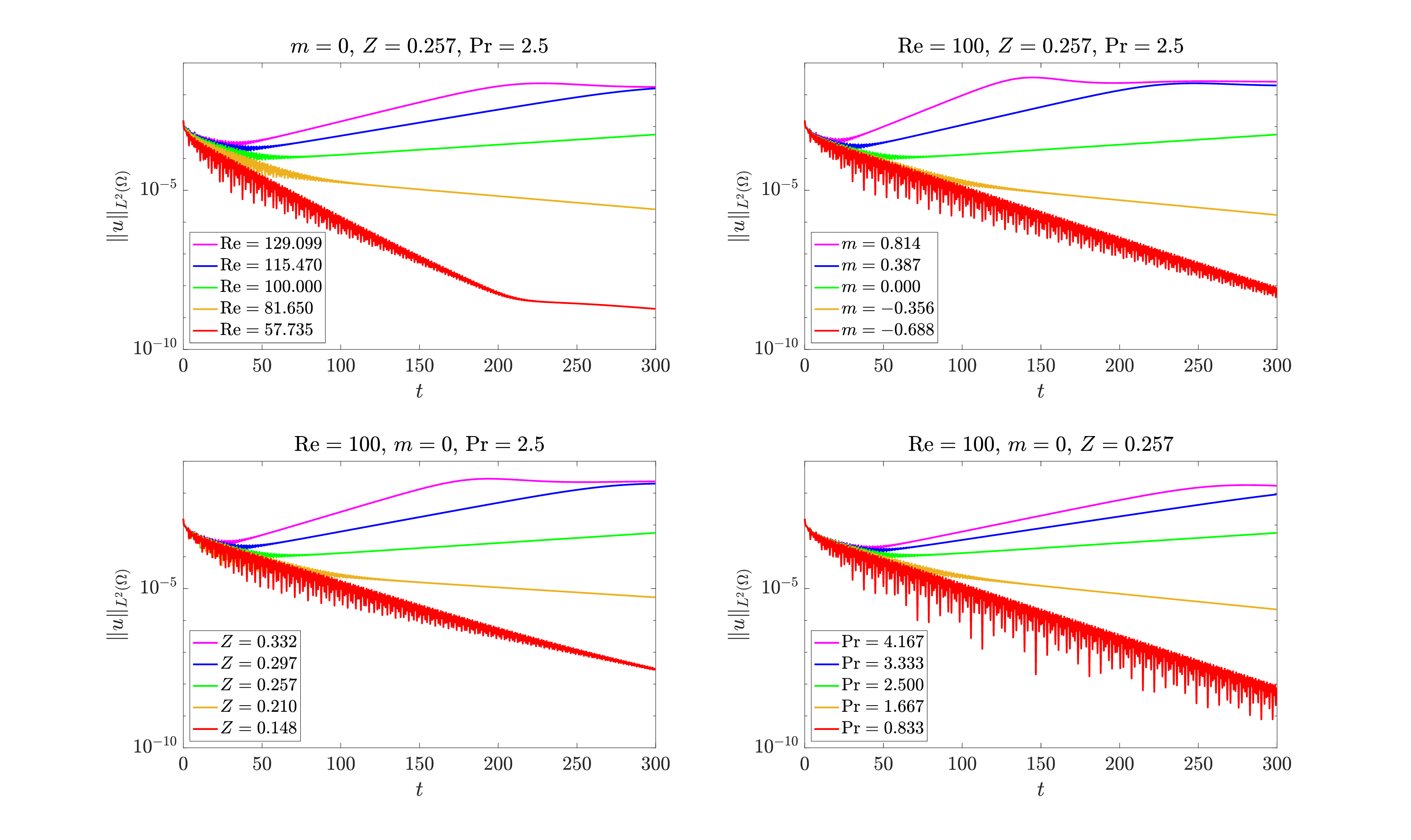}
\caption{The onset of Rayleigh-B\'enard convection with nonhomogeneous Neumann boundary conditions.}
\label{fig:Neumann}
\end{figure}

\paragraph{Dirichlet boundary conditions.}
In the series of experiments below, we fix $\gamma=1.1$ and test the onset of Rayleigh-B\'enard convection for several values of the four parameters Re, $m$, $Z$, Pr, for Dirichlet boundary conditions $T(x,0)=1+Z$ and $T(x,1)=1$, corresponding to the case of ideally conducting upper and lower plates.
Starting with $\mathrm{Re}=100$, $m=0$, $Z=0.256905$, and $\mathrm{Pr}=2.5$, which corresponds to a flow with Rayleigh number 1500, we varied each of the four parameters individually to realize flows with Rayleigh numbers $\mathrm{Ra} \in \{500,1000,1500,2000,2500\}$.  We ran each of these simulations from time $t=0$ to time $t=300$ and plotted the evolution of the $L^2$-norm of the velocity field in Figure~\ref{fig:Dirichlet}.  In our simulations, we used a uniform triangulation of $\Omega$ with $h=\frac{\sqrt{2}}{16}$, a time step $\Delta t = 0.4$, penalty parameter $\eta=0.01\kappa$ (where $\kappa = \frac{1}{\mathrm{Re}}\frac{1}{\mathrm{Pr}}\frac{\gamma}{\gamma-1}$), and finite element spaces $V_h=DG_1(\mathcal{T}_h)$ and $U_h^{\rm grad}=CG_2(\mathcal{T}_h)^2$.  One can see in Figure~\ref{fig:Dirichlet} that in each experiment, the  flow begins to exhibit instability when the Rayleigh number exceeds a critical value lying somewhere between $\mathrm{Ra}=1500$ (the green curves) and $\mathrm{Ra}=2000$ (the blue curves), regardless of which of the four parameters Re, $m$, $Z$, Pr is varied.

\paragraph{Prescribed heat flux boundary conditions.} In order to illustrate the impact of the boundary condition on the critical value of the Rayleigh number, we repeat below the same experiments, now with prescribed heat flux boundary conditions $j_q \cdot n= q_0$ (nonhomogeneous Neumann boundary conditions). This boundary condition describes cases where the upper and lower plates conduct heat poorly compared with the fluid. 
In order to use exactly the same initial temperature profile, we set $q_0(x,0)= - \frac{1}{\rm Re} \frac{1 }{\rm Pr }  \frac{\gamma }{ \gamma -1}  Z$, $q_0(x,1)=  \frac{1}{\rm Re} \frac{1 }{\rm Pr }  \frac{\gamma }{ \gamma -1}  Z$.  This change of boundary condition is known to have a significant impact on the onset of convection by decreasing the critical Rayleigh number, see \cite{HuJaPi1967,ChPr1980} for the Boussinesq case.
Figure~\ref{fig:Neumann} confirms this; one can see in Figure~\ref{fig:Neumann} that the flow begins to exhibit instability when the Rayleigh number exceeds a critical value lying somewhere between $\mathrm{Ra}=1000$ (the yellow curves) and $\mathrm{Ra}=1500$ (the green curves), regardless of which of the four parameters Re, $m$, $Z$, Pr is varied.

We also repeated the above experiment with $\mathrm{Re}=100$, $m=0$, $Z=2$, and $\mathrm{Pr}=2.5$, which corresponds to a much higher Rayleigh number $\mathrm{Ra}= 90909.1$.  We fixed $\Delta t = \frac{1}{80}$ and tested several values of $h$. Plots of the computed temperature at time $t = 10$ are shown in the leftmost column of Figure~\ref{fig:contours}.  Note that since the prescribed heat flux $q_0$ satisfies $\int_{z=0} q_0 \, {\rm d}a + \int_{z=1} q_0 \, {\rm d}a = 0$, this system conserves energy (and mass) exactly in the continuous setting.  The leftmost plot in Figure~\ref{fig:energy} shows that energy and mass are conserved to machine precision in the discrete setting as well.

\subsection{Comparison}

To see how our scheme compares with others in the literature, we implemented and tested the finite volume scheme proposed in \cite{BaLuMiShYu2023,LuShYu2023}.  This scheme uses a formulation of the Navier-Stokes-Fourier equations in terms of $u$, $T$, and $\rho$ which, for a perfect gas in the presence of gravity, has nondimensional form
\begin{align}
\partial_t (\rho u) + \dv(\rho u \otimes u) &= \dv\left( \frac{\Sigma}{\rm Re} - p\delta \right) - \frac{1}{\mathrm{Fr}} \rho \hat{z}, &\Sigma &=\operatorname{Def} u - \frac{1}{d} (\operatorname{div}u) \delta, \nonumber \\
C_V \left(\partial_t (\rho T) + \dv(\rho T u) \right) &= \left( \frac{\Sigma}{\rm Re} - p\delta\right) : \nabla u - \dv(Tj_s), \hspace{-0.05in}&j_s &=-\frac{1}{\rm Re} \frac{1 }{\rm Pr }  \frac{\gamma }{ \gamma -1}  \frac{1}{T} \nabla   T, \nonumber \\
\partial_t \rho + \dv(\rho u) &= 0, && \nonumber 
\end{align}
where $C_V = \frac{1}{\gamma-1}$ and $p=\rho T$.  
We will detail the scheme below, beginning with the case where periodic boundary conditions are imposed on $\partial\Omega$.  The scheme discretizes all variables with piecewise constant functions on a regular quadrilateral mesh $\mathcal{T}_h$ of $\Omega$.   In our notation (which differs considerably from \cite{BaLuMiShYu2023,LuShYu2023}), it seeks $u \in  DG_0(\mathcal{T}_h)^d$ and $T,\rho \in  DG_0(\mathcal{T}_h)$ such that
\begin{align}
\langle \partial_t (\rho u), v \rangle + \beta_h(v,\rho u, u) &= -\alpha_h( v, \Xi) - \frac{1}{\mathrm{Fr}} \langle \rho \hat{z}, v \rangle, &&\forall v \in DG_0(\mathcal{T}_h)^d, \label{FV_u} \\
C_V \left( \langle \partial_t(\rho T), w \rangle + \beta_h(w,\rho T, u) \right) &= \alpha_h( u, w \Xi ) - \gamma_h(T,w), &&\forall w \in DG_0(\mathcal{T}_h), \label{FV_T} \\
\langle \partial_t \rho, \theta \rangle + \beta_h(\theta,\rho,u) &= 0, &&\forall \theta \in DG_0(\mathcal{T}_h), \label{FV_rho}
\end{align}
where
\[
\Xi = \frac{1}{\mathrm{Re}} \left( \frac{1}{2}(X + X^\mathsf{T}) - \frac{1}{d} \operatorname{Tr} X \right) - \rho T \delta
\]
and $X \in DG_0(\mathcal{T}_h)^{d \times d}$ is determined from
\[
\langle X, Y \rangle = \alpha_h \left(u, Y  \right), \quad \forall Y \in DG_0(\mathcal{T}_h)^{d \times d}.
\]
Here, the maps $\alpha_h$, $\beta_h$, and $\gamma_h$ have the following definitions.  The map $\alpha_h$ is given by
\[
\alpha_h(v,\Xi) = \sum_{e \in \mathcal{E} ^0_h}\int_e \{v\} \cdot \llbracket \Xi \rrbracket \, {\rm d} a,
\]
where $\{v\}$ denotes the average of a vector field $v$ and $\llbracket \Xi \rrbracket$ denotes the jump of a tensor field $\Xi$ across $e = K_1 \cap K_2 \in \mathcal{E}_h^0$.  These are defined by
\[
\{v\} = \frac{1}{2}(v_1+v_2), \quad \llbracket \Xi \rrbracket = \Xi_1 n_1 + \Xi_2 n_2,
\]
where $v_i=v|_{K_i}$, $\Xi_i=\Xi|_{K_i}$, and $n_i$ is the unit normal vector to $e$ pointing outward from $K_i$.  The map $\beta_h(\cdot,\cdot,u)$ has two meanings. For scalar fields $f$ and $g$, it is the bilinear form
\begin{align}
\beta_h(f,g,u) 
&= \sum_{e \in \mathcal{E} ^0_h}\int_e (\{u\} \cdot \llbracket f \rrbracket \{g\} + (\beta_e(\{u\}) + h^\xi) \llbracket f \rrbracket \cdot \llbracket g \rrbracket) \, {\rm d} a, \label{bh_FV}
\end{align}
where $\beta_e(\{u\})=\frac{1}{2}|\{u\} \cdot n|$ and $\xi \in (0,1)$.  (We used $\xi=\frac{1}{2}$ together with the smooth approximation~\eqref{abs} of $\beta_e(\{u\})$ in our implementation.)  When $f$ and $g$ are vector fields, $\beta_h(f,g,u)$ is understood as $\sum_{i=1}^d \beta_h(f_i,g_i,u)$.  Lastly,
\[
\gamma_h(T,w) = \sum_{e \in \mathcal{E} ^0_h}\int_e \frac{\kappa}{h}  \llbracket T \rrbracket \cdot \llbracket w \rrbracket \, {\rm d}a.
\]
This completes the specification of the scheme when periodic boundary conditions are imposed.  Notice that $\alpha_h(v,\Xi)$ approximates $-\int_\Omega v \cdot \dv \Xi \, {\rm d}x$, and when $u$ is continuous, $\beta_h(f,g,u)$ is the same as our $\bar{b}_h(u;f,g,u)$ for scalar fields $f$ and $g$, up to the addition of a term proportional to $h^\xi$ that \cite{BaLuMiShYu2023,LuShYu2023} include for stability reasons.

To impose the no-slip boundary condition $u=0$ on $\partial \Omega$, we follow \cite{LuShYu2023} and penalize $u|_{\partial\Omega}$ by adding $\frac{1}{\varepsilon} \int_{\partial\Omega} u \cdot v \, {\rm d}a$ to the right-hand side of~\eqref{FV_u}, where $\varepsilon$ is an $h$-dependent penalty parameter.  (We took $\varepsilon=h^2$ in our implementation, following Theorem 6.2 in \cite{BaLuMiShYu2023}.)  Likewise, to impose $T=T_0$ on $\partial\Omega$, we follow \cite{BaLuMiShYu2023} and add $\frac{1}{\varepsilon} \int_{\partial\Omega} (T-T_0) w \, {\rm d}a$ to the right-hand side of~\eqref{FV_T}.  To impose $Tj_s \cdot n = q_0$ on $\partial\Omega$, we subtract $\int_{\partial\Omega} q_0 w \, {\rm d}a$ from the right-hand side of~\eqref{FV_T}.

We discretized~\eqref{FV_u}-\eqref{FV_rho} in time with
\begin{align*}
\langle D_{\Delta t} (\rho u), v \rangle &+ \beta_h(v,(\rho u)_{k+1/2}, u_{k+1/2}) && \\&\quad = -\alpha_h( v, \Xi_{k+1/2}) - \frac{1}{\mathrm{Fr}} \langle \rho_{k+1/2} \hat{z}, v \rangle, &&\forall v \in DG_0(\mathcal{T}_h)^d, \\
C_V \big( \langle D_{\Delta t} (\rho T), w \rangle &+ \beta_h(w,(\rho T)_{k+1/2}, u_{k+1/2}) \big) && \\&\quad = \alpha_h( u_{k+1/2}, w \Xi_{k+1/2} ) - \gamma_h(T_{k+1/2},w), &&\forall w \in DG_0(\mathcal{T}_h), \\
\langle D_{\Delta t} \rho, \theta \rangle &+ \beta_h(\theta,\rho_{k+1/2},u_{k+1/2}) = 0, &&\forall \theta \in DG_0(\mathcal{T}_h), 
\end{align*}
where $u_{k+1/2} = \frac{u_k+u_{k+1}}{2}$, $\rho_{k+1/2} = \frac{\rho_k+\rho_{k+1}}{2}$, $T_{k+1/2} = \frac{T_k+T_{k+1}}{2}$, $(\rho u)_{k+1/2} = \frac{\rho_k u_k + \rho_{k+1} u_{k+1}}{2}$, $(\rho T)_{k+1/2} = \frac{\rho_k T_k + \rho_{k+1} T_{k+1}}{2}$,
\[
\Xi_{k+1/2} = \frac{1}{\mathrm{Re}} \left( \frac{1}{2}(X_{k+1/2} + X_{k+1/2}^\mathsf{T}) - \frac{1}{d} \operatorname{Tr} X_{k+1/2} \right) - (\rho T)_{k+1/2} \delta,
\]
and
\[
\langle X_{k+1/2}, Y \rangle = \alpha_h \left(u_{k+1/2}, Y  \right), \quad \forall Y \in DG_0(\mathcal{T}_h)^{d \times d}.
\]

Using the above finite volume scheme, we ran the same experiment described in the last paragraph of Section~\ref{sec:rayleigh}, this time with smaller values of $h$ (detailed in the caption of Fig.~\ref{fig:contours}) to ensure that the experiments with the above finite volume scheme used a similar number of degrees of freedom as the experiments we did with our scheme.  (Our scheme used 1056, 4160, 16512, and 65792 degrees of freedom in the four experiments of Fig.~\ref{fig:contours}, while the other used 2048, 8192, 32768, and 131072.) 

The results obtained with the finite volume scheme are shown in the second column of Fig.~\ref{fig:contours}.  At low resolutions, the temperature contours are qualitatively incorrect due to excessive numerical diffusion.  This appears to be due primarily to the term proportional to $h^\xi$ that appears in~\eqref{bh_FV}.  Removing it improved the appearance of the temperature contours significantly; see the rightmost column in Fig.~\ref{fig:contours}.  Nevertheless, the finite volume scheme still artificially dissipates energy, regardless of whether the term proportional to $h^\xi$ is included or not; see the middle and rightmost columns of Fig.~\ref{fig:energy}.

\begin{figure}
\centering
\includegraphics[width=0.32\textwidth]{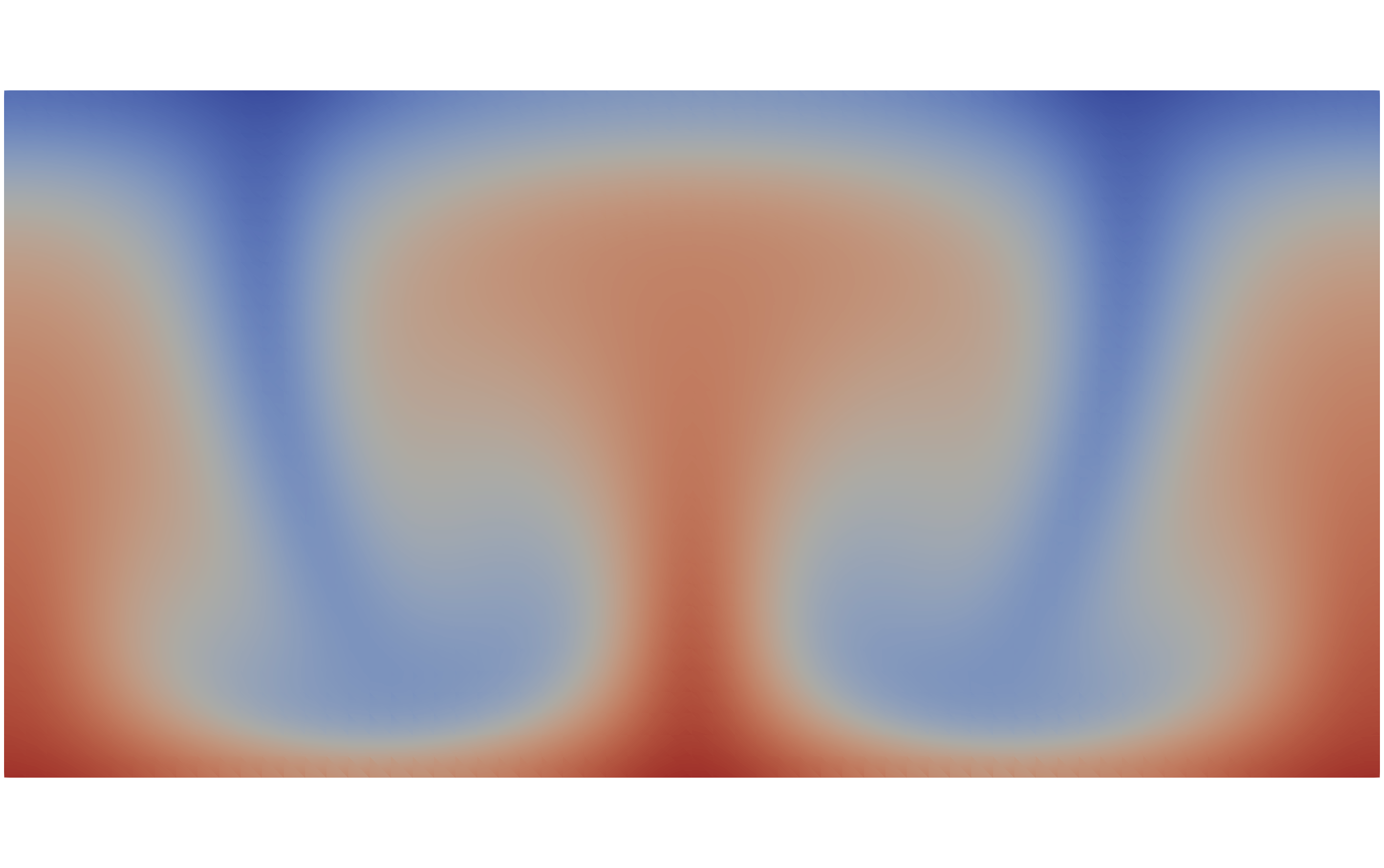}  \includegraphics[width=0.32\textwidth]{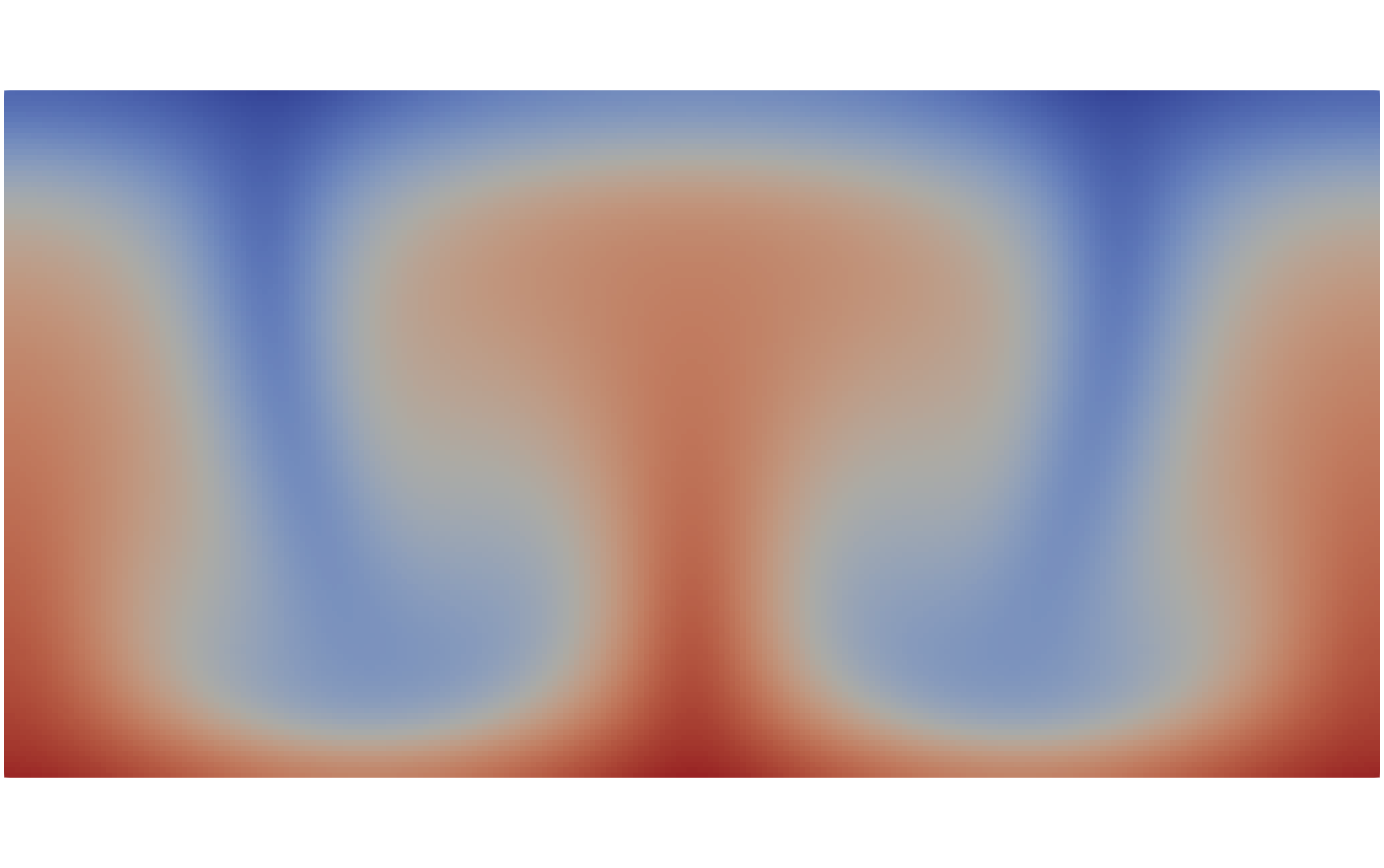} \includegraphics[width=0.32\textwidth]{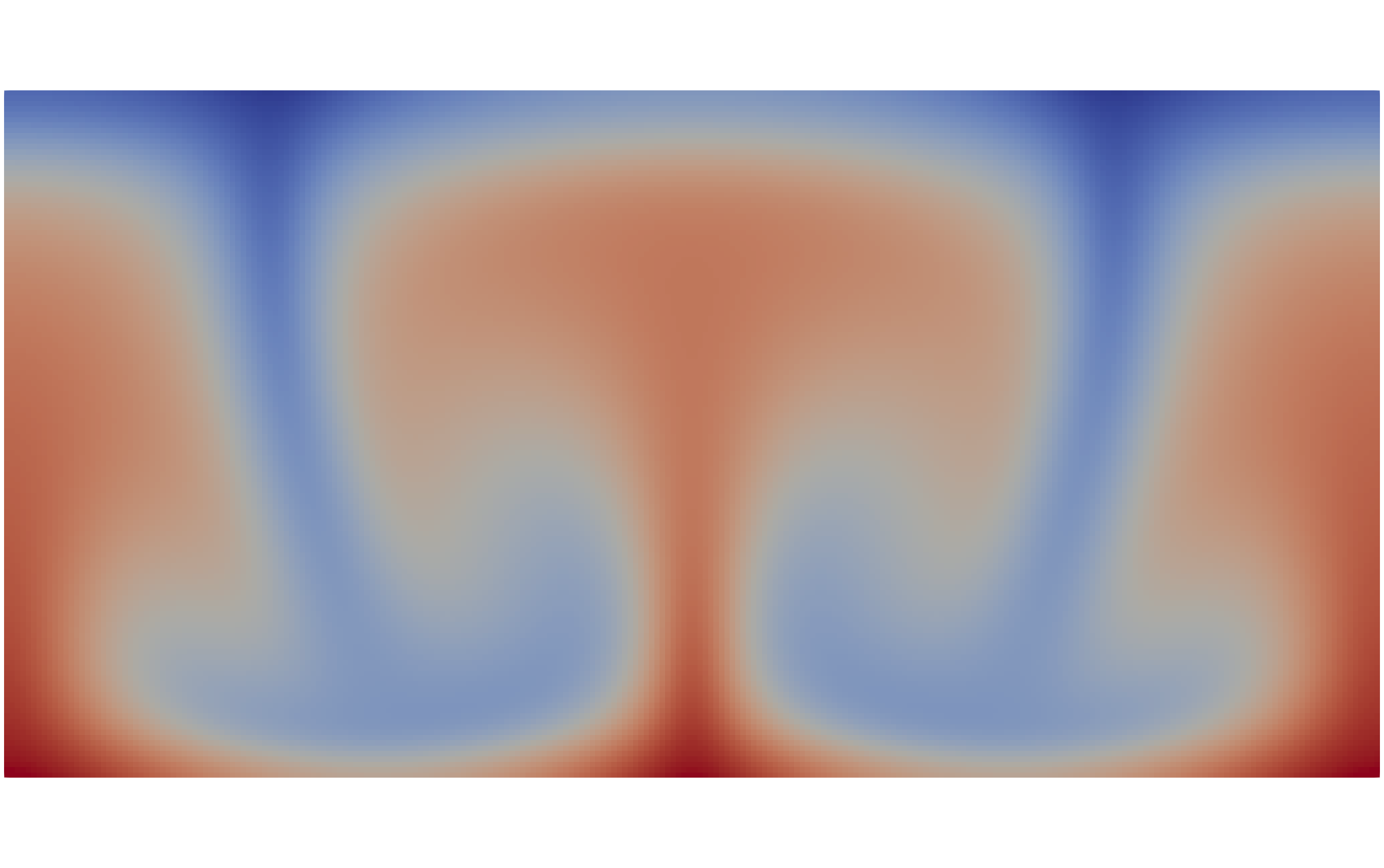} \\ \vspace{-0.2in}
\includegraphics[width=0.32\textwidth]{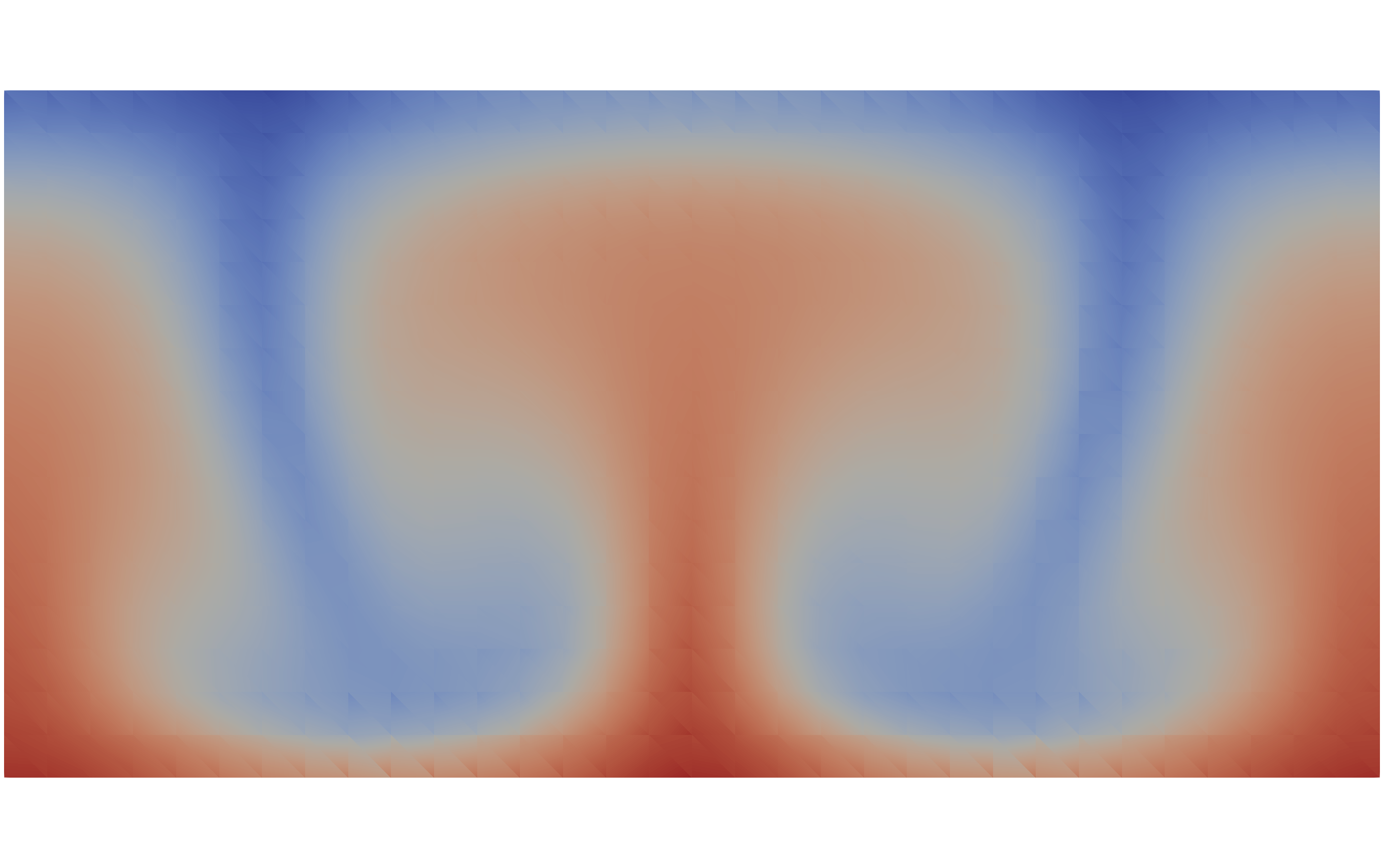}  \includegraphics[width=0.32\textwidth]{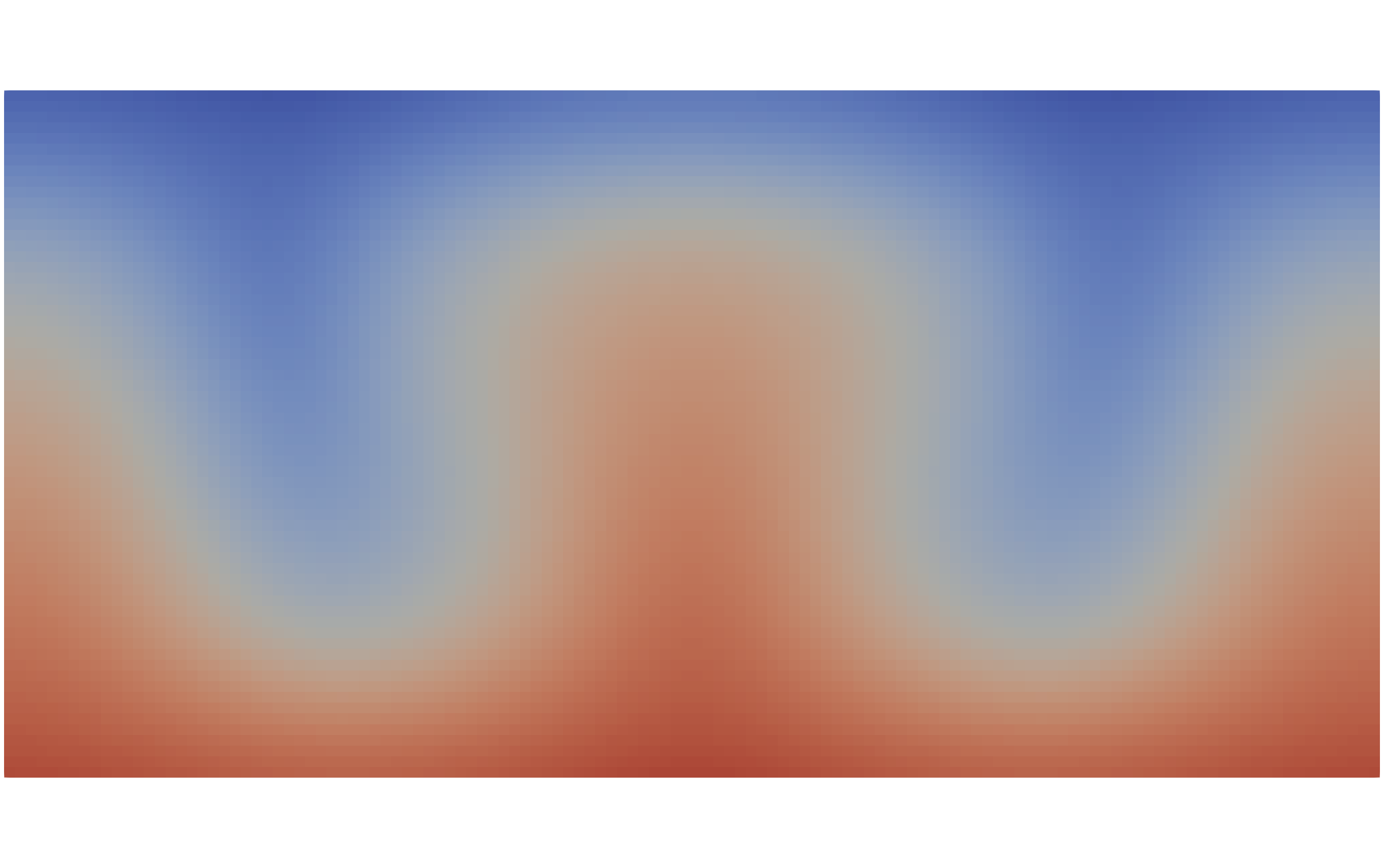} \includegraphics[width=0.32\textwidth]{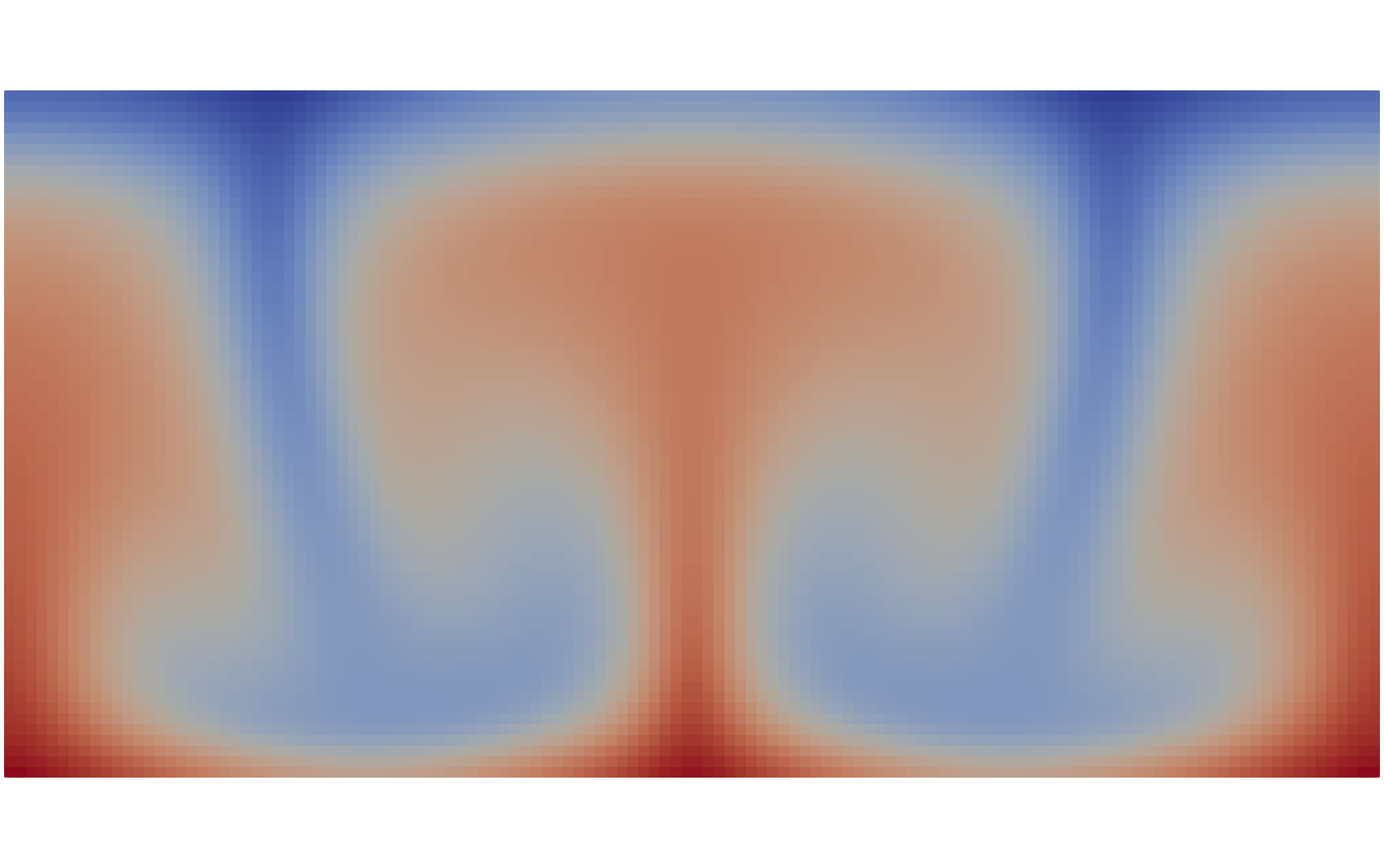} \\ \vspace{-0.2in}
\includegraphics[width=0.32\textwidth]{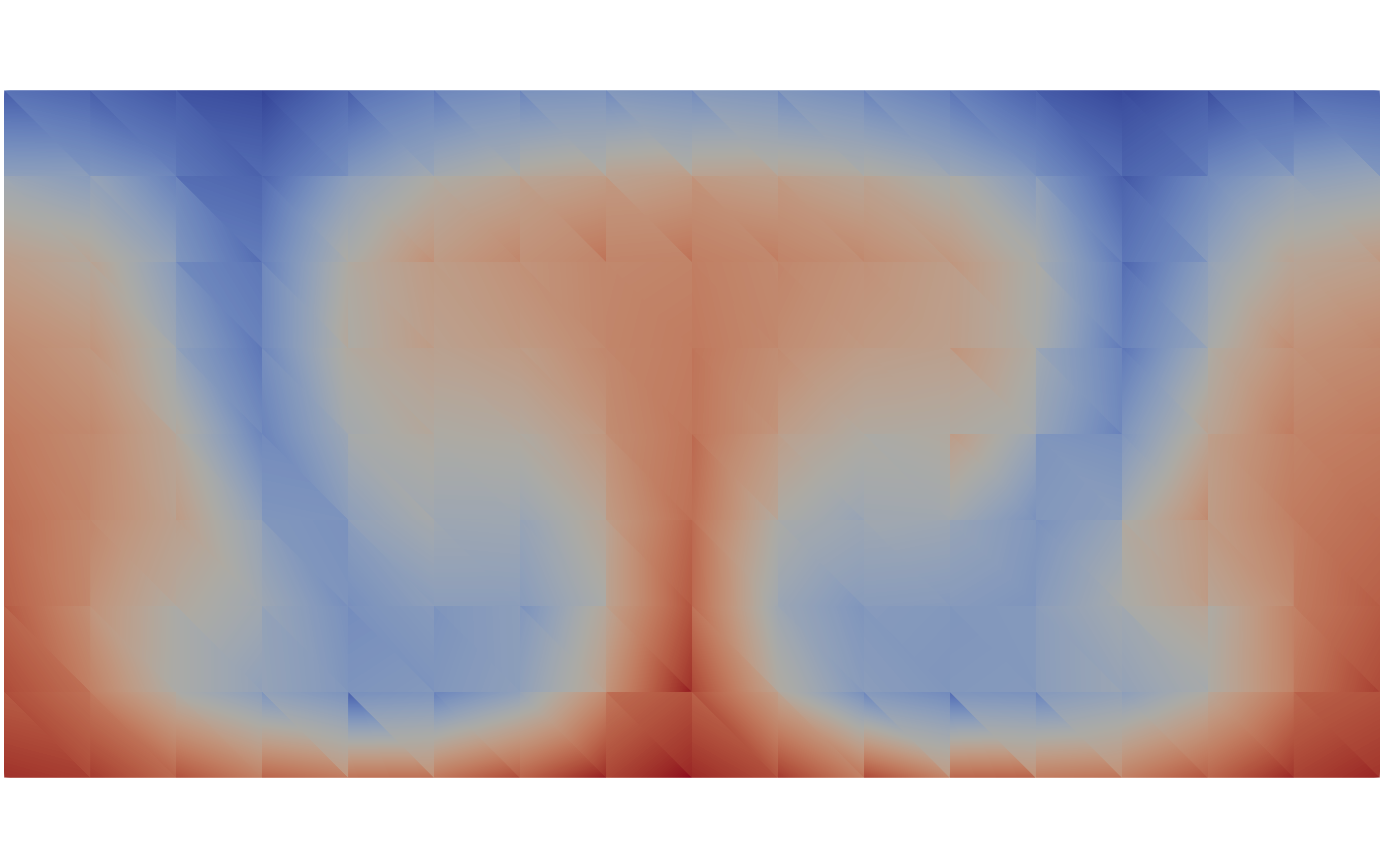}  \includegraphics[width=0.32\textwidth]{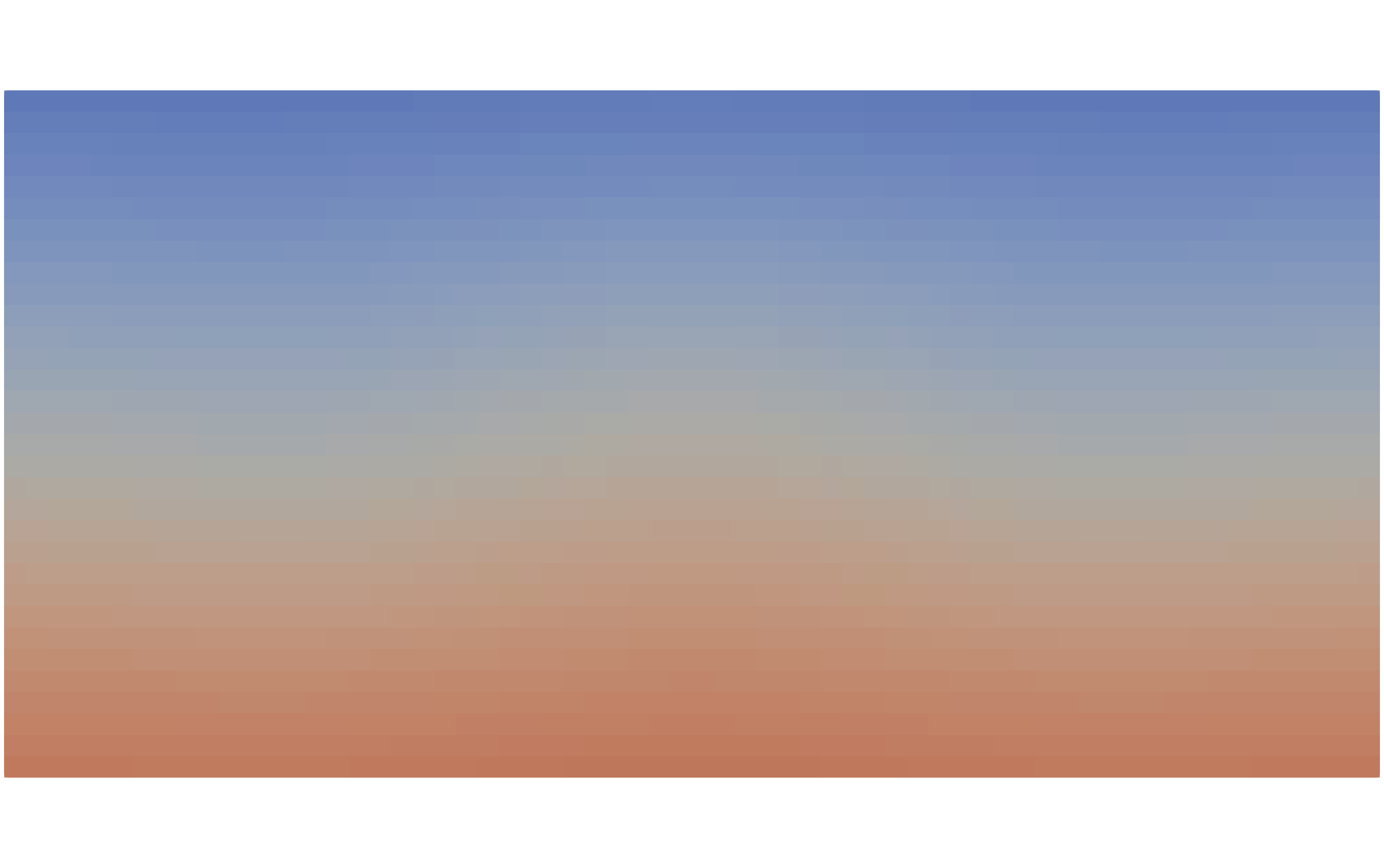} \includegraphics[width=0.32\textwidth]{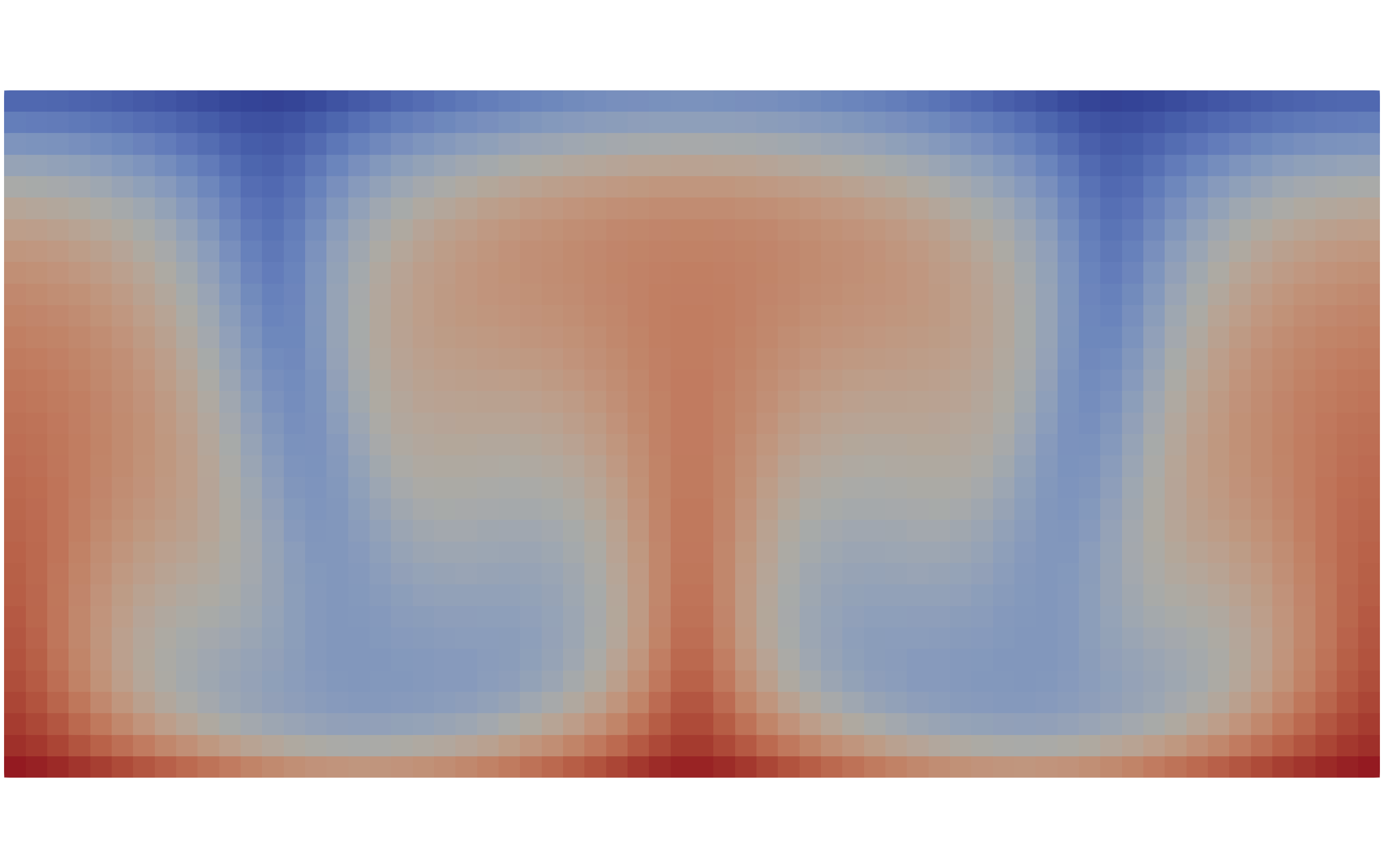} \\ \vspace{-0.2in}
\includegraphics[width=0.32\textwidth]{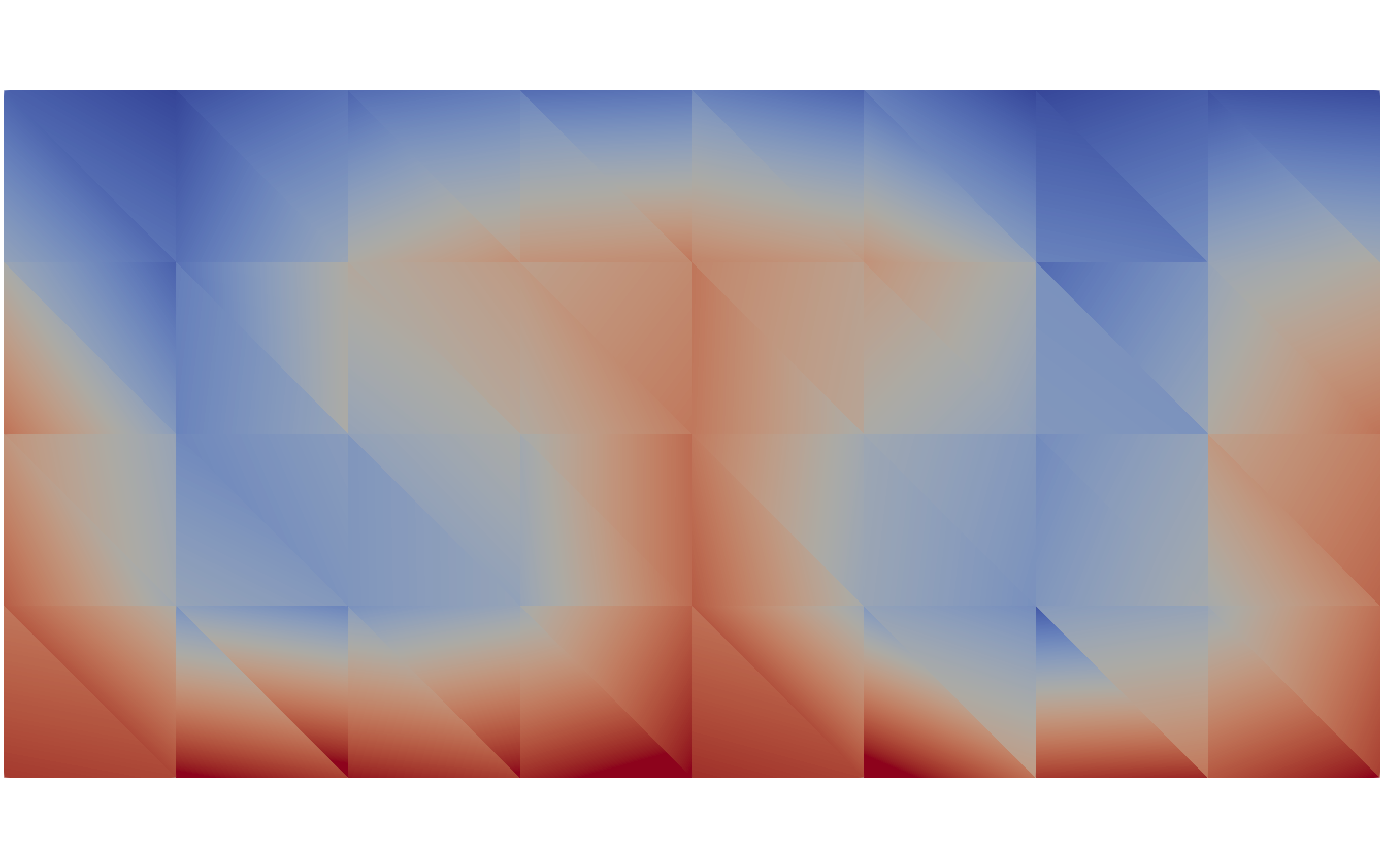}  \includegraphics[width=0.32\textwidth]{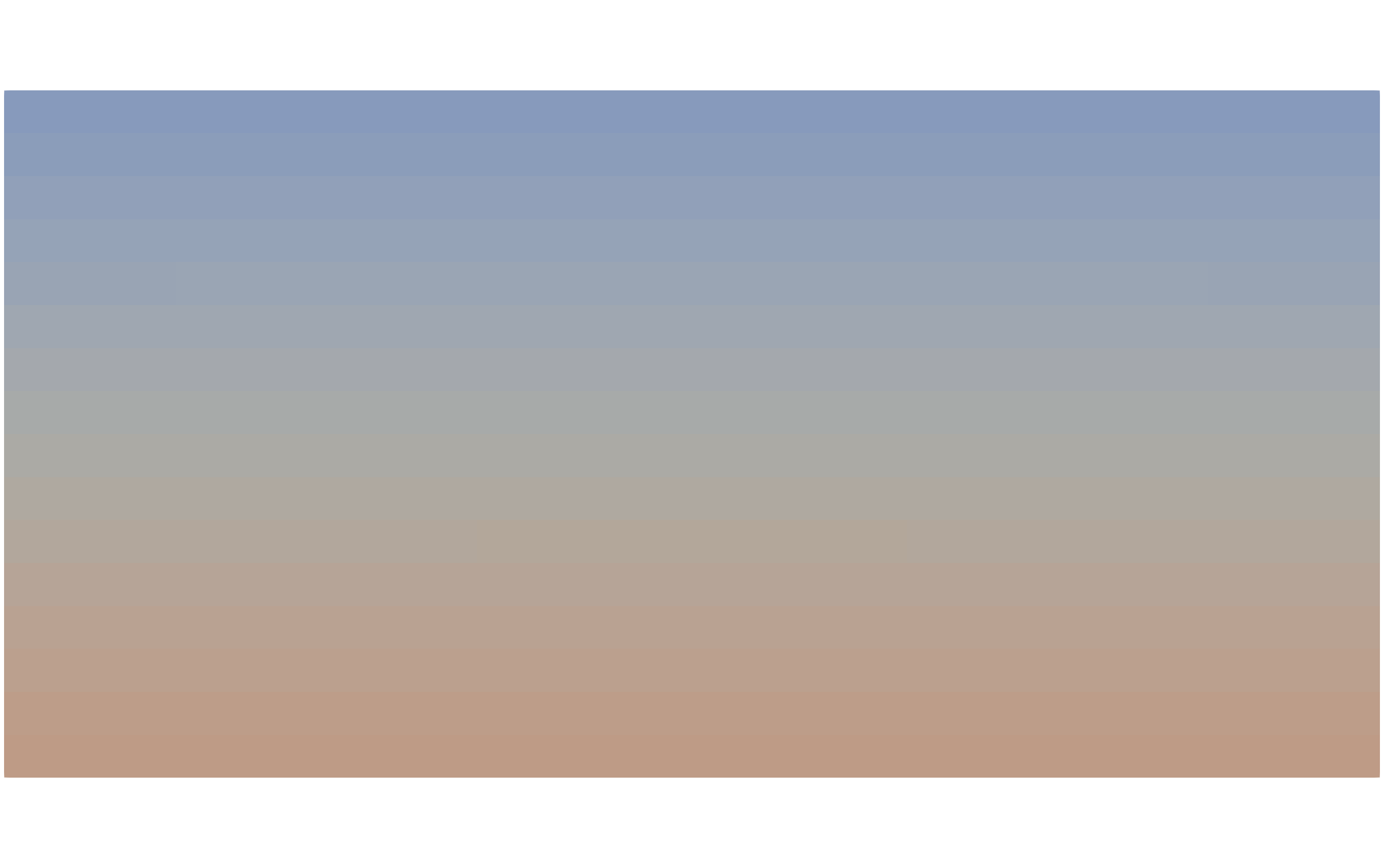} \includegraphics[width=0.32\textwidth]{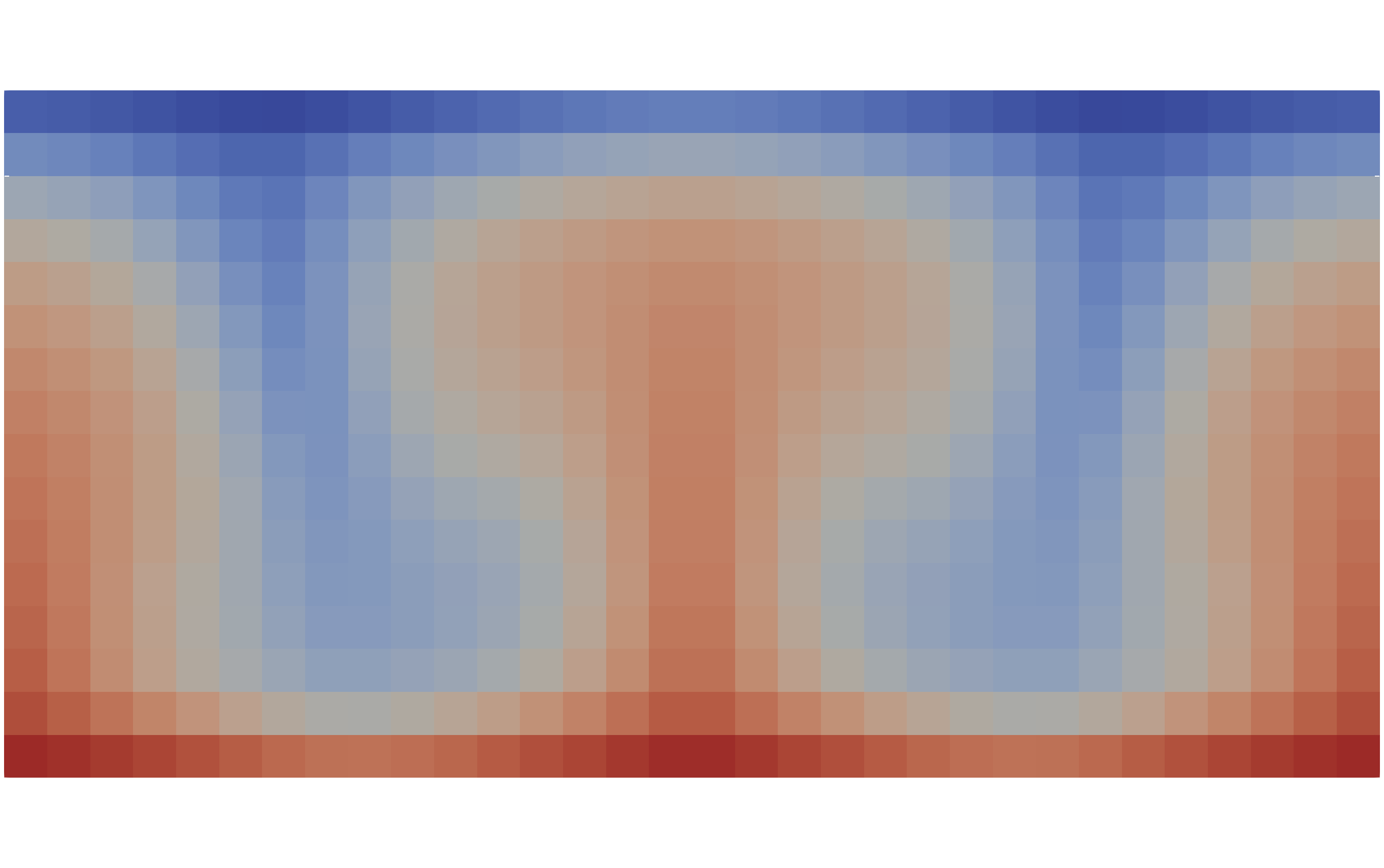}
\caption{Temperature contours at time $t=10$ during a simulation of Rayleigh-B\'enard convection with $\mathrm{Re}=100$, $m=0$, $Z=2$, and $\mathrm{Pr}=2.5$, so that $\mathrm{Ra}= 90909.1$.  Left column: Our scheme with $h \in \{\frac{\sqrt{2}}{4}, \frac{\sqrt{2}}{8}, \frac{\sqrt{2}}{16}, \frac{\sqrt{2}}{32} \}$, ordered from smallest $h$ to largest.  Middle column: The scheme~\eqref{FV_u}-\eqref{FV_rho} with $h \in \{\frac{1}{16}, \frac{1}{32}, \frac{1}{64}, \frac{1}{128} \}$, ordered from smallest $h$ to largest.  Right column: Same as the middle column, but with the term proportional to $h^\xi$ excluded from~\eqref{bh_FV}.} \label{fig:contours}
\end{figure}

\begin{figure}
\centering
\includegraphics[width=0.32\textwidth,trim=3.8in 0.3in 5.7in 0.5in,clip=true]{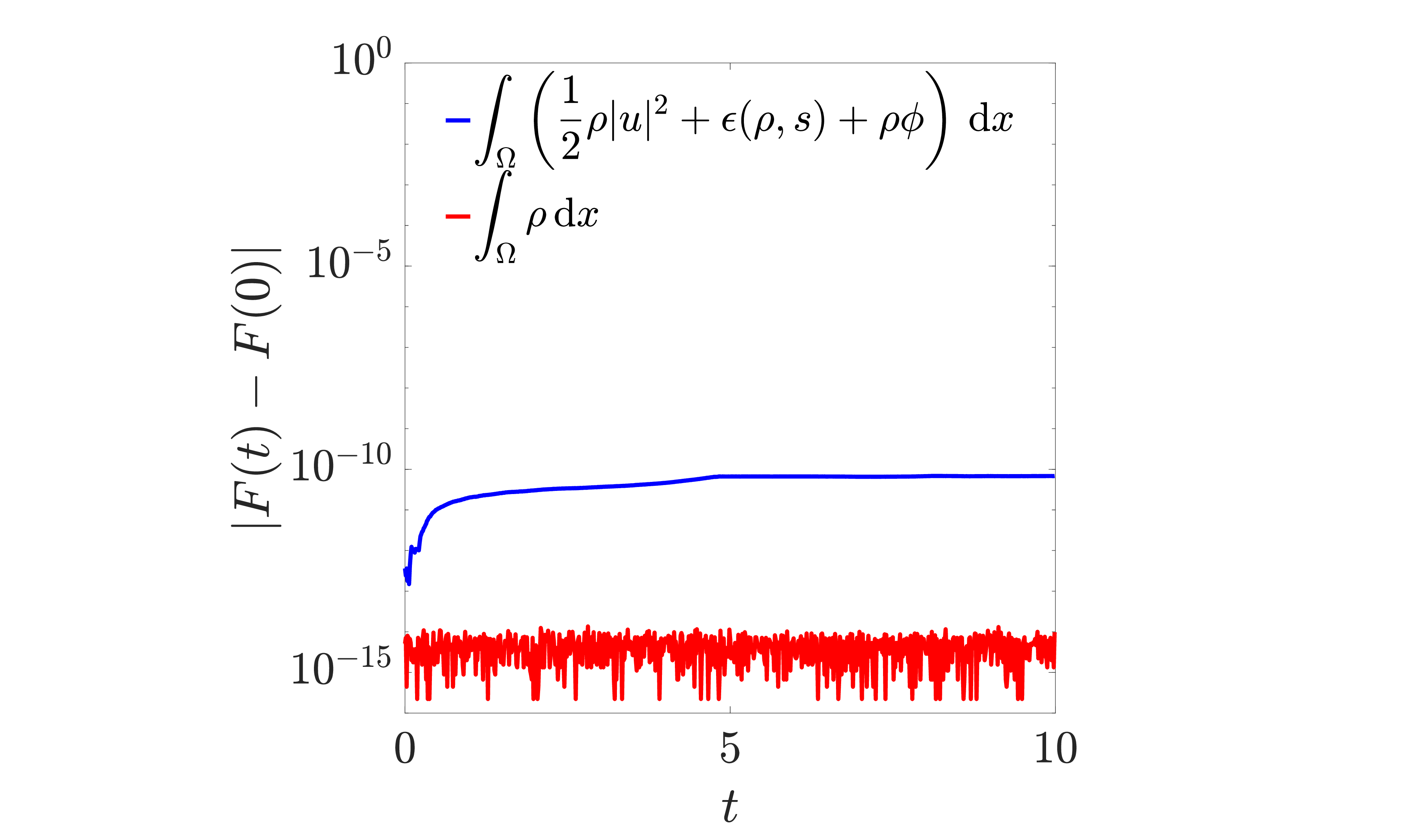} 
\includegraphics[width=0.32\textwidth,trim=3.8in 0.3in 5.7in 0.5in,clip=true]{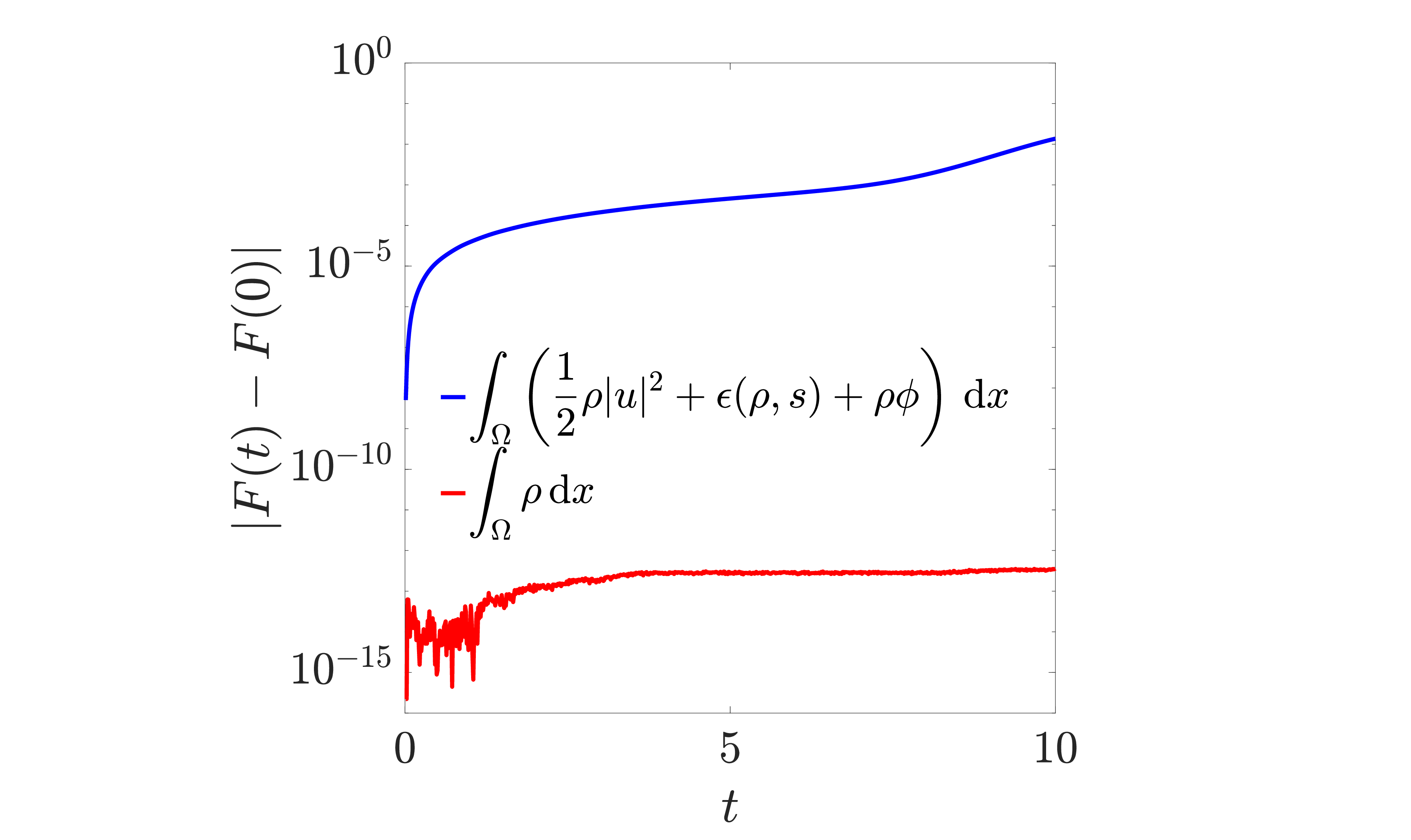} 
\includegraphics[width=0.32\textwidth,trim=3.8in 0.3in 5.7in 0.5in,clip=true]{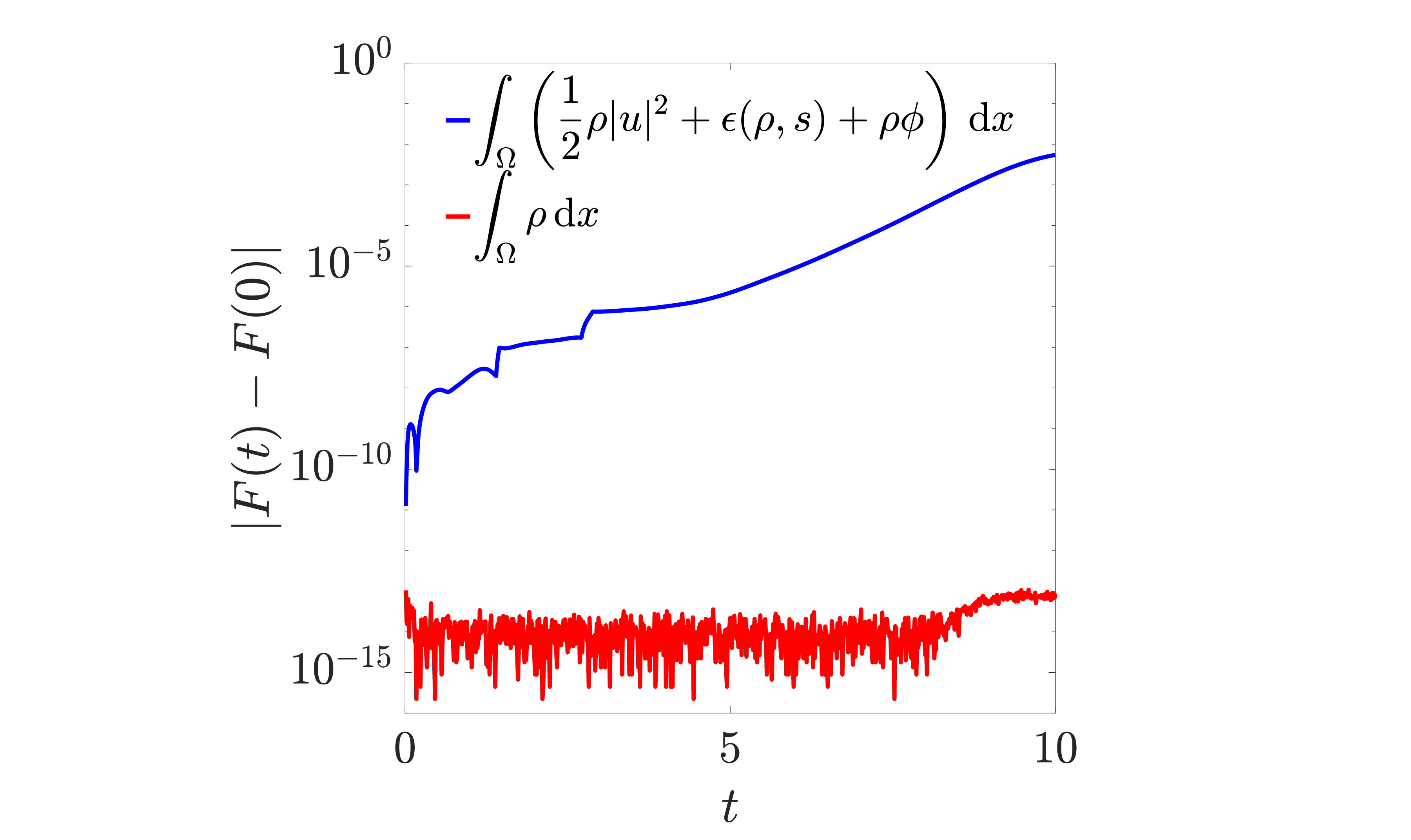}
 \caption{Evolution of mass and energy during the three simulations in the top row of Fig.~\ref{fig:contours}.  The absolute deviations $|F(t)-F(0)|$ are plotted for each conserved quantity $F(t)$.} \label{fig:energy}
\end{figure}

It is worth mentioning that the finite volume scheme also fails to preserve the total entropy $\int_\Omega s \, dx$ in the purely reversible setting ($\kappa=\mu=\lambda=0$).  This is illustrated in the blue and red curves in Fig.~\ref{fig:stotal}, where we ran the same experiment as above using the finite volume scheme with $h=\frac{1}{64}$, this time with $\mathrm{Re}=\infty$ and $\mathrm{Fr}=0.5$.  In contrast, our structure-preserving scheme preserves $\int_\Omega s \, dx$ to machine precision in the purely reversible setting if one uses the finite element space $V_h = DG_0(\mathcal{T}_h)$.  One can see this by taking $w=\frac{1}{D_2\epsilon}$ in~\eqref{fully_discrete_NSF}.  This fact is illustrated in the black curve in Fig.~\ref{fig:stotal}, where we used our structure-preserving scheme with $V_h = DG_0(\mathcal{T}_h)$, $U_h^{\operatorname{grad}} = CG_1(\mathcal{T}_h)^2$, and $h=\frac{\sqrt{2}}{64}$.

\begin{figure}
\centering
\includegraphics[width=0.34\textwidth,trim=2.9in 0.3in 5.5in 0.5in,clip=true]{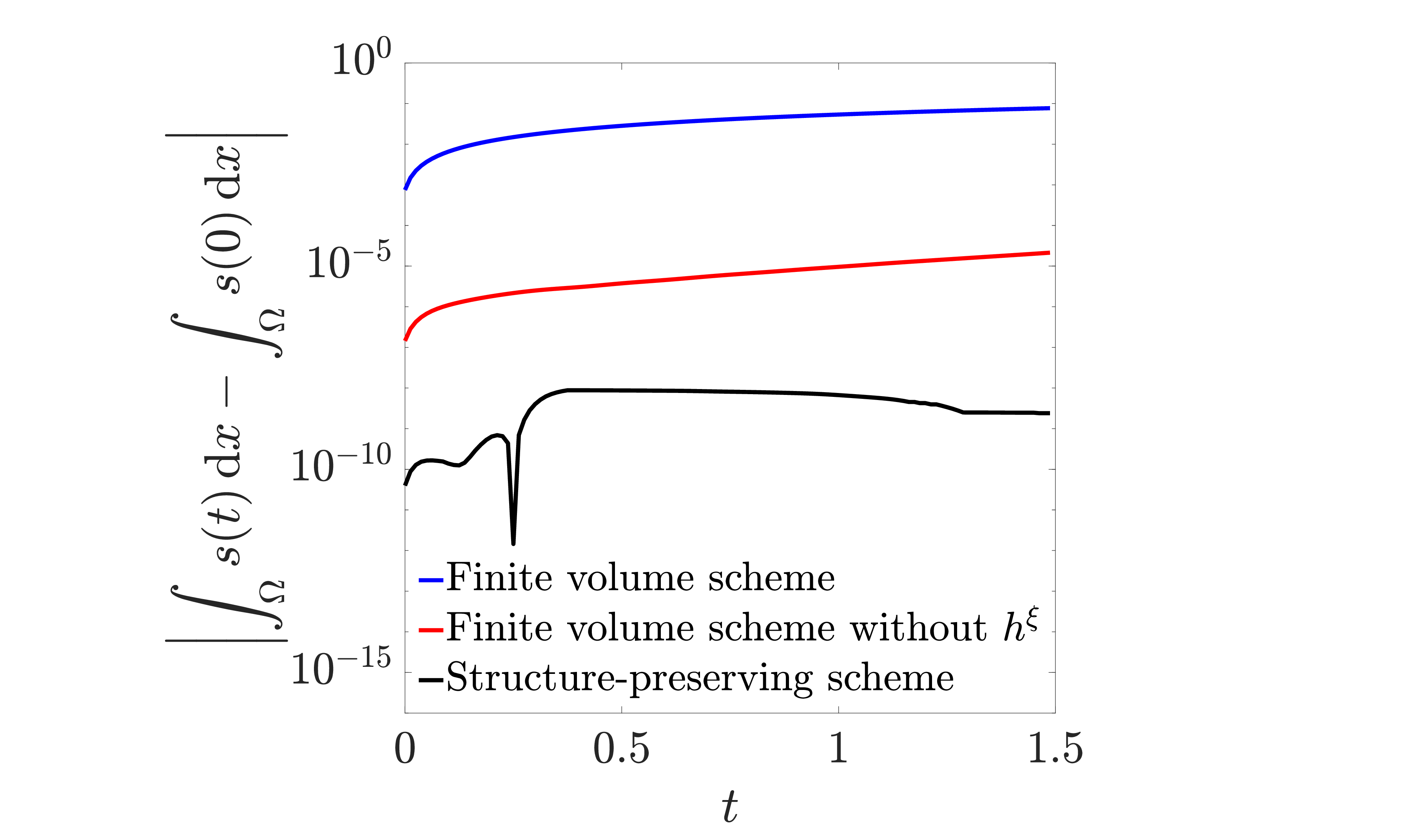}
 \caption{Evolution of total entropy in the purely reversible setting.} \label{fig:stotal}
\end{figure}

\subsection{Numerical convergence}

\begin{table}
\centering
\begin{filecontents*}{thermoconv.dat}
   2.0000000e+00   5.5355276e-02             NaN   1.7807322e+00             NaN   6.7986613e-02             NaN   5.8663815e-02             NaN
   4.0000000e+00   1.4344004e-02   1.9482731e+00   1.0715958e+00   7.3270970e-01   4.1185056e-02   7.2312976e-01   2.9857408e-02   9.7438194e-01
   8.0000000e+00   6.4377884e-03   1.1558107e+00   5.0923893e-01   1.0733462e+00   1.9506506e-02   1.0781656e+00   1.5141862e-02   9.7954633e-01
   1.6000000e+01   4.3017504e-03   5.8164136e-01   2.2495250e-01   1.1787223e+00   8.5450934e-03   1.1907872e+00   7.6735064e-03   9.8058475e-01
   2.0000000e+00   5.1092387e-02             NaN   2.3283829e+00             NaN   9.1828714e-02             NaN   1.5945107e-02             NaN
   4.0000000e+00   1.9410990e-02   1.3962346e+00   1.2168803e+00   9.3614102e-01   4.8321429e-02   9.2628224e-01   6.4266849e-03   1.3109671e+00
   8.0000000e+00   4.6616809e-03   2.0579515e+00   5.6392573e-01   1.1096102e+00   2.2553820e-02   1.0992913e+00   2.4306329e-03   1.4027427e+00
   1.6000000e+01   1.0060975e-03   2.2120802e+00   2.4574769e-01   1.1983273e+00   9.9039604e-03   1.1872944e+00   9.1800681e-04   1.4047553e+00
   2.0000000e+00   1.1965017e-02             NaN   3.7259456e-01             NaN   1.4826731e-02             NaN   5.2207444e-03             NaN
   4.0000000e+00   2.1327315e-03   2.4880481e+00   8.5579781e-02   2.1222647e+00   3.4365319e-03   2.1091754e+00   1.1478190e-03   2.1853604e+00
   8.0000000e+00   2.0666105e-04   3.3673639e+00   2.4715376e-02   1.7918611e+00   1.0014002e-03   1.7789346e+00   2.6431683e-04   2.1185550e+00
   1.6000000e+01   2.1590751e-05   3.2587811e+00   5.7551752e-03   2.1024771e+00   2.3583075e-04   2.0861949e+00   6.4113542e-05   2.0435673e+00
   2.0000000e+00   1.1811870e-02             NaN   5.7968887e-01             NaN   2.3120355e-02             NaN   2.6747667e-03             NaN
   4.0000000e+00   2.1469937e-03   2.4598475e+00   1.5243266e-01   1.9271067e+00   6.1705613e-03   1.9056899e+00   4.4581941e-04   2.5848817e+00
   8.0000000e+00   2.2532610e-04   3.2522316e+00   4.1741165e-02   1.8686293e+00   1.7116587e-03   1.8500067e+00   5.7895207e-05   2.9449436e+00
   1.6000000e+01   2.4045121e-05   3.2281983e+00   1.0675976e-02   1.9671029e+00   4.3984425e-04   1.9603304e+00   8.1492156e-06   2.8287108e+00
   2.0000000e+00   2.1952968e-03             NaN   5.6872022e-02             NaN   2.3229265e-03             NaN   4.3049619e-04             NaN
   4.0000000e+00   1.3710294e-04   4.0010847e+00   7.8265949e-03   2.8612624e+00   3.2188351e-04   2.8513329e+00   2.4037051e-05   4.1626687e+00
   8.0000000e+00   5.7474617e-06   4.5761907e+00   9.3079302e-04   3.0718525e+00   3.8515881e-05   3.0630133e+00   1.9883486e-06   3.5956172e+00
   1.6000000e+01   2.7629483e-07   4.3786445e+00   1.0359525e-04   3.1675025e+00   4.2825525e-06   3.1689106e+00   2.1438040e-07   3.2133258e+00
   2.0000000e+00   2.3232144e-03             NaN   1.2517249e-01             NaN   5.0503799e-03             NaN   4.1800756e-04             NaN
   4.0000000e+00   1.4538358e-04   3.9981861e+00   1.4859303e-02   3.0744792e+00   6.0973519e-04   3.0501372e+00   2.6986759e-05   3.9532054e+00
   8.0000000e+00   6.4930183e-06   4.4848312e+00   1.6318777e-03   3.1867617e+00   6.7326637e-05   3.1789335e+00   1.5511385e-06   4.1208523e+00
   1.6000000e+01   3.2717954e-07   4.3107349e+00   1.7222065e-04   3.2442028e+00   7.1124201e-06   3.2427650e+00   9.4236627e-08   4.0408958e+00
\end{filecontents*}
\pgfplotstabletypeset[
header=false,
font=\small,
clear infinite,
every head row/.style={before row=\hline,after row=\hline},
every last row/.style={after row=\hline},
every row no 4/.style={before row=\hline},
every row no 8/.style={before row=\hline},
every row no 12/.style={before row=\hline},
every row no 16/.style={before row=\hline},
every row no 20/.style={before row=\hline},
create on use/newcol/.style={
        create col/set list={NaN,1,NaN,NaN,NaN,1,NaN,NaN,NaN,2,NaN,NaN,NaN,2,NaN,NaN,NaN,3,NaN,NaN,NaN,3,NaN,NaN,}
},
create on use/newcol2/.style={
        create col/set list={NaN,0,NaN,NaN,NaN,1,NaN,NaN,NaN,1,NaN,NaN,NaN,2,NaN,NaN,NaN,2,NaN,NaN,NaN,3,NaN,NaN,}
},
columns={newcol,newcol2,0,1,2,3,4,5,6},
columns/newcol/.style={column type/.add={|}{},column name={$r$}},
columns/newcol2/.style={column type/.add={|}{},column name={$q$}},
columns/0/.style={sci zerofill,column type/.add={|}{|},column name={$\frac{\sqrt{2}}{h}$}},
columns/1/.style={dec sep align={c|},sci,sci 10e,sci zerofill,precision=2,column type/.add={}{|},column name={$\|u_h-u\|_{L^2(\Omega)}\hspace{-0.05in}$}},
columns/3/.style={dec sep align={c|},sci,sci 10e,sci zerofill,precision=2,column type/.add={}{|},column name={$\|s_h-s\|_{L^2(\Omega)}\hspace{-0.05in}$}}, 
columns/5/.style={dec sep align={c|},sci,sci 10e,sci zerofill,precision=2,column type/.add={}{|},column name={$\|\rho_h-\rho\|_{L^2(\Omega)}\hspace{-0.05in}$}}, 
columns/2/.style={dec sep align={c|},fixed zerofill,precision=2,column type/.add={}{|},column name={Rate}},
columns/4/.style={dec sep align={c|},fixed zerofill,precision=2,column type/.add={}{|},column name={Rate}},
columns/6/.style={dec sep align={c|},fixed zerofill,precision=2,column type/.add={}{|},column name={Rate}},
]
{thermoconv.dat}
\caption{$L^2(\Omega)$-errors in the velocity, entropy, and density at time $t=0.5$ for various values of $h$, $r$, and $q$ (with $\Delta t$ fixed).} \label{tab:conv}
\end{table}

\begin{table}
\centering
\begin{filecontents*}{thermoconv_dt.dat}
   8.0000000e+00   1.8222160e-04             NaN   6.3759422e-03             NaN   2.4458187e-04             NaN   1.1091565e-04             NaN
   1.6000000e+01   4.6901134e-05   1.9579993e+00   1.5663534e-03   2.0252288e+00   6.0137707e-05   2.0239757e+00   2.8043571e-05   1.9837210e+00
   3.2000000e+01   1.1801646e-05   1.9906347e+00   3.8954616e-04   2.0075436e+00   1.4961455e-05   2.0070194e+00   7.0618993e-06   1.9895419e+00
   6.4000000e+01   2.9550077e-06   1.9977544e+00   9.7309734e-05   2.0011383e+00   3.7380311e-06   2.0009000e+00   1.7710717e-06   1.9954336e+00
\end{filecontents*}
\pgfplotstabletypeset[
header=false,
font=\small,
clear infinite,
every head row/.style={before row=\hline,after row=\hline},
every last row/.style={after row=\hline},
every row no 4/.style={before row=\hline},
columns={0,1,2,3,4,5,6},
columns/0/.style={sci zerofill,column type/.add={|}{|},column name={$\Delta t^{-1}$}},
columns/1/.style={dec sep align={c|},sci,sci 10e,sci zerofill,precision=2,column type/.add={}{|},column name={$\|u_h-u\|_{L^2(\Omega)}$}},
columns/3/.style={dec sep align={c|},sci,sci 10e,sci zerofill,precision=2,column type/.add={}{|},column name={$\|s_h-s\|_{L^2(\Omega)}$}}, 
columns/5/.style={dec sep align={c|},sci,sci 10e,sci zerofill,precision=2,column type/.add={}{|},column name={$\|\rho_h-\rho\|_{L^2(\Omega)}$}}, 
columns/2/.style={dec sep align={c|},fixed zerofill,precision=2,column type/.add={}{|},column name={Rate}},
columns/4/.style={dec sep align={c|},fixed zerofill,precision=2,column type/.add={}{|},column name={Rate}},
columns/6/.style={dec sep align={c|},fixed zerofill,precision=2,column type/.add={}{|},column name={Rate}},
]
{thermoconv_dt.dat}
\caption{$L^2(\Omega)$-errors in the velocity, entropy, and density at time $t=0.5$ for various values of $\Delta t$ (with $h$, $r$, and $q$ fixed).} \label{tab:conv_dt}
\end{table}

To test the convergence of our scheme, we considered the same setup as in~\S\ref{sec:rayleigh}, this time with initial conditions
\begin{align*}
u(x,z) &= \frac{1}{10} ( \cos(\pi x)\sin(\pi z), \sin(\pi x)\sin(\pi z) ), \\
T(x,z) &= 1+Z(1-z) + \frac{1}{10} \sin(\pi z), \\
\rho(x,z) &= 1.
\end{align*}
We imposed periodic boundary conditions in the $x$-direction and the no-slip boundary condition $u=0$ along $z=0$ and $z=1$.
We added forcing terms to the right-hand sides of~\eqref{adim_u}-\eqref{adim_rho} to make the exact solution equal to
\begin{align*}
u(x,z,t) &= \frac{1}{10} ( \cos(\pi x)\sin(\pi z) \cos t, \sin(\pi x)\sin(\pi z) \cos t ), \\
T(x,z,t) &= 1+Z(1-z) + \frac{1}{10} \sin(\pi z) \cos t, \\
\rho(x,z,t) &= 1 + \frac{1}{10} \sin(\pi x) \sin(\pi z) \sin t.
\end{align*}
The exact entropy is therefore 
\[
s(x,z,t) =\frac{\rho(x,z,t)}{\gamma-1} \log\left( \frac{T(x,z,t)}{(\gamma-1)\rho(x,z,t)^{\gamma-1}}\right),
\]
as one can verify by solving $T = \frac{\partial\epsilon}{\partial s}$ for $s$.  We imposed Dirichlet boundary conditions for $T$ along $z=0$ and $z=1$ using the exact values of $T$ given above.

We took $\gamma=1.1$, $\mathrm{Re}=100$, $Z=0.42$, $\mathrm{Pr} = 2.5$, and $\mathrm{Fr} = \frac{1}{Z}$, and we used a penalty parameter $\eta = 0.01\kappa$ with $\kappa = \frac{1}{\mathrm{Re}}\frac{1}{\mathrm{Pr}}\frac{\gamma}{\gamma-1}$.  We used a small time step $\Delta t = \frac{1}{320}$ to ensure that temporal discretization errors would be negligible, and we used finite element spaces $V_h = DG_q(\mathcal{T}_h)$ and $U_h^{\rm grad} = CG_r(\mathcal{T}_h)^2$.   

Table~\ref{tab:conv} shows the $L^2(\Omega)$-errors in the computed solution $(u_h,s_h,\rho_h)$ at time $t=0.5$ for $h \in \{\frac{\sqrt{2}}{2},\frac{\sqrt{2}}{4},\frac{\sqrt{2}}{8},\frac{\sqrt{2}}{16}\}$, $r \in \{1,2,3\}$, and $q \in \{r-1,r\}$.  In all cases except $(r,q)=(1,0)$, the velocity, entropy, and density converged at rates given approximately by $O(h^{r+1})$, $O(h^r)$, and $O(h^r)$, respectively, both with $q=r-1$ and with $q=r$.  In the case $(r,q)=(1,0)$, the entropy and density converged linearly but the velocity did not converge quadratically.  Our experiments with $(r,q)=(1,0)$ on even finer meshes (not shown in the table) suggest that the velocity's convergence rate is asymptotically linear when $(r,q)=(1,0)$; the errors $\|u_h-u\|_{L^2(\Omega)}$ were $2.44 \cdot 10^{-3}$ and $1.28 \cdot 10^{-3}$, respectively, for $h = \frac{\sqrt{2}}{32}$ and $h=\frac{\sqrt{2}}{64}$.

We also tested the convergence rate of our temporal discretization by fixing $h=\frac{\sqrt{2}}{16}$, $r=4$, and $q=4$, and running the above experiment with $\Delta t \in \{\frac{1}{8},\frac{1}{16},\frac{1}{32},\frac{1}{64}\}$.  The results, reported in Table~\ref{tab:conv_dt}, indicate that our scheme is second-order accurate in time.

\section{Conclusion}

We have constructed thermodynamically consistent finite element methods for compressible viscous heat conducting fluids by focusing on the Navier-Stokes-Fourier case with the following  types of thermal boundary conditions: insulated boundaries, prescribed heat flux, or prescribed temperature boundary conditions. It was shown that the resulting class of schemes satisfies the two laws of thermodynamics at the fully discrete level. More precisely, regarding the first law, the total energy was shown to be exactly preserved at the fully discrete level when the fluid is adiabatically closed, while a discrete energy balance was proven to hold in the presence of external heating, thereby exactly reproducing the energy balance of the continuous case. Regarding the second law, the entropy generated by the internal irreversible processes was shown to grow at each time step on each cell.

We derived the scheme by discretizing a variational formulation for heat conducting viscous fluids that extends the classical Hamilton principle for fluid dynamics to include irreversible processes.
The treatment of different thermal boundary conditions led to nontrivial changes in the way the thermodynamic fluxes associated with heat conduction are discretized. In each case the variational derivation was crucial to achieve thermodynamic consistency. Natural extensions of our scheme to the case of rotating fluids, variable viscous and thermal coefficients, and upwinding techniques were discussed.

The qualitative properties of the scheme were illustrated by simulating the onset of convection in the celebrated Rayleigh-B\'enard experiment with both prescribed temperature and prescribed heat flux boundary conditions. 
We also tested the convergence rate of our scheme with respect to spatial and temporal refinement and observed second-order accuracy in time and high-order accuracy in space.
A comparison between our scheme and a recently proposed finite volume method for the Navier-Stokes-Fourier equation indicated that our scheme more faithfully respects the first two laws of thermodynamics.
We are not aware of other methods sharing the same properties within the finite element/finite volume literature.

Since the techniques developed in this paper draw upon a general variational formulation that underlies a large class of continuum thermodynamic systems, they are potentially applicable to a large range of models that we plan to study in the future.
Immediate applications include ongoing work on plasma physics and geophysical fluid dynamics, where the accurate long term behavior and the thermodynamic consistency of the discrete models play a crucial role. Future studies should also include error analyses for the schemes derived in the paper.

\appendix

\section{Finite dimensional guiding examples}\label{A}

The treatment of finite dimensional systems is useful for understanding several aspects of the variational formulation for the heat conducting viscous fluid, such as the occurrence of two entropy variables ($S$ and $ \Sigma $ in the Lagrangian description, $s$ and $\varsigma$ in the Eulerian description) and the form of the constraints, especially in the case when the system is not adiabatically closed, which arises when the temperature is prescribed at the boundary. We follow the variational setting developed in \cite{GBYo2017a} and \cite{GBYo2018b} for adiabatically closed and for open systems.

Consider a thermodynamic system described by a mechanical variable $q \in Q$ in a finite dimensional configuration manifold and an entropy variable $S \in \mathbb{R} $. We assume the system has a Lagrangian $L:TQ \times \mathbb{R} \rightarrow \mathbb{R} $ and involves a friction force $F^{\rm fr}:TQ\times \mathbb{R} \rightarrow T^* Q$, with $F^{\rm fr}(q, v, S)\in T^*_qQ$. The total energy and the temperature of the system are defined from $L$ by
\begin{equation}\label{E_T} 
E= \left\langle \frac{\partial L}{\partial \dot q}, \dot q \right\rangle - L \quad\text{and}\quad T= - \frac{\partial L}{\partial S}.
\end{equation}
It is assumed that the Lagrangian is such that $ T>0$ for all $(q, v, S)$.

\paragraph{Adiabatically closed system.} When the system does not interact with its surroundings, the variational formulation is given as follows: find the curves $q(t) \in Q$, $S(t) \in \mathbb{R} $ which are critical for the \textit{variational condition}
\begin{equation}\label{LdA_thermo_simple} 
\delta \int_{t_0}^{t_1}L(q , \dot q , S){\rm d}t =0,
\end{equation}
subject to the \textit{phenomenological constraint}
\begin{equation}\label{CK_simple} 
\frac{\partial L}{\partial S} \dot S  =  \big\langle F^{\rm fr} , \dot q \big\rangle, 
\end{equation}
and for variations subject to the  \textit{variational constraint}
\begin{equation}\label{CV_simple} 
\frac{\partial L}{\partial S} \delta S= \big\langle F^{\rm fr} , \delta  q \big\rangle,
\end{equation}
with $ \delta q(t_0)=\delta q(t_1)=0$.

A direct application of \eqref{LdA_thermo_simple}--\eqref{CV_simple} yields the coupled mechanical-thermal evolution equations:
\begin{equation}\label{simple_systems} 
\frac{d}{dt}\frac{\partial L}{\partial \dot q}- \frac{\partial L}{\partial q}=  F^{\rm fr}, \qquad \frac{\partial L}{\partial S}\dot S= \big\langle F^{\rm fr}, \dot q \big\rangle.
\end{equation}
With these equations, one directly gets the balance of total energy, see \eqref{E_T},  and entropy
\begin{equation}\label{balance_E_T} 
\dot  E=0 \qquad\text{and}\qquad \dot S= - \frac{1}{T} \big\langle F^{\rm fr}, \dot q \big\rangle.
\end{equation}
In particular, the second law requires that the force $F^{\rm fr}$ must be dissipative.

A time discretization of this variational approach has been proposed in \cite{GBYo2018a,CoGB2020}.

\paragraph{System with a heat source.} Assume that the system has an external heat source of temperature $T_H$ with an entropy flow $ \mathcal{J} _{S,H}$. The variational formulation needs the introduction of two additional variables, namely, (i) the thermal displacement $ \Gamma $ with $\dot \Gamma =T$ the temperature and (ii) the entropy variable $ \Sigma $ with $\dot \Sigma =I\geq 0$ the rate of internal entropy production. The two relations $\dot \Gamma =T$ and $ \dot  \Sigma =I$ are not imposed a priori but follow from the variational formulation. For such systems the variational formulation is given as follows: find the curves $q(t) \in Q$, $S(t), \Gamma (t), \Sigma (t) \in \mathbb{R} $ which are critical for the \textit{variational condition}
\begin{equation}\label{LdA_thermo_simple_open}
\delta  \int_{t_0}^{t_1}\!\! \Big[ L(q, \dot q, S)+ \dot \Gamma (S- \Sigma) \Big] {\rm d}t =0,
\end{equation}
subject to the \textit{kinematic constraint}
\begin{equation}\label{CK_simple_open}
\frac{\partial L}{\partial S} \dot \Sigma   =  \big\langle F^{\rm fr} , \dot q \big\rangle + \mathcal{J} _{S,H}(\dot \Gamma - T_H)
\end{equation}
and for variations subject to the  \textit{variational constraint}
\begin{equation}\label{CV_simple_open}
\frac{\partial L}{\partial S} \delta \Sigma  =  \big\langle F^{\rm fr} , \delta  q \big\rangle + \mathcal{J} _{S,H} \delta  \Gamma,
\end{equation}
with $ \delta q|_{t=t_0,t_1}=0$.

Application of \eqref{LdA_thermo_simple_open}--\eqref{CV_simple_open} yields the conditions
\[
\delta q: \;\;\frac{d}{dt}\frac{\partial L}{\partial \dot q}- \frac{\partial L}{\partial q}=  F^{\rm fr}, \qquad \delta S:\;\;\dot \Gamma = - \frac{\partial L}{\partial S} , \qquad \delta \Gamma :\;\;\dot S = \dot  \Sigma + \mathcal{J} _{S,H}.
\]
The second condition imposes that $ \Gamma $ is the thermal displacement while the third condition splits the rate of entropy production $\dot S$ as the sum of the rate of internal entropy production $\dot \Sigma $ and the entropy flow rate $ \mathcal{J} _{S,H}$. Finally, one gets the coupled mechanical-thermal evolution equations:
\begin{equation}\label{simple_systems_open} 
\frac{d}{dt}\frac{\partial L}{\partial \dot q}- \frac{\partial L}{\partial q}=  F^{\rm fr}, \qquad \frac{\partial L}{\partial S} (\dot S - \mathcal{J} _{S,H})   =  \big\langle F^{\rm fr} , \dot q \big\rangle + \mathcal{J} _{S,H} \Big( -\frac{\partial L}{\partial S}  - T_H \Big) .
\end{equation} 
With this system, the balance of energy and entropy \eqref{balance_E_T} is modified as
\begin{equation}\label{balance_E_T_open}
\dot E= T_H \mathcal{J} _{S,H} \qquad\text{and}\qquad \dot S= - \frac{1}{T} \big\langle F^{\rm fr}, \dot q \big\rangle - \frac{1}{T} \mathcal{J} _{S,H}(T-T_H) +  \mathcal{J} _{S,H}.
\end{equation} 
The second law requires that $- \frac{1}{T} \big\langle F^{\rm fr}, \dot q \big\rangle\geq 0$ and $- \frac{1}{T} \mathcal{J} _{S,H}(T-T_H)\geq 0$, while the entropy flow $ \mathcal{J} _{S,H}$ can have an arbitrary sign.

\paragraph{Structure of the variational formulation.} In the adiabatically closed case, one passes from the phenomenological constraint \eqref{CK_simple} to the variational constraint \eqref{CV_simple} by formally replacing the time rate of changes (here $ \dot  S$ and $ \dot  q$) by $ \delta$-variations (here $ \delta S$ and $ \delta q$). This is a general variational setting for adiabatically closed systems, see \cite{GBYo2017a}, which is reminiscent of the Lagrange-d'Alembert principle used in nonholonomic mechanics. The same approach is used to pass from the phenomenological constraint \eqref{KC_fluid} to the variational constraint \eqref{VC_fluid} for fluids in the adiabatically closed case.

In the case with a heat source the kinematic constraint is affine in the rate of changes, hence one passes from \eqref{CK_simple_open} to \eqref{CV_simple_open} by formally replacing the rate of changes (here $ \dot  S$, $ \dot  q$, and $ \dot  \Gamma $) by $ \delta$-variations (here $ \delta S$, $ \delta q$, and $ \delta \Gamma $) and by removing the affine terms associated to the exterior (here $ \mathcal{J} _{S,H}T_H$). This is a general approach for open systems, see \cite{GBYo2018b}. It is used to pass from the constraint \eqref{KC_fluid_W} to the constraint \eqref{VC_fluid_W} for fluids with prescribed boundary temperature

\section{Variational derivation of heat conducting viscous fluid equations}\label{B}

\subsection{Lagrangian description}\label{B_0}

In the Lagrangian description, the computation of the critical condition is much simpler than its Eulerian counterpart and follows closely its finite dimensional analog mentioned above.

For the adiabatically closed case, one gets from \eqref{VP_fluid}--\eqref{VC_fluid} the fluid conducting viscous fluid equation in Lagrangian form as
\begin{equation}\label{equations_mat} 
\left\{
\begin{array}{l}
\vspace{0.2cm}\displaystyle\frac{d}{dt}\frac{ \delta  L }{ \delta  \dot \varphi }   - \frac{\delta  L }{\delta  \varphi }= \operatorname{DIV} P \\
\vspace{0.2cm}\displaystyle
-\frac{\delta  L }{\delta   S}(\dot S + \operatorname{DIV}J_S ) =P: \nabla\dot \varphi  +J _S \cdot \nabla \frac{\delta  L }{\delta   S},
\end{array}
\right.
\end{equation}
together with the conditions
\begin{equation}\label{additional_conditions_mat} 
\dot \Gamma = - \frac{\delta  L}{\delta  S}, \qquad  \dot \Sigma = \dot S + \operatorname{DIV}J_S \qquad\text{and}\qquad J_S \cdot n=0 \;\;\text{on}\;\; \partial \Omega 
\end{equation} 
associated to the variations $\delta S$ and $ \delta \Gamma $.
Since $ -\delta L/ \delta S>0$ is identified with the temperature, the first condition implies that $ \Gamma $ is the thermal displacement. The expression $ \dot S + \operatorname{DIV}J_S$ is the rate of internal entropy production, positive by the second law, hence the second condition implies that $ \Sigma $ is the internal entropy density. The last condition is the insulated boundary condition. It follows from the variation $ \delta \Gamma $ on the boundary $ \partial  \Omega$ and implies that the fluid system is adiabatically closed.

For the case with prescribed boundary temperature, one applies the variational principle \eqref{VP_fluid}-\eqref{KC_fluid_W}-\eqref{VC_fluid_W} and directly gets the equations \eqref{equations_mat} and \eqref{additional_conditions_mat} with the last equation of \eqref{additional_conditions_mat} replaced by the boundary condition
\[
(J_s \cdot n)( \mathfrak{T} - \mathfrak{T} _0)=0 \quad\text{on}\quad  \partial \Omega.
\]
Here $ \mathfrak{T} =- \delta L/ \delta S$ is the temperature in the Lagrangian frame.

\subsection{Neumann boundary conditions}\label{B_1}

We show that the variational formulation \eqref{VP_NSF_spatial}--\eqref{EP_constraints}  yields the equations \eqref{Equations_Eulerian} together with the conditions \eqref{additional_conditions}.

Taking the variations in \eqref{VP_NSF_spatial} and using $ \delta \gamma |_{t=t_0,t_1}=0$ and $ u|_{ \partial \Omega }=0$, we get
\begin{equation}\label{taking_variations} 
\int_{t_0}^{t_1} \int_ \Omega \Big[ \frac{\delta  \ell}{\delta  u} \cdot \delta u + \frac{\delta  \ell}{\delta  \rho  } \delta   \rho  + \frac{\delta  \ell}{\delta  s} \delta s + \delta s D_t \gamma - \delta \varsigma D_t \gamma -  \bar D_t(s- \varsigma ) \delta \gamma + (s- \varsigma ) \nabla \gamma \cdot \delta u \Big] {\rm d} x{\rm d}t=0.
\end{equation} 
Since $ \delta s $ is arbitrary, we get the condition
\begin{equation}\label{1_conditions} 
D_t \gamma = - \frac{\delta  \ell}{\delta  s}
\end{equation} 
thereby recovering the first condition in \eqref{additional_conditions}. Making use of this condition and of the variational constraint \eqref{VC_NSF_spatial} yields 
\begin{equation}\label{intermediate}
\begin{aligned} 
&\int_{t_0}^{t_1}  \int_ \Omega \Big[ \frac{\delta  \ell}{\delta  u} \cdot \delta u + \frac{\delta  \ell}{\delta  \rho  } \delta   \rho  - \sigma : \nabla \vv + j_s \cdot \nabla ( \delta \gamma + \vv \cdot \nabla \gamma ) - \operatorname{div}(  \varsigma \vv ) \frac{\delta  \ell}{\delta  s} \\
& \hspace{3cm} -  \bar D_t(s- \varsigma ) \delta \gamma + (s- \varsigma ) \nabla \gamma \cdot \delta u \Big] {\rm d} x{\rm d}t=0.
\end{aligned}
\end{equation} 
Since $ \delta \gamma $ is arbitrary, we get
\begin{equation}\label{2_conditions} 
\bar D_t \varsigma = \bar D_t s + \operatorname{div} j_s  \qquad\text{and}\qquad j_s \cdot n =0 \;\;\text{on}\;\; \partial \Omega
\end{equation} 
thereby recovering the second and third conditions in \eqref{additional_conditions}. Having found these conditions, we now use $ \vv |_{ \partial \Omega }=0$ and the second condition of \eqref{EP_constraints}, i.e. $\delta \rho  =- \operatorname{div}( \rho  \vv )$, to rewrite \eqref{intermediate} in the form 
\[
 \int_{t_0}^{t_1} \int_ \Omega \Big[ \Big( \frac{\delta  \ell}{\delta  u} + (s- \varsigma ) \nabla \gamma \Big) \cdot \delta u +  \Big( \rho  \nabla \frac{\delta  \ell}{\delta  \rho  } + \varsigma  \nabla \frac{\delta  \ell}{\delta  s  } + \operatorname{div} \sigma  \Big) \cdot \vv  - (\operatorname{div}j_s)  \vv \cdot \nabla \gamma \Big] {\rm d} x{\rm d}t=0.
\]
Using the first condition of \eqref{EP_constraints}, i.e. $ \delta u = \partial _t \vv + [u, \vv]$, integrating by parts, and using the identity
\[
\left( \partial _t + \pounds _u \right)  \big((s- \varsigma ) \nabla \gamma\big) = \bar D_t  (s- \varsigma ) \nabla \gamma + (s- \varsigma ) \nabla D_t \gamma= - \operatorname{div} j_s  \nabla \gamma - (s- \varsigma ) \nabla \frac{\delta  \ell}{\delta  s} ,
\]
with $ \pounds _u$ the Lie derivative of one-form densities, we get the first equation in \eqref{Equations_Eulerian}.
The second equation in \eqref{Equations_Eulerian} follows from the phenomenological constraint \eqref{KC_NSF_spatial} together with  \eqref{1_conditions} and the first condition in \eqref{2_conditions}.

\subsection{Dirichlet boundary conditions}\label{B_2}

We briefly explain how the variational formulation \eqref{VP_NSF_spatial}-\eqref{KC_NSF_spatial_w_Dirichlet}-\eqref{VC_NSF_spatial_w_Dirichlet}-\eqref{EP_constraints} yields the equations \eqref{Equations_Eulerian} together with the conditions in \eqref{additional_conditions}, with the last one replaced by \eqref{Dirichlet_BC}.

We proceed similarly as in \S\ref{B_1} before. Making use of \eqref{VC_NSF_spatial_w_Dirichlet} instead of \eqref{VC_NSF_spatial} we get only the first condition in \eqref{2_conditions}, namely
\begin{equation}\label{2_conditions_Dirichlet} 
\bar D_t \varsigma = \bar D_t s + \operatorname{div} j_s
\end{equation}
since the boundary term involving $j_s \cdot n$ is cancelled by the additional term in the constraint \eqref{VC_NSF_spatial_w_Dirichlet}. The first equation in \eqref{Equations_Eulerian} follows exactly as before. Now the kinematic constraint \eqref{KC_NSF_spatial_w_Dirichlet}, together with \eqref{1_conditions} and the first condition in \eqref{2_conditions} give both the second equation in \eqref{Equations_Eulerian} and the boundary condition  \eqref{Dirichlet_BC}.

\subsection{Weak forms}\label{B_3}

As explained in \S\ref{weak_form_continuous}, the constraints \eqref{KC_NSF_spatial}--\eqref{VC_NSF_spatial}  and \eqref{KC_NSF_spatial_w_Dirichlet}--\eqref{VC_NSF_spatial_w_Dirichlet} can be written in a unified way by using the expressions $d( \cdot , \cdot , \cdot )$ and $e( \cdot , \cdot )$, see \eqref{unified_writing_KC}--\eqref{unified_writing_VC}. The computations reviewed above can be equivalently formulated with these expressions. However, it is useful to explain how the form of the entropy equation (the second equation in \eqref{good_weak_form}) emerges since its form plays a crucial role in our approach.

After taking the variation as earlier in \eqref{taking_variations}, one uses in it the variational constraint with $w=1$, i.e.
\[
\Big\langle 1, \frac{\delta   \ell}{ \delta  s}  \bar D_\delta  \varsigma \Big\rangle =- c(1,u, \vv ) + d\Big( 1, - \frac{\delta \ell}{\delta s}, D_\delta  \gamma \Big).
\]
Proceeding as earlier, this yields
\begin{equation}\label{intermadiate_varsigma} 
\left\langle \bar D_t( \varsigma -s), \delta \gamma \right\rangle = -d\Big( 1, - \frac{\delta \ell}{\delta s}, \delta  \gamma \Big), \;\;\forall\, \delta \gamma 
\end{equation} 
which is the weak form of both \eqref{2_conditions} and \eqref{2_conditions_Dirichlet} depending on the boundary conditions considered.
One finally gets the first equation in \eqref{good_weak_form}, while the  second equation in \eqref{good_weak_form} follows from the kinematic constraint \eqref{unified_writing_KC} in which \eqref{intermadiate_varsigma} is used with $ \delta \gamma = w \frac{\delta \ell}{\delta s}$. This explains the very specific forms of the two terms involving $d( \cdot , \cdot , \cdot )$ in the weak form of the entropy equation, namely
\[
d \Big( 1, -\frac{\delta \ell}{\delta s}, \frac{\delta \ell}{\delta s} w\Big) \quad\text{and}\quad d\Big( w,-\frac{\delta \ell}{\delta s}, \frac{\delta \ell}{\delta s}\Big).
\]

\section*{Acknowledgment}

EG was supported by NSF grant DMS-2012427 and the Simons Foundation award MP-TSM-00002615. FGB was supported by a start-up grant from the Nanyang Technological University.


\end{document}